\documentclass[a4paper,10pt,leqno,english]{amsart}
        \title{On the Farrell--Jones Conjecture for Waldhausen's $A$--theory}

          \author[N. Enkelmann]{Nils-Edvin Enkelmann}
          \address{Rheinische Wilhelms-Universit\"at Bonn\\
               Mathematisches Institut\\
               Endenicher Allee 62, 53115 Bonn, Germany}      
       \email{nils.edvin@gmail.com}
       \urladdr{https://www.hcm.uni-bonn.de/people/profile/nils-edvin-enkelmann/}
      
       \author[W. L\"uck]{Wolfgang L\"uck}
          \address{Rheinische Wilhelms-Universit\"at Bonn\\
               Mathematisches Institut\\
               Endenicher Allee 62, 53115 Bonn, Germany}      
       \email{wolfgang.lueck@him.uni-bonn.de}
      \urladdr{https://www.him.uni-bonn.de/lueck/}

           \author[M. Pieper]{Malte Pieper}
          \address{Rheinische Wilhelms-Universit\"at Bonn\\
               Mathematisches Institut\\
               Endenicher Allee 62, 53115 Bonn, Germany}      
       \email{mpieper@uni-bonn.de}
       \urladdr{https://www.hcm.uni-bonn.de/people/profile/malte-pieper/}

      \author[M. Ullmann]{Mark Ullmann}
          \address{FU Berlin, Insitut f\"ur Mathematik, Arnimallee 7, 14195 Berlin, Germany}      
       \email{mark.ullmann@math.fu-berlin.de}
      \urladdr{http://www.mi.fu-berlin.de/math/groups/top/members/postdoc/ullmann.html}

     \author[C. Winges]{Christoph Winges}
          \address{Hausdorff Research Institute for Mathematics (HIM), Poppelsdorfer Allee 45, 53115 Bonn, Germany} 
       \email{christoph.winges@wwu.de}
      \urladdr{http://wwwmath.uni-muenster.de/reine/u/christoph.winges/}


\usepackage[utf8]{inputenc}
\usepackage[american]{babel}
\usepackage{fixltx2e}
\usepackage{amsmath,amsfonts,amscd,amssymb} 
\usepackage{amsthm}
\usepackage{hyperref}
\usepackage{pdfsync}
\usepackage{xcolor} 
\usepackage{calc}
\usepackage{enumerate}
\usepackage[arrow,curve,matrix,tips,2cell]{xy}
  \SelectTips{eu}{10} \UseTips
  \UseAllTwocells
\usepackage{tikz}
\usetikzlibrary{matrix,arrows,trees}
\usepackage[draft]{fixme}
\usepackage{xspace}

\DeclareMathAlphabet\EuR{U}{eur}{m}{n}
\SetMathAlphabet\EuR{bold}{U}{eur}{b}{n}

\hypersetup{
 colorlinks = true,
 linkcolor = blue,
 citecolor = orange,
 urlcolor = cyan,
}

\newcommand{\calall}{{\mathcal A} {\mathcal L}{\mathcal L}}

\newcommand{\calfj}{{\mathcal F}\!{\mathcal J}}

\newcommand{\calfin}{{\mathcal F}{\mathcal I}{\mathcal N}}

\newcommand{\calvcyc}{{\mathcal V}{\mathcal C}{\mathcal Y}}

\newcommand{\calc}{{\mathcal C}}

\newcommand{\cale}{{\mathcal E}}
\newcommand{\calf}{{\mathcal F}}
\newcommand{\calg}{{\mathcal G}}
\newcommand{\calh}{{\mathcal H}}

\newcommand{\cals}{{\mathcal S}}
\newcommand{\calt}{{\mathcal T}}

\newcommand{\cA}{\mathcal{A}}

\newcommand{\cC}{\mathcal{C}}
\newcommand{\cD}{\mathcal{D}}

\newcommand{\cF}{\mathcal{F}}

\newcommand{\cH}{\mathcal{H}}

\newcommand{\cR}{\mathcal{R}}
\newcommand{\cS}{\mathcal{S}}

\newcommand{\IF}{{\mathbb F}}
  
\newcommand{\IH}{{\mathbb H}}

\newcommand{\IQ}{{\mathbb Q}}

\newcommand{\IZ}{{\mathbb Z}}

\newcommand{\FF}{\mathbb{F}}

\newcommand{\JJ}{\mathbb{J}}
\newcommand{\KK}{\mathbb{K}}

\newcommand{\NN}{\mathbb{N}}

\newcommand{\RR}{\mathbb{R}}
\newcommand{\Ss}{\mathbb{S}}

\newcommand{\ZZ}{\mathbb{Z}}

\newcommand{\bfA}{{\mathbf A}}

\newcommand{\bfE}{{\mathbf E}}

\newcommand{\bfH}{{\mathbf H}}

\newcommand{\bfi}{{\mathbf i}}

\newcommand{\bfK}{{\mathbf K}}  
 
\newcommand{\bfL}{{\mathbf L}}

\newcommand{\bfNA}{{\mathbf{NA}}}

\newcommand{\bfP}{{\mathbf P}} 
\newcommand{\bfp}{{\mathbf p}}

\newcommand{\bfS}{{\mathbf S}}
 
\newcommand{\bfT}{{\mathbf T}}

\newcommand{\bfU}{{\mathbf U}}

\newcommand{\bfWh}{{\mathbf W}{\mathbf h}}

\newcommand{\bfWhs}{{\mathbf W}{\mathbf h}{\mathbf s}}

\newcommand{\curs}{\EuR}

\newcommand{\CWCOMPLEXES}{\curs{CW}\text{-}\curs{COMPLEXES}}
\newcommand{\GCWCOMPLEXES}{G\text{-}\curs{CW}\text{-}\curs{COMPLEXES}}

\newcommand{\GROUPOIDS}{\curs{GROUPOIDS}}

\newcommand{\Or}{\curs{Or}}

\newcommand{\SPACES}{\curs{SPACES}}

\newcommand{\SPECTRA}{\curs{SPECTRA}}

\newcommand{\fC}{\mathfrak{C}}

\newcommand{\fS}{\mathfrak{S}}

\newcommand{\fZ}{\mathfrak{Z}}

\newcommand{\CAT}{\operatorname{CAT}}

\newcommand{\zentrum}{\operatorname{center}}

\newcommand{\colim}{\operatorname{colim}}

\newcommand{\DIFF}{\operatorname{DIFF}}

\newcommand{\G}{\operatorname{G}}

\newcommand{\hocolim}{\operatorname{hocolim}}

\newcommand{\id}{\operatorname{id}}

\newcommand{\ind}{\operatorname{ind}}

\newcommand{\map}{\operatorname{map}}

\newcommand{\Out}{\operatorname{Out}}
\newcommand{\PL}{\operatorname{PL}}

\newcommand{\pr}{\operatorname{pr}}

\newcommand{\supp}{\operatorname{supp}}

\newcommand{\TOP}{\operatorname{TOP}}

\DeclareMathOperator{\diam}{diam}
\DeclareMathOperator{\trans}{trans}
\DeclareMathOperator{\incl}{incl}

\usepackage[nameinlink]{cleveref}

\numberwithin{equation}{section}
\theoremstyle{plain}
\newtheorem{theorem}[equation]{Theorem}
\newtheorem{thm}[equation]{Theorem}

\newtheorem{lemma}[equation]{Lemma}
\newtheorem{lem}[equation]{Lemma}

\newtheorem{corollary}[equation]{Corollary}
\newtheorem{cor}[equation]{Corollary}

\newtheorem{prop}[equation]{Proposition}
\newtheorem{conjecture}[equation]{Conjecture}

\newtheorem*{theorem*}{Theorem}
\newtheorem*{mtheorem*}{Main Theorem}

\theoremstyle{definition}
\newtheorem{definition}[equation]{Definition}
\newtheorem{dfn}[equation]{Definition}

\newtheorem{example}[equation]{Example}

\newtheorem{remark}[equation]{Remark}
\newtheorem{rem}[equation]{Remark}
\newtheorem{notation}[equation]{Notation}

\theoremstyle{remark}
\newtheorem*{summary*}{Summary}

\newcommand{\pt}{\{\bullet\}}
\renewcommand{\epsilon}{\varepsilon}
\renewcommand{\theta}{\vartheta}
\renewcommand{\phi}{\varphi}
\renewcommand{\le}{\leq}
\renewcommand{\subset}{\subseteq}

\newcommand{\abs}[1]{\lvert #1 \rvert}
\newcommand{\norm}[1]{\lVert #1 \rVert}
\newcommand{\cells}{\diamond\,}
\newcommand{\cellsPlus}{\diamond_+\,}
\newcommand{\gen}[1]{\langle #1 \rangle}
\newcommand{\sing}[2]{\cS^{#1}_{#2}}
\newcommand{\skel}[2]{\mathrm{sk}_{#1}(#2)}

\newcommand{\EGF}[2]{E_{#2}(#1)}                   
 
\newcommand{\higherlim}[3]{{\setbox1=\hbox{\rm lim}
        \setbox2=\hbox to \wd1{\leftarrowfill} \ht2=0pt \dp2=-1pt
        \mathop{\vtop{\baselineskip=5pt\box1\box2}}
        _{#1}}^{#2}#3}

\begin{document}

\begin{abstract}
  We prove the Farrell--Jones Conjecture for (non-connective) $A$--theory with coefficients
  and finite wreath products for hyperbolic groups, $\mathrm{CAT}(0)$-groups, cocompact lattices in
  almost connected Lie groups and fundamental groups of manifolds of dimension less or
  equal to three.  Moreover, we prove inheritance properties such as passing to subgroups,
  colimits of direct systems of groups, finite direct products and finite free
  products. These results hold also for Whitehead spectra and spectra of stable pseudo-isotopies
  in the topological, piecewise linear and smooth category.  
\end{abstract}

\maketitle

\newlength{\origlabelwidth} \setlength\origlabelwidth\labelwidth


\section{Introduction}
\label{sec:Introduction}

In this paper we investigate the Farrell--Jones Conjecture for Waldhausen's $A$--theory. Our main result is

\begin{theorem}[Main result]
\label{the:main_result}
Let $\calfj_A$ be the class of groups for which the Farrell--Jones~Conjecture \ref{con:FJC_for_A-theory_with_c_and_fwp}  
for (non-connective) $A$--theory with coefficients and finite wreath products holds.

\begin{enumerate}
\item \label{the:main_result:groups} 
The class  $\calfj_A$ contains the following groups:\

\begin{itemize}

\item Hyperbolic groups

\item  $\mathrm{CAT}(0)$--groups

\item Virtually poly-cyclic groups

\item  Cocompact lattices in almost connected Lie groups

\item Fundamental groups of (not necessarily compact) $d$--dimensional manifolds (possibly with boundary) for $d \le 3$

\end{itemize}

\item \label{the:main_result:inheritance} The class $\calfj_A$ has the following inheritance properties:

\begin{itemize}

\item If $G_1$ and $G_2$ belong to
      $\calfj_A$, then $G_1 \times G_2$ and $G_1 \ast G_2$ belong to $\calfj_A$.

\item If $H$ is a subgroup of $G$ and
      $G \in \calfj_A$, then $H \in \calfj_A$.

    \item Let $1 \to K \to G \xrightarrow{p} Q \to 1$ be an extension of groups. Suppose
      that $K$, $Q$ and $p^{-1}(C)$ for every infinite cyclic subgroup $C \subseteq Q$ belong
      to $\calfj_A$. Then $G$ belongs to $\calfj_A$.

\item If $H \subseteq G$ is a subgroup of $G$  with $[G:H] < \infty$ and $H \in \calfj_A$, then
          $G \in \calfj_A$.

 \item Let $\{G_i \mid i\in I\}$ be a
      directed system of groups (with not necessarily injective structure maps)
      such that $G_i \in \calfj_A$ for $i \in I$.
      Then $\colim_{i \in I} G_i$ belongs to $\calfj_A$.

\end{itemize}
\end{enumerate}
\end{theorem}

The Farrell--Jones Conjecture for $A$--theory aims at the computation of the homotopy groups
of $\bfA(BG)$ for a group $G$, where $\bfA \colon \SPACES \to \SPECTRA$ sends a space $X$ 
to the non-connective $A$--theory spectrum $\bfA(X)$ modeling Waldhausen's $A$--theory space $A(X)$.
More precisely, it predicts the bijectivity of the assembly map
\[
H_n^G(\EGF{G}{\calvcyc};\bfA^B) \xrightarrow{\cong} H_n^G(G/G;\bfA^B) = \pi_n(\bfA(BG))
\]
induced by the projection of the classifying space $\EGF{G}{\calvcyc}$ for the family of virtually cyclic subgroups of $G$
to $G/G$. It essentially reduces the computation of $\pi_n(\bfA(BG))$ to the computation of the system $\{\pi_n(\bfA(BV))\}$,
where $V$ ranges over the virtually cyclic subgroups $V$ of $G$.
Following the setup of Davis--L\"uck~\cite{Davis-Lueck(1998)}, we give the precise formulations of the various versions
of the Farrell--Jones Conjecture in \Cref{sec:The_Isomorphism_Conjecture}.
The equivalent original formulation of the Farrell--Jones Conjecture can be found in Farrell--Jones~\cite{Farrell-Jones(1993a)}.

\Cref{sec:Relations_between_the_various_theories} relates the Farrell--Jones Conjecture for $A$--theory
to the corresponding conjectures for other functors.
In particular, we discuss the equivalence of the conjectures for $\bfA$, $\bfWh^{\CAT}$ and $\bfP^{\CAT}$,
where the latter denote the non-connective spectra modeling the Whitehead space and the space of stable pseudo-isotopies,
with $\CAT$ being $\TOP$, $\PL$ or $\DIFF$.

As an illustration about the impact of the Farrell--Jones Conjecture for $A$--theory we discuss
applications to the automorphism groups of aspherical closed manifolds in 
\Cref{subsec:Topological_automorphism_groups_of_aspherical_closed_manifolds}, where also the proof 
of the following \Cref{the:widetilde(TOP)(M)/TOP(M)_and_bfN(A)} is given.

Let $\bfNA(\pt)$ be the Nil-term occurring in the Bass--Heller--Swan-isomorphisms for non-connective $A$--theory,
see~\cite{Huettemann-Klein-Waldhausen-Williams(2001),Huettemann-Klein-Vogell-Waldhausen-Williams(2002)}.
\begin{equation}
\pi_n(\bfA(S^1)) = \pi_n(\bfA(\pt)) \oplus \pi_{n-1}(\bfA(\pt)) \oplus \pi_n(\bfNA(\pt)) \oplus \pi_n(\bfNA(\pt)).
\label{Bass-Heller_Swan_for_A-theory}
\end{equation}
We conclude 
$\pi_n(\bfNA(\pt)) = \{0\}$ for $n \le 1$ and $\pi_n(\bfNA(\pt)) \otimes_{\IZ} \IQ = \{0\}$ for $n \in \IZ$ 
from \Cref{the:Relating_A-theory_to_algebraic_K_theory} and~\cite[Theorem~0.3]{Lueck-Steimle(2015splitasmb)}.
On the other hand, $\pi_n(\bfNA(\pt))$
for $n = 2,3$ is an infinite-dimensional $\IF_2$--vector space. For more information about $\pi_n(\bfNA(\pt))$
we refer to~\cite{Grunewald-Klein-Macko(2008),Hesselholt(2009)}.
The next result is already explained in the special case of closed manifolds with negative sectional curvature 
in~\cite[Section~6.3]{Weiss-Williams(2001)}, based on the work of 
Farrell and Jones~\cite{Farrell-Jones(1990c),Farrell-Jones(1991comp),Farrell-Jones(1990b),Farrell-Jones(1991),Farrell-Jones(1993a)},
and we can extend it  to torsionfree  hyperbolic groups.

\begin{theorem}\label{the:widetilde(TOP)(M)/TOP(M)_and_bfN(A)}
\
\begin{enumerate}

\item\label{the:widetilde(TOP)(M)/TOP(M)_and_bfN(A)-it1} Let $G$ be a torsionfree hyperbolic group. Then we get an equivalence
 \[ \bfWh^{\TOP}(BG) \simeq \bigvee_C \bfWh^{\TOP}(BC) \simeq \bigvee_C \bfNA(\pt) \vee \bfNA(\pt), \]
 where $C$ ranges over the conjugacy classes of maximal infinite cyclic subgroups of $G$.
 
 In particular, $\bfWh^{\TOP}(BG)$ is connective.
 
\item\label{the:widetilde(TOP)(M)/TOP(M)_and_bfN(A)-it2} Let $M$ be a smoothable aspherical closed manifold of dimension $\ge 10$,
whose fundamental group $\pi$ is hyperbolic.

Then there is a $\ZZ/2$--action on $\bfWh^{\TOP}(B\pi)$ such that we obtain for $1 \le n  \leq \min\{( \dim M -7 ) / 2, (\dim M - 4)/3\}$ isomorphisms
\[
 \pi_n(\TOP(M)) \cong \pi_{n+2}\Big( E\ZZ/2_+ \wedge_{\ZZ/2} \big( \bigvee_C \bfWh^{\TOP}(BC) \big) \Big)
\]
and an exact sequence
\[
 1 \to \pi_{2}\Big( E\ZZ/2_+ \wedge_{\ZZ/2} \big( \bigvee_C \bfWh^{\TOP}(BC) \big) \Big)   \to\pi_0(\TOP(M)) \to \Out(\pi) \to 1,
\]
where $C$ ranges over the conjugacy classes of maximal infinite cyclic subgroups of $\pi$.
\end{enumerate}
\end{theorem}

\begin{remark}
 The $\ZZ/2$--action we refer to in \Cref{the:widetilde(TOP)(M)/TOP(M)_and_bfN(A)} is induced by the one given in \cite{Weiss-Williams(1988)}.
 Vogell described in \cite{Vogell(1984)} another $\ZZ/2$--action on $\bfWh(B\pi)$ which depends on the choice of a spherical fibration over $B\pi$.
 It was shown in \cite{Huettemann-Klein-Vogell-Waldhausen-Williams(2002)} that for a certain fibration this action corresponds under the Bass--Heller--Swan decomposition to switching the Nil-terms via a homeomorphism.
 If the arguments presented in \cite{Huettemann-Klein-Vogell-Waldhausen-Williams(2002)} carry over to other spherical fibrations and the two actions on $\bfWh^{\TOP}(B\pi)$ agree for a suitably chosen fibration, then the homotopy orbits appearing in \Cref{the:widetilde(TOP)(M)/TOP(M)_and_bfN(A)} can be identified as
 \[
  E\ZZ/2_+ \wedge_{\ZZ/2} \big( \bigvee_C \bfWh^{\TOP}(BC) \big) \simeq \bigvee_C \bfNA(\pt).
 \]
 These issues will be discussed in \cite{PieperPhD}.
\end{remark}

The rest of the paper is devoted to the proof of \Cref{the:main_result}.
The main technical part of this paper concerns the proof for hyperbolic groups and $\mathrm{CAT}(0)$--groups.
It is given in \Cref{Proof of the Farrell--Jones Conjecture for hyperbolic and CAT(0)--groups} and \Cref{The Transfer: Final Part of the Proof}, and is motivated by the proof of the $K$--theoretic Farrell--Jones Conjecture for $\mathrm{CAT}(0)$--groups given in \cite{Wegner(2012)} based on the method of Bartels--L\"uck \cite{Bartels-Lueck(2012annals)}. Our approach, which is based on~\cite{Ullmann-Winges(2015)}, requires us to define an analog of the transfer on geometric modules which works on Waldhausen categories of controlled retractive spaces. Virtually poly-cyclic groups have already been treated by Ullmann--Winges~\cite{Ullmann-Winges(2015)}. 

In conjunction with the inheritance properties in \Cref{the:main_result}~\ref{the:main_result:inheritance},
the case of a cocompact lattice in an almost connected Lie group or a fundamental group of a (not necessarily compact) $d$--dimensional manifold (possibly with boundary) for $d \le 3$
follows via the argument presented in ~\cite{Bartels-Farrell-Lueck(2014)}.
The inheritance properties for the $A$--theoretic conjecture are taken care of in~\Cref{sec:Inheritance_properties_of_the_Isomorphism_Conjectures}.

\begin{remark}[Solvable groups] \label{rem_Solvable_groups} If one can show that
  $\calfj_A$ contains all virtually solvable groups, then it contains any (not necessarily cocompact)
  lattice in a second countable locally compact Hausdorff group with finitely many path
  components, the groups $GL_n(\IQ)$ and $GL_n(F(t))$ for $F(t)$ the function field over a
  finite field $F$, and all $S$--arithmetic groups. The arguments in
\cite{Kammeyer-Lueck-Rueping(2016),Rueping(2016_S-arithmetic)} carry over to show the prerequisites of \Cref{the:main_result} and \Cref{cor:hypcatfollowproof}, respectively.
\end{remark}

\subsection*{Acknowledgments.}
The idea to use the categories of ``cellwise $0$-controlled morphisms'' in
\Cref{The Transfer: Final Part of the Proof} is due to Arthur Bartels
and Paul Bubenzer. 
The paper is financially supported by the Leibniz-Preis of
the second author granted by the {DFG}, the ERC Advanced Grant ``KL2MG-interactions''
(no.  662400) of the second author granted by the European Research Council,
the Cluster of Excellence  ``Hausdorff Center for Mathematics'' at Bonn, the SFB 647 ``Space -- Time -- Matter'' at Berlin, the SFB 878 ``Groups, Geometry \& Actions'' at M\"unster,
and the Junior Hausdorff Trimester Program ``Topology'' at the Hausdorff Research Institute for Mathematics (HIM).

\tableofcontents


\section{The  Isomorphism Conjecture}
\label{sec:The_Isomorphism_Conjecture}

In this section we state various versions of the Isomorphism Conjectures we
want to consider.  We assume familiarity with the notion of a $G$-equivariant
homology theory from~\cite{Davis-Lueck(1998)} and the notion of an equivariant
homology theory from~\cite{Lueck(2002b)}.  As usual, we use a convenient
category of compactly generated spaces.


\subsection{The  Meta-Isomorphism Conjecture for functors from spaces to spectra}
\label{subsec:The_Meta_Isomorphism_Conjecture_for_functors_from_spaces_to_spectra}

Let $\bfS \colon \SPACES\to \SPECTRA$ be a covariant functor. Throughout this section we
will assume that \emph{$\bfS$  respects weak equivalences and disjoint unions}, i.e., a weak
homotopy equivalence of spaces $f \colon X \to Y$ is sent to a weak homotopy equivalence of spectra $\bfS(f) \colon \bfS(X) \to \bfS(Y)$
and for a collection of spaces $\{X_i \mid i \in I\}$ 
for an arbitrary index set $I$ the canonical map 
\begin{equation}
\bigvee_{i \in I} \bfS(X_i) \to \bfS\biggl(\coprod_{i \in I} X_i\biggr)
\end{equation}
is a weak homotopy equivalence of spectra.  Weak equivalences of spectra are
understood to be the stable equivalences, i.e., the maps which induce
isomorphisms on all stable homotopy groups.
We obtain a covariant functor
\begin{equation}
\bfS^B \colon \GROUPOIDS \to \SPECTRA, \quad \calg \mapsto \bfS(B\calg),
\label{bfs_upper_B}
\end{equation}
where $B\calg$ is the classifying space of the groupoid $\calg$ which is the geometric
realization of the simplicial set given by its nerve and denoted by
$B^{\operatorname{bar}}\calg$ in~\cite[page~227]{Davis-Lueck(1998)}.  Let
$H^?_n(-;\bfS^B)$ be the equivariant homology theory in the sense
of~\cite[Section~1]{Lueck(2002b)} which is associated to $\bfS^B$ by the construction
in~\cite[Proposition~157 on page~796]{Lueck-Reich(2005)}. Equivariant homology theory
essentially means that we get for every group $G$ a $G$--homology theory $H^G_n(-;\bfS^B)$
satisfying the disjoint union axiom and for every group homomorphism $\alpha \colon H \to G$ 
and $H$--CW--pair $(X,A)$ we get natural maps compatible with boundary homomorphisms of
pairs $H^H_*(X,A;\bfS^B) \to H^G_*(\alpha_*(X,A);\bfS^B)$ which are bijective if the kernel of
$\alpha$ acts freely on $X$.  Moreover, for any group $G$, subgroup $H \subseteq G$ and $n \in \IZ$ 
we have canonical identifications
\begin{equation*}
H_n^G(G/H;\bfS^B) \cong H_n^H(H/H;\bfS^B) \cong \pi_n(\bfS(BH)).
\end{equation*}

\begin{conjecture}[Meta-Isomorphism Conjecture for functors from spaces to spectra] 
\label{con:Meta_Isomorphisms_Conjecture_for_functors_from_spaces_to_spectra}
Let $\bfS \colon \SPACES\to \SPECTRA$ be a covariant functor which 
respects weak equivalences and disjoint unions. The group $G$ satisfies the 
\emph{Meta-Isomorphism Conjecture for $\bfS$}
with respect to the family $\calf$ of subgroups of $G$
if  the assembly map induced by the projection $\pr \colon \EGF{G}{\calf} \to G/G$
\[
H_n^G(\pr;\bfS^B) \colon H^G_n(\EGF{G}{\calf};\bfS^B) \to H_n^G(G/G;\bfS^B) \cong \pi_n(\bfS(BG))
\]
is bijective for all $n \in \IZ$.
\end{conjecture}

\begin{example}[The $K$-- and $L$--theoretic Farrell--Jones Conjecture]
\label{exa:The_K-and_L-theoreticFJC}
Let $R$ be a ring (with involution). There are covariant
functors~\cite[Theorem 158]{Lueck-Reich(2005)}
\begin{eqnarray*}
\bfK_R \colon \GROUPOIDS & \to & \SPECTRA,
\\
\bfL_R^{\langle -\infty \rangle} \colon \GROUPOIDS & \to & \SPECTRA,
\end{eqnarray*}
such that for every group $G$, which we can consider as a groupoid $\underline{G}$ with precisely one
object and $G$ as its group of automorphisms, and $n \in \IZ$ we have
\begin{eqnarray*}
K_n(RG) 
& = & 
\pi_n(\bfK_R(\underline{G})),
\\
L_n^{\langle -\infty \rangle} (RG) 
& = & 
\pi_n(\bfL^{\langle -\infty \rangle} _R(\underline{G})).
\end{eqnarray*}
Then the $K$--theoretic and $L$--theoretic Farrell--Jones Conjectures, which
were originally
formulated in~\cite[1.6 on page 257]{Farrell-Jones(1993a)}, are equivalent to the
statement that the covariant functors $\bfS \colon \SPACES \to \SPECTRA$, given by the
composition of $\bfK_R$ and$\bfL_R$, with the functor sending a space to its fundamental
groupoid, satisfy the Meta-Isomorphism
Conjecture~\ref{con:Meta_Isomorphisms_Conjecture_for_functors_from_spaces_to_spectra} for
the family $\calvcyc$ of virtually cyclic subgroups of $G$.
\end{example}

Our main example is the following case.  Let $\bfA \colon \SPACES\to \SPECTRA$ 
be the functor sending a space $X$ to the
spectrum $\bfA(X)$ given by the non-connective version
of Waldhausen's algebraic $K$--theory of spaces in the sense of \cite{Ullmann-Winges(2015)}.

\begin{lemma}\label{lem:bfA_respects_weak_equivalences_and_disjoint_unions}\
\begin{enumerate}

\item \label{lem:bfA_respects_weak_equivalences_and_disjoint_unions:dis_union}
The functor  $\bfA \colon \SPACES\to \SPECTRA$ respects weak equivalences
and disjoint unions.

\item \label{lem:bfA_respects_weak_equivalences_and_disjoint_unions:hocolim}
 For any directed systems of spaces $\{X_i \mid i \in I\}$ indexed over an arbitrary
  directed set $I$ the canonical map
  \[
  \hocolim_{i \in I} \bfA(X_i) \to \bfA\bigl(\hocolim_{i \in I} X_i\bigr)
  \]
  is a weak homotopy equivalence. 
\end{enumerate}

\end{lemma}
\begin{proof}
In the connective case, Waldhausen proved in~\cite[Proposition~2.1.7]{Waldhausen(1985)}) that $A$--theory preserves weak equivalences. The other two properties follow upon inspection of the explicit model as finite retractive CW--complexes.

Note that the algebraic $K$--theory functor which sends $X$ to $K(\IZ\Pi(X))$ enjoys the properties claimed for $\bfA$. Since Vogell showed that the linearization map $\bfL \colon \bfA \rightarrow \bfK$ induces an isomorphism on all non-positive homotopy groups \cite{Vogell(1991)}, the general case follows.

It will be shown in \cite{PieperPhD} that the non-connective deloopings described by Ullmann--Winges in~\cite{Ullmann-Winges(2015)} and Vogell in~\cite{Vogell(1990)} are equivalent.
\end{proof}

\begin{conjecture}[The Farrell--Jones Conjecture for $A$--theory]
  \label{con:FJC_for_A-theory}
  A group $G$ satisfies the Farrell--Jones Conjecture for $A$--theory if the
  Meta-Isomorphism
  Conjecture~\ref{con:Meta_Isomorphisms_Conjecture_for_functors_from_spaces_to_spectra}
  holds for $\bfA \colon \SPACES\to \SPECTRA$ and the family $\calvcyc$, i.e., for
  every $n \in \IZ$ the projection $\EGF{G}{\calvcyc} \to G/G$ induces an isomorphism
  \[
  H_n^G(\pr;\bfA^B) \colon H^G_n(\EGF{G}{\calvcyc};\bfA^B) \to H_n^G(G/G;\bfA^B) =  \pi_n(\bfA(BG)).
  \]
\end{conjecture}


\subsection{The Meta--Isomorphism Conjecture for functors from spaces to spectra with coefficients}
\label{subsec:The_Meta_Isomorphism_Conjecture_for_functors_from_spaces_to_spectra_with_c}

Let $G$ be a group and $Z$ be a $G$--CW--complex.  Define a covariant $\Or(G)$--spectrum
\begin{equation}
  \bfS^G_Z \colon\Or(G) \to \SPECTRA, 
\quad G/H \mapsto \bfS(G/H \times_G Z),
  \label{bfS:upper_G_Z}
\end{equation}
where  $G/H \times_G Z$ is the orbit space of the diagonal $G$--action on $G/H \times Z$.
Notice that there is an obvious homeomorphism $G/H \times_G Z \xrightarrow{\cong} Z/H$.
Denote by $H_n^G(-;\bfS^G_Z)$ the $G$--homology theory  in the sense of~\cite[Section~1]{Lueck(2002b)} which is associated to  $\bfS^G_Z$ 
by the construction of~\cite[Proposition~156 on page~795]{Lueck-Reich(2005)} and satisfies 
$H_n^G(G/H;\bfS^G_Z) \cong \pi_n(\bfS^G_Z(G/H)) = \pi_n(\bfS(Z/H))$ for any homogeneous $G$--space $G/H$ and $n \in \IZ$.

\begin{conjecture}[Meta-Isomorphism Conjecture for functors from spaces to
  spectra with coefficients]
  \label{con:Meta_Isomorphisms_Conjecture_for_functors_from_spaces_to_spectra_with_c}
  Let $\bfS \colon \SPACES\to \SPECTRA$ be a covariant functor which respects
  weak equivalences and disjoint unions. The group $G$ satisfies the
  \emph{Meta-Isomorphism Conjecture for $\bfS$ with coefficients}
  \index{Conjecture!Meta-Isomorphism Conjecture for functors from spaces to
  spectra with coefficients} with respect to the family $\calf$ of subgroups
  of $G$ if for any free $G$--CW--complex $Z$ the assembly map
  \[
  H^G_n(\pr;\bfS^G_Z) \colon H^G_n(\EGF{G}{\calf};\bfS^G_Z) \to
  H_n^G(G/G;\bfS^G_Z) = \pi_n(\bfS(Z/G)),
  \]
  induced by the projection $\pr \colon \EGF{G}{\calf} \to G/G$, is bijective
  for all $n \in \IZ$.
\end{conjecture}

\begin{example}[$Z = EG$] \label{exa:Z_is_EG} If we take $Z = EG$ in
  \Cref{con:Meta_Isomorphisms_Conjecture_for_functors_from_spaces_to_spectra_with_c},
  then
  \Cref{con:Meta_Isomorphisms_Conjecture_for_functors_from_spaces_to_spectra_with_c}
  reduces to
  \Cref{con:Meta_Isomorphisms_Conjecture_for_functors_from_spaces_to_spectra}.
  Namely,  for a $G$--set $S$ let $\calt^G(S)$ be its \emph{transport groupoid} whose set of objects is $S$,
  the set of morphisms from $s_1$ to $s_2$ is the set $\{g \in G \mid s_2 = gs_1\}$ and composition comes from the 
  multiplication in $G$. There is a homotopy equivalence $B\calt^G(G/H) \xrightarrow{\simeq} G/H \times_G EG$
  which is natural in $G/H$. Hence we get a weak homotopy equivalence of $\Or(G)$--spectra
  $\bfS^B(\calt^G(G/?)) \xrightarrow{\simeq} \bfS^G_{EG} $. It induces an isomorphism of $G$--homology theories,
  see~\cite[Lemma~4.6]{Davis-Lueck(1998)} 
  \[
  H_*^G(-;\bfS^B) \xrightarrow{\cong}   H_*^G(-;\bfS^G_{EG}).
  \]
\end{example}

\begin{remark}[Relation to the original formulation]
  \label{rem:relation_to_original}
  In~\cite[Section~1.7 on page~262]{Farrell-Jones(1993a)} Farrell and Jones formulate a
  fibered version of their conjectures for a covariant functor $\bfS \colon \SPACES \to \SPECTRA$ for
  every (Serre) fibration $\xi \colon Y \to X$ over a
  connected CW-complex $X$.  In our set-up this corresponds to choosing $Z$ to be the
  total space of the fibration obtained from $Y \to X$ by pulling back along the universal
  covering $\widetilde{X} \to X$. This space $Z$ is a free $G$--CW--complex  for $G= \pi_1 (X)$.
  Note that  every free $G$--CW--complex $Z$ can always be obtained in this fashion
  from the fiber bundle  $EG \times_G Z \to BG$ up to $G$--homotopy, 
  compare~\cite[Corollary~2.2.1 on page~263]{Farrell-Jones(1993a)}. 

  We sketch the proof of this identification. Let $A$ be a $G$--CW--complex.  Let 
  $f \colon \cale(X) \to X$ be the map obtained by taking the quotient of the 
  $G   = \pi_1(X)$--action on the $G$--map $A \times \widetilde{X} \to \widetilde{X}$ given by the
  projection. Denote by $\widehat{p} \colon \cale(\xi) \to \cale(X)$ the pullback of
  $\xi$ with $f$. Let $q\colon \cale(\xi) \to A/G$ be the composite of 
  $\widehat{p}$ with the map $\cale(X) \to A/G$ induced by the projection $A \times \widetilde{X} \to A$.
  This is a stratified fibration and one can consider the spectrum
  $\IH(A/G;\cals(q))$ in the sense of Quinn~\cite[Section~8]{Quinn(1979a)}. Put
  \[
  H_n^G(A;\xi) := \pi_n(\IH(A/G;\cals(q))).
  \]
  The projection $\pr \colon A \to G/G$ induces a map
  \begin{equation}
  a(A) \colon  \IH(A/G;\cals(q)) \to  \IH(G/G;\cals(Y\to G/G)) = \bfS(Y),
  \label{assembly_a_la_Quinn}
\end{equation} 
which is the assembly map in~\cite[Section~1.7 on page~262]{Farrell-Jones(1993a)} if we take $A = \EGF{G}{\calvcyc}$.  The construction of $H_n^G(A;\xi) := \IH(A/G;\cals(q))$ is very complicated, but,
fortunately, for us only two facts are relevant.  We obtain a $G$--homology
theory $H_n^G(-;\xi)$ and for every $H \subseteq G$ we get a natural
identification $H_n^G(G/H;\xi) = \bfS^G_Z(G/H)$. Hence the
functor $\GCWCOMPLEXES \to \SPECTRA$ given by $A \mapsto \IH(A/G;\cals(q))$ is
weakly excisive and its restriction to $\Or(G)$ is the functor $\bfS^G_Z$. We conclude 
from~\cite[Theorem~6.3]{Davis-Lueck(1998)} that the map~\eqref{assembly_a_la_Quinn} can
be identified with the map induced by the projection $A \to G/G$
\[
H_n^G(A;\bfS^G_Z) \to H_n^G(G/G;\bfS^G_Z) = \pi_n(\bfS(Z/G)) = \pi_n(\bfS(Y)),
\]
which appears in Meta-Isomorphism Conjecture~\ref{con:Meta_Isomorphisms_Conjecture_for_functors_from_spaces_to_spectra_with_c}
for functors from spaces to spectra with coefficients.
\end{remark}

\begin{remark}[The condition free is necessary in
  \Cref{con:Meta_Isomorphisms_Conjecture_for_functors_from_spaces_to_spectra_with_c}]
  The \label{rem:The_condition_free_is_necessary_in_MC_spaces_spectra_with_coeff}
  \Cref{con:Meta_Isomorphisms_Conjecture_for_functors_from_spaces_to_spectra_with_c}
  is only true very rarely if we drop the condition that $Z$ is free. Take for instance $Z = G/G$.
  Then
  \Cref{con:Meta_Isomorphisms_Conjecture_for_functors_from_spaces_to_spectra_with_c}
  predicts that the projection $\EGF{G}{\calf}/G \to G/G$ induces for all $n \in \IZ$ an
  isomorphism
  \[
  H_n(\pr;\bfS(\pt)) \colon H_n(\EGF{G}{\calf}/G;\bfS(\pt)) \to H_n(\pt,\bfS(\pt))
  \]
  where $H_*(-;\bfS(\pt))$ is the (non-equivariant) homology theory associated to the
  spectrum $\bfS(\pt)$.  This statement is in general wrong, except in extreme cases such
  as $\calf = \calall$.
\end{remark}

\begin{conjecture}[The Farrell--Jones Conjecture for $A$--theory with coefficients]
  \label{con:FJC_for_A-theory_with_c}
  A group $G$ satisfies the \emph{Farrell--Jones Conjecture for $A$--theory with coefficients} if the
  Meta-Isomorphism   Conjecture~\ref{con:Meta_Isomorphisms_Conjecture_for_functors_from_spaces_to_spectra_with_c}
  with coefficients 
  holds for $\bfA \colon \SPACES \to \SPECTRA $ and the family $\calvcyc$,
  i.e., for
  every $n \in \IZ$ and free $G$--CW--complex $Z$ the projection $\EGF{G}{\calvcyc} \to G/G$ induces an isomorphism
  \[
  H_n^G(\pr;\bfA^G_Z) \colon H^G_n(\EGF{G}{\calvcyc};\bfA^G_Z) 
  \to H_n^G(G/G;\bfA^G_Z) =  \pi_n(\bfA(Z/G)).
  \]
\end{conjecture}


\subsection{The Meta--Isomorphism Conjecture for functors from spaces to spectra with coefficients and finite wreath products}
\label{subsec:The_Meta_Isomorphism_Conjecture_for_functors_from_spaces_to_spectra_with_c_and_fwp}

There are also versions with finite wreath products.
Recall that for groups $G$ and $F$ their \emph{wreath product $G\wr F$} is defined to be the semi-direct product
$\left(\prod_F G \right) \rtimes F$, where $F$ acts on $\prod_F G$ by permuting the factors.
Fix a class of groups $\calc$ which is closed under
isomorphisms, taking subgroups and taking quotients. Examples are the classes 
$\calfin$ and $\calvcyc$ of finite and of virtually cyclic groups.
For a group $G$ define the family of subgroups $\calc(G) :=   \{ K \subseteq G, K \in \calc\}$.

\begin{conjecture}[The Meta-Isomorphism Conjecture for functors from spaces to spectra
  with coefficients and finite wreath products]
  \label{con:The_Meta_Isomorphisms_Conjecture_for_functors_from_spaces_to_spectra_with_c_and_fwp}
  Let $\bfS \colon \SPACES\to \SPECTRA$ be a covariant functor which respects weak
  equivalences and disjoint unions. The group $G$ satisfies the \emph{Meta-Isomorphism
    Conjecture with coefficients and finite wreath products} for the functor
  $\bfS \colon \SPACES \to \SPECTRA$ with respect to the class $\calc$  of groups,
   if, for any finite group $F$, the wreath product $G \wr F$ satisfies the 
  Meta-Isomorphism   Conjecture~\ref{con:Meta_Isomorphisms_Conjecture_for_functors_from_spaces_to_spectra_with_c}
  with coefficients for the functor $\bfS \colon \SPACES \to \SPECTRA$ with respect to the family $\calc(G \wr F)$ of   subgroups of $G$.
\end{conjecture}

\begin{conjecture}[The Farrell--Jones Conjecture for $A$--theory with coefficients and finite wreath products]
  \label{con:FJC_for_A-theory_with_c_and_fwp} A group $G$ satisfies
the \emph{Farrell--Jones Conjecture for $A$--theory with coefficients and finite
  wreath products} if the Meta-Isomorphism
Conjecture~\ref{con:The_Meta_Isomorphisms_Conjecture_for_functors_from_spaces_to_spectra_with_c_and_fwp}
with  coefficients and finite wreath products holds for $\bfA \colon \SPACES \to \SPECTRA $ and the class $\calvcyc$ of virtually cyclic groups.
\end{conjecture}

The next two lemmas will be needed later.

\begin{lemma}\label{lem:bfS_homology_theory}
Let $\bfE$ be a spectrum such that $\bfS \colon \SPACES \to \SPECTRA$ is given by $Y \mapsto Y_+ \wedge \bfE$. 

\begin{enumerate}

\item \label{lem:bfS_homology_theory:pr}
Then for any group $G$, any $G$--CW--complex $X$ which is contractible (after forgetting
the $G$--action), and any free $G$--CW--complex $Z$ the projection $X \to G/G$ induces for
all $n \in \IZ$ an isomorphism
\[
H_n^G(X;\bfS^G_Z) \xrightarrow{\cong} H_n^G(G/G;\bfS^G_Z).
\]

\item \label{lem:bfS_homology_theory:MCwc}
\Cref{con:Meta_Isomorphisms_Conjecture_for_functors_from_spaces_to_spectra},
\Cref{con:Meta_Isomorphisms_Conjecture_for_functors_from_spaces_to_spectra_with_c} 
and \Cref{con:The_Meta_Isomorphisms_Conjecture_for_functors_from_spaces_to_spectra_with_c_and_fwp}
hold for such an $\bfS$ for every group $G$ and every family $\calf$ of subgroups of $G$.
\end{enumerate}
\end{lemma}

An $\bfS$ given by $Y \mapsto Y_+ \wedge \bfE$ is a homology theory, and thus
the lemma states that the conjectures hold for homology theories.

\begin{proof}~\ref{lem:bfS_homology_theory:pr}
There are   natural isomorphisms of spectra
\begin{eqnarray*}
\lefteqn{\map_G((G/?),X)_+ \wedge_{\Or(G)} \bigl((G/? \times_G Z)_+ \wedge \bfE\bigr)}
\\
& \xrightarrow{\cong} &
\bigl((\map_G((G/?),X) \times_{\Or(G)} G/?) \times_G Z\bigr)_+ \wedge \bfE
\\
& \xrightarrow{\cong} &
 (X \times_G Z)_+ \wedge \bfE,
\end{eqnarray*}
where the second isomorphism comes from the $G$--homeomorphism
\[
\map_G((G/?),X) \times_{\Or(G)} G/? \xrightarrow{\cong} X
\]
of~\cite[Theorem~7.4~(1)]{Davis-Lueck(1998)}.  Since
$Z$ is a free $G$--CW--complex and $X$ is contractible (after forgetting the group action),
the projection $X \times_G Z \to G/G \times_G Z$ is a homotopy equivalence and hence
induces a weak homotopy equivalence
\[
(X \times_G Z)_+ \wedge \bfE \xrightarrow{\simeq} (G/G \times_G Z)_+ \wedge \bfE,
\]
Thus we get a weak homotopy equivalence
\[
\map_G((G/?),X)_+ \wedge_{\Or(G)} \bigl((G/? \times_G Z)_+ \wedge \bfE\bigr)
\to (G/G \times_G Z)_+ \wedge \bfE.
\]
Under the identifications coming from the definitions
\begin{eqnarray*}
H_n^G(X;\bfS^G_Z)  
& = & 
\pi_n\left(\map_G((G/?),X)_+ \wedge_{\Or(G)} \bigl((G/? \times_G Z)_+ \wedge \bfE\bigr)\right),
\\
H_n^G(G/G;\bfS^G_Z) 
& = & 
\pi_n\left((G/G \times_G Z)_+ \wedge \bfE\right),
\end{eqnarray*}
this weak homotopy equivalence induces on homotopy groups the isomorphism
$H_n^G(X;\bfS^G_Z) \to H_n^G(G/G;\bfS^G_Z)$.
\\[1mm]~\ref{lem:bfS_homology_theory:MCwc} This follows from assertion~\ref{lem:bfS_homology_theory:pr}.
\end{proof}

\begin{lemma}
  \label{lem:cofibration_sequences_of_bfS}
  Let $\bfS, \bfT, \bfU \colon \SPACES\to \SPECTRA$ be covariant functors which respect
  weak equivalences and disjoint unions.  Let $\bfi \colon \bfS \to \bfT$ and $\bfp \colon
  \bfT \to \bfU$ be natural transformations such that for any space $Y$ the sequence of spectra
  $\bfS(Y) \xrightarrow{\bfi(Y)} \bfT(Y) \xrightarrow{\bfp(Y)} \bfU(Y)$ is up to weak
  homotopy equivalence a cofiber sequence of spectra.

  \begin{enumerate}

  \item \label{lem:cofibration_sequences_of_bfS:long_exact_sequence} 
  Then we obtain for every group $G$  and all $G$--CW--complexes $X$ and $Z$ a
  natural long exact sequence
\begin{multline*}
\cdots \to H_n^G(X;\bfS_Z^G) \to H_n^G(X;\bfT_Z^G) \to H_n^G(X;\bfU_Z^G) 
\\
\to H_{n-1}^G(X;\bfS_Z^G) \to H_{n-1}^G(X;\bfT_Z^G) \to H_{n-1}^G(X;\bfU_Z^G) \to \cdots \;.
\end{multline*}

\item \label{lem:cofibration_sequences_of_bfS:MCwc}  
Let $G$ be a group and  $\calf$ be a family of subgroups of $G$.
Then Meta-Isomorphism Conjecture~\ref{con:Meta_Isomorphisms_Conjecture_for_functors_from_spaces_to_spectra} 
for functors from spaces to spectra holds for all three functors $\bfS$, $\bfT$ and $\bfU$ for $(G,\calf)$ 
if  it holds  for two of the  functors $\bfS$, $\bfT$ and $\bfU$ for $(G,\calf)$.

The analogous statement is true for the Meta-Isomorphism 
Conjecture~\ref{con:Meta_Isomorphisms_Conjecture_for_functors_from_spaces_to_spectra_with_c}
for functors from spaces to spectra with coefficients and for the
Meta-Isomorphism Conjecture~\ref{con:The_Meta_Isomorphisms_Conjecture_for_functors_from_spaces_to_spectra_with_c_and_fwp} 
for functors from spaces to spectra with coefficients and finite wreath products.
\end{enumerate}
\end{lemma}
\begin{proof}~\ref{lem:cofibration_sequences_of_bfS:long_exact_sequence} 
The version for spectra of~\cite[Theorem~3.11]{Davis-Lueck(1998)} implies that we obtain, up to weak homotopy equivalence, a cofiber sequence of spectra
\begin{multline*}
\map_G(G/?,X)_+ \wedge_{\Or(G)} \bfS(G/? \times_{G} Z) \to \map_G(G/?,X)_+ \wedge_{\Or(G)} \bfT(G/? \times_{G} Z) 
\\
\to \map_G(G/?,X)_+ \wedge_{\Or(G)} \bfU(G/? \times_{G} Z).
\end{multline*}%
and passing to its associated long exact sequence of homotopy groups yields the result.
\ref{lem:cofibration_sequences_of_bfS:MCwc}
This follows from assertion~\ref{lem:cofibration_sequences_of_bfS:long_exact_sequence} 
and the Five-Lemma.
\end{proof}


\section{Relations between the conjectures for various theories}
\label{sec:Relations_between_the_various_theories}

There are other prominent covariant functors $\SPACES \to \SPECTRA$ which
respect weak homotopy equivalences and disjoint unions. Notice in the sequel
that we are always considering the non-connective versions. We are thinking of
the stable pseudo-isotopy spectrum $\bfP^{\CAT}$ and the Whitehead spectrum
$\bfWh^{\CAT}$, where $\CAT$ can be the topological category $\TOP$, the
$PL$--category $\PL$ or the smooth category $\DIFF$. For the
definition of $\bfP^{\CAT}$
we refer to~\cite{EnkelmannPhD, PieperPhD, Weiss-Williams(1988)}. 

Usually, the Whitehead spectrum is defined as a connective spectrum,
see~\cite{Jahren-Rognes-Waldhausen(2013)}, and see \cite[Section~2.2]{Weiss-Williams(1993)} for a definition of the classical assembly
map. We make the obvious generalization.
\begin{dfn}\label{def: non-conn whspectrum}
Let $\CAT$ be $\TOP$ or $\PL$. The \emph{topological or piecewise-linear
non-connective Whitehead spectrum} $\bfWh^{\CAT}(X)$ is the homotopy cofiber
of the classical assembly map in non-connective $A$--theory:
\[
 X_+ \wedge \bfA(\pt) \to \bfA(X) \to \bfWh^{\CAT}(X).
\]
Further, we define the \emph{smooth non-connective Whitehead spectrum}
$\bfWh^{\DIFF}(X)$ as the homotopy cofiber of the sequence
\[
 \Sigma^{\infty} X_+ \to \bfA(X) \to \bfWh^{\DIFF}(X)
\]
where $\Sigma^{\infty} X_+ \to \bfA(X)$ factors as the unit map $\Sigma^\infty X_+ = X_+\wedge\Ss \to X_+ \wedge \bfA(\pt)$ and assembly.
\end{dfn}

\begin{theorem}[Relations between the various functors]
\label{the:relations_between_the_various_functors}
\
\begin{enumerate}

\item \label{the:relations_between_the_various_functors:calp_PL_and_wh_PL}
There is a zigzag of natural equivalences,
\[
\bfP^{\CAT} \stackrel{\simeq}{\longleftrightarrow} \Omega^2\bfWh^{\CAT},
\]
where $\CAT$ can be taken to be $\TOP$, $\PL$, or $\DIFF$.

\item  \label{the:relations_between_the_various_functors:wh_TOP_and_wh_PL}
The canonical map
\[
\bfP^{\PL} \xrightarrow{\simeq}  \bfP^{\TOP}
\]
is a natural equivalence.
\end{enumerate}
\end{theorem}
\begin{proof}
 The connective, objectwise case of \ref{the:relations_between_the_various_functors:calp_PL_and_wh_PL} follows from the equivalence $\bfP(M)\simeq \Omega^2\bfWh^{PL}(M)$, which was originally stated in~\cite{Waldhausen(1978)} and fully proved in~\cite[Theorem~0.2]{Jahren-Rognes-Waldhausen(2013)}.
 
 There are some issues concerning the full functoriality of pseudo-isotopy, which will be clarified in \cite{EnkelmannPhD} and \cite{PieperPhD}. The full statement will be established in \cite{PieperPhD}.
 
 The objectwise version of \ref{the:relations_between_the_various_functors:wh_TOP_and_wh_PL} has been shown in \cite{Burghelea-Lashof(1974)}. An argument for the full statement will be given in \cite{EnkelmannPhD}.

\end{proof}

\begin{lemma} \label{lem:equivalence_of_FJ_for_the_various_functors} 
  If the Meta-Isomorphism
  Conjecture~\ref{con:Meta_Isomorphisms_Conjecture_for_functors_from_spaces_to_spectra}
  for functors from spaces to spectra holds for the group $G$ and the family $\calf$
  for one of the functors $ \bfA$, $\bfWh^{\TOP}$, $\bfWh^{\PL}$, $\bfWh^{\DIFF}$,
  $\bfP^{\TOP}$, $\bfP^{\PL}$, and $\bfP^{\DIFF}$, then it holds for all of them.

  The analogous statement holds for the Meta-Isomorphism
  Conjecture~\ref{con:Meta_Isomorphisms_Conjecture_for_functors_from_spaces_to_spectra_with_c}
  for functors from spaces to spectra with coefficients and for the Meta-Isomorphism
  Conjecture~\ref{con:The_Meta_Isomorphisms_Conjecture_for_functors_from_spaces_to_spectra_with_c_and_fwp}
  for functors from spaces to spectra with coefficients and finite wreath products.
\end{lemma}
\begin{proof} 
This follows from \Cref{lem:bfS_homology_theory} and \Cref{lem:cofibration_sequences_of_bfS}.
\end{proof}

\begin{remark}[The non-connective spectrum of stable $h$--cobordisms]
  \label{rem:stable_space_of_h-cobordism}
  There is also the non-connective stable $h$--cobordism spectrum $\bfH^{\CAT}(M)$ of a compact manifold (possibly with boundary) $M$. Note that $h$--cobordisms are (usually) only defined as a functor in codimension zero embeddings. As such, they are related to the previous functors. For every compact manifold $M$ (possibly with boundary) there are natural weak homotopy equivalences
  \[
  \bfH^{\CAT}(M) \xrightarrow{\simeq} \Omega\bfWh^{\CAT}(M)
  \]
  and
  \[
  \bfP^{\CAT}(M)\xrightarrow{\simeq} \Omega\bfH^{\CAT}(M).
  \]
  For the proof and more information we refer to \cite{Jahren-Rognes-Waldhausen(2013)}.
\end{remark}

Finally, we explain the relationship between $A$--theory and algebraic $K$--theory of integral group rings.

For a space $X$ denote its fundamental groupoid by $\Pi(X)$. There is a so called \emph{linearization map}, natural in $X$,
\begin{eqnarray}
\bfL(X) \colon \bfA(X) & \to & \bfK(\IZ \Pi_1(X))
\label{bfl_bfA_to_bfK}
\end{eqnarray}
The next result follows combining~\cite[Section~4]{Vogell(1991)} 
and~\cite[Proposition~2.2 and Proposition~2.3]{Waldhausen(1978)}.

\begin{theorem}[Connectivity of the linearization map]
  \label{the:Connectivity_of_the_linearization_map}
  Let $X$ be a  CW--complex. Then:

\begin{enumerate}

\item \label{the:Connectivity_of_the_linearization_map:2_connected}
The linearization map $\bfL(X)$ of~\eqref{bfl_bfA_to_bfK} is $2$--connected, i.e., the map
\[
L_n(X)  := \pi_n(\bfL(X)) \colon A_n(X) \to K_n(\IZ \Pi(X))
\]
is bijective for $n \le 1$ and surjective for $n = 2$.

\item \label{the:Connectivity_of_the_linearization_map:rational} 
The map $L_n$ is rationally  bijective for all $n \in \IZ$,  provided that each component of $X$ is aspherical.
\end{enumerate}
\end{theorem}

This implies that the $K$--theoretic Farrell--Jones Conjecture for $\IZ G$ and the $A$--theoretic Farrell--Jones
Conjecture for $\bfA(BG)$ are equivalent in degree $\le 1$ and rationally equivalent
in all degrees. More precisely, we have

\begin{theorem}[Relating A-theory to algebraic $K$--theory]
\label{the:Relating_A-theory_to_algebraic_K_theory}
Consider a group $G$ and a family $\calf$ of subgroups of $G$.
The linearization map~\eqref{bfl_bfA_to_bfK} and the projection $\EGF{G}{\calf} \to G/G$ yield
a commutative diagram
\[
\xymatrix{H_n^G(\EGF{G}{\calf};\bfA^B) \ar[r] \ar[d]
&
H_n^G(G/G;\bfA^B) = A_n(BG) \ar[d]
\\
H_n^G(\EGF{G}{\calf};\bfK_{\IZ}) \ar[r] 
&
H_n^G(G/G;\bfK_{\IZ}) = K_n(\IZ G)
}
\]
where $\bfK_{\IZ} \colon \GROUPOIDS \to \SPECTRA$ has been recalled in
\Cref{exa:The_K-and_L-theoreticFJC}.  The vertical arrows are
bijective for $n \le 1$ and surjective for $n =2$. They are rationally
bijective for all $n \in \IZ$.
\end{theorem}


\section{Some applications to automorphism groups of aspherical closed manifolds}
\label{sec:Some_applications_to_automorphism_groups_of_and_bundles_over_aspherical_closed_manifolds}

Before we begin with the proof of \Cref{the:main_result}, we want to illustrate
the impact of the Farrell--Jones Conjecture by discussing  automorphism
groups of aspherical closed manifolds. For rational computations the Farrell--Jones
Conjecture for $K$--theory and $L$--theory suffices. For potential integral computations one
needs the Farrell--Jones Conjecture for $A$--theory and for $L$--theory.  More details
about automorphism groups of closed manifolds can be found in~\cite{Weiss-Williams(2001)}.


\subsection{Topological automorphism groups of aspherical closed manifolds}
\label{subsec:Topological_automorphism_groups_of_aspherical_closed_manifolds}

Let $\TOP(M)$ be the topological group of self-homeomorphisms of the closed manifold $M$.
Denote by $\G(M)$ the monoid of self homotopy equivalences $M \to M$. Let  $\widetilde{\TOP}(M)$ and
$\widetilde{\G}(M)$ be the block versions, see \cite[p.~168]{Weiss-Williams(2001)} for a survey and further references. There are natural maps making the diagram
\[
\xymatrix{\TOP(M) \ar[d] \ar[r] 
&
\widetilde{\TOP}(M) \ar[d] 
\\
\G(M) \ar[r]
&
\widetilde{G}(M)
}
\]
commute.

Define $\widetilde{\TOP}(M)/\TOP(M)$, $\widetilde{G}(M)/\TOP(M)$ and  $\G(M)/\TOP(M)$ to be the homotopy fibers of the maps
$ B\!\TOP(M) \to B\widetilde{\TOP}(M)$, $B\!\TOP(M) \to B\widetilde{\G}(M)$ and $B\!\TOP(M) \to B\!\G(M)$. 
We obtain a commutative diagram with horizontal fiber sequences
\[
\xymatrix{
\widetilde{\TOP}(M)/\TOP(M) \ar[r] \ar[d]
&
B\!\TOP(M) \ar[r] \ar[d]^{\id} &
\ar[d] B\widetilde{\TOP}(M)
\\
\widetilde{\G}(M)/\TOP(M) \ar[r]
&
B\!\TOP(M) \ar[r]
&
B\widetilde{\G}(M)
\\
\G(M)/\TOP(M) \ar[r] \ar[u]
&
B\!\TOP(M) \ar[r] \ar[u]_{\id} 
&
B\!\G(M). \ar[u]
}
\]

According to \cite[Theorem~5.8]{Rourke-Sanderson(1968)}, there is no real difference between self homotopy equivalences and their block version.

\begin{lemma}\label{lem:.widetildeG_andG}
The map $G(M) \to \widetilde{G}(M)$ and hence the map $BG(M) \to B\widetilde{G}(M)$ are weak homotopy equivalences.
\end{lemma}

The relative homotopy groups of the map $\widetilde{\TOP}(M) \to \widetilde{\G}(M)$
can be identified with the groups $\cals^s(M \times D^n,\partial)$ as
explained in~\cite[page~285]{Farrell(2002)}. The next lemma follows in
combination with \cite[Proposition~0.3]{Bartels-Lueck(2012annals)}.  Recall
that a space $X$ is \emph{aspherical} if $\pi_i(X) = 0$ for $i \neq 1$.

\begin{lemma}\label{lem:widetildeTOP_andG}
Suppose that $M$ is an aspherical closed manifold of dimension $\ge 5$ and both the $K$-- and $L$--theoretic Farrell--Jones Conjecture hold for $\IZ \pi_1(M)$. 

Then $\cals^s(M \times D^n,\partial)$ is trivial for $n \ge 0$ and the map
\[
\widetilde{\TOP}(M)  \to \widetilde{\G}(M)
\]
is a weak homotopy equivalence.
\end{lemma}

For aspherical spaces $X$, the homotopy groups of $\G(X)$ can be computed from
the long exact sequence of homotopy groups associated to the evaluation map $\G(X) \xrightarrow{ev_{x_0}} X$ for some basepoint $x_0 \in X$:

\begin{lemma} \label{lem:pi_n(G(X)_X_aspherical}
Let $X$ be an aspherical CW--complex. Then
\[\pi_n(\G(X))
\cong
\begin{cases}
\Out(\pi_1(X)) & n = 0,
\\
\zentrum(\pi_1(X)) & n = 1,
\\
0 & n \ge 2.
\end{cases}
\]
\end{lemma}

We conclude from \Cref{lem:.widetildeG_andG}, \Cref{lem:widetildeTOP_andG}
and \Cref{lem:pi_n(G(X)_X_aspherical}:

\begin{corollary}\label{lem:pi_n(TOP(X)_X_aspherical}
If $M$ is an aspherical closed manifold of dimension $\ge 5$ with fundamental group $\pi$,
and both the $K$--theoretic and $L$--theoretic Farrell--Jones Conjecture hold for $\IZ \pi$,
then there are natural zig-zags of homotopy equivalences
\[
\widetilde{\TOP}(M)  \simeq \G(M)
\]
and
\[
B\!\widetilde{\TOP}(M)  \simeq B\!\G(M)
\]
and we get
\[
\pi_n(\widetilde{\TOP}(M)) 
\cong
\begin{cases}
\Out(\pi) & n = 0,
\\
\zentrum(\pi) & n = 1,
\\
0 & n \ge 2.
\end{cases}
\]
\end{corollary}

\begin{theorem}\label{the:widetilde(TOP)(M)/TOP(M)_and_bfH}
There is a map
\[
\widetilde{\TOP}(M)/\TOP(M) \to \Omega^{\infty} \bigl(E\IZ/2_+ \wedge_{\IZ/2} \Omega \bfWhs^{\TOP}(M)\bigr)
\]
which is $(k+1)$--connected if $k$ is in the topological concordance stable range for $M$. Here $ \bfWhs^{\TOP}(M)$ denotes the connective cover of the Whitehead spectrum $\bfWh^{\TOP}(M)$.
\end{theorem}
\begin{proof}
It suffices to show that the spectrum denoted by $\bfWh^{\TOP}(M)$ in \cite{Weiss-Williams(1988)} is a model for the homotopy cofiber of the assembly map, see \Cref{def: non-conn whspectrum}. This follows from combining~\cite[Theorem~A]{Weiss-Williams(1988)}, the equivalence $\bfP^{TOP}(M)\simeq \Omega \bfH^{TOP}(M)$ and \cite[Theorem~0.2]{Jahren-Rognes-Waldhausen(2013)}.
\end{proof}

We conclude from \Cref{the:relations_between_the_various_functors}~\ref{the:relations_between_the_various_functors:wh_TOP_and_wh_PL}, \Cref{lem:pi_n(TOP(X)_X_aspherical}, \Cref{the:widetilde(TOP)(M)/TOP(M)_and_bfH} and the lower bound on the topological concordance stable range given in ~\cite[Corollary~1.4.2]{Jahren-Rognes-Waldhausen(2013)}:

\begin{theorem}\label{the:widetilde(TOP)(M)/TOP(M)_and_bfWh}
Let $M$ be a smoothable aspherical closed manifold of dimension $\ge 10$  with fundamental group $\pi$. Suppose that 
the Farrell--Jones Conjecture for $A$--theory for $B\pi$  and the Farrell--Jones Conjecture for $L$--theory for $\IZ\pi$ hold.

Then we obtain for $2 \le n \le \min\{( \dim M -7 ) / 2, (\dim M - 4)/3\}$ isomorphisms
\[
\pi_{n}(\TOP(M)) \xrightarrow{\cong}  \pi_{n+2}\bigl(E\IZ/2_+ \wedge_{\IZ/2} \bfWh^{\TOP}(B\pi)\bigr)),
\]
and an exact sequence
\begin{multline*}
1 \to \pi_{3}\bigl(E\IZ/2_+ \wedge_{\IZ/2} \bfWh^{\TOP}(B\pi)\bigr)) \to\pi_1(\TOP(M)) \to \zentrum(\pi) 
\\
\to  \pi_{2}\bigl(E\IZ/2_+ \wedge_{\IZ/2} \bfWh^{\TOP}(B\pi)\bigr)) \to\pi_0(\TOP(M)) \to \Out(\pi) \to 1.
\end{multline*}
\end{theorem}

Next we give the proof of \Cref{the:widetilde(TOP)(M)/TOP(M)_and_bfN(A)}.

\begin{proof}[Proof of \Cref{the:widetilde(TOP)(M)/TOP(M)_and_bfN(A)}]
 Let $G$ be a torsionfree hyperbolic group.
 Then the Farrell--Jones Conjecture for $A$--theory for $BG$ and the Farrell--Jones Conjectures
for algebraic $K$--theory and for $L$--theory for $\IZ G$  hold  by
\Cref{the:main_result} and \cite{Bartels-Lueck(2012annals),Bartels-Lueck-Reich(2008hyper)}.

Since $G$ is torsionfree, we have $K_{-i}(\IZ G) = 0$ for all $i \ge 1$ and
$\widetilde{K}_0(\IZ G) = 0$, and thus the spectra under consideration are connective by
\Cref{the:Connectivity_of_the_linearization_map}~\ref{the:Connectivity_of_the_linearization_map:2_connected}. It follows from \Cref{lem:equivalence_of_FJ_for_the_various_functors} that there is a weak homotopy equivalence
\[\cH^G(\EGF{G}{\calvcyc};(\bfWh^{\TOP})^B) \xrightarrow{\simeq} \bfWh^{\TOP}(BG).
\]
The arguments in~\cite[Section~10] {Lueck-Steimle(2015splitasmb)} based 
on~\cite[Corollary~2.8 and Example~3.6]{Lueck-Weiermann(2012)} for algebraic $K$--theory carry over to $\bfWh^{\TOP}$ and 
imply that there is a weak homotopy equivalence induced by the various inclusions $C \to G$ of representatives of the conjugacy classes of maximal cyclic subgroups of $G$
\[
\bigvee_C  \bfWh^{\TOP}(BC)  \xrightarrow{\simeq}  \bfWh^{\TOP}(BG).
\]
From the Bass--Heller--Swan decomposition~\eqref{Bass-Heller_Swan_for_A-theory}  we obtain a weak homotopy equivalence 
\[
\bfNA(\pt) \vee \bfNA(\pt)  \xrightarrow{\simeq}  \bfWh^{\TOP}(BC).
\]
This proves part \ref{the:widetilde(TOP)(M)/TOP(M)_and_bfN(A)-it1} of \Cref{the:widetilde(TOP)(M)/TOP(M)_and_bfN(A)}.

Part \ref{the:widetilde(TOP)(M)/TOP(M)_and_bfN(A)-it2} follows from part \ref{the:widetilde(TOP)(M)/TOP(M)_and_bfN(A)-it1} and \Cref{the:widetilde(TOP)(M)/TOP(M)_and_bfWh} together with the fact that the center of a hyperbolic group which is torsionfree and not cyclic is trivial.
\end{proof}

\Cref{the:Relating_A-theory_to_algebraic_K_theory} and \Cref{the:widetilde(TOP)(M)/TOP(M)_and_bfWh} imply

\begin{theorem}[Rational homotopy groups of $\TOP(M)$ for an aspherical  closed manifold]
\label{the:pi_i(TOP(M)_M_aspherical}
Let $M$ be a smoothable aspherical closed manifold of dimension $\ge 10$ with fundamental group
$\pi$.  Suppose that the Farrell--Jones Conjecture for $K$--theory and for $L$--theory
for $\IZ\pi$ hold. 

Then for $1 \leq n\leq \min\{( \dim M -7 ) / 2, (\dim M - 4)/3\}$ we have
\[
\pi_n (\TOP(M)) \otimes_{\IZ} \IQ 
= \begin{cases} 
\zentrum(\pi) \otimes_{\IZ} \IQ  & \text{if}\; n = 1,
\\
\{0\} & \text{if} \; n \ge 2.
\end{cases}
\]
\end{theorem}


\subsection{Smooth automorphism groups of aspherical closed smooth manifolds}
\label{subsec:Smooth_automorphism_groups_of_aspherical_closed_smooth_manifolds}

Taking the computation of $K_i(\IZ) \otimes_{\IZ} \IQ$ of Borel~\cite{Borel(1972)} into account, we get
from \Cref{the:relations_between_the_various_functors},
\Cref{the:Relating_A-theory_to_algebraic_K_theory}
and~\cite[Theorem~0.3]{Lueck-Steimle(2015splitasmb)}

\begin{theorem} \label{the:computation_of_diff_pseudos} Let $M$ be an aspherical closed
  smooth manifold of dimension $\ge 10$ with fundamental group $\pi$. Suppose that the
  Farrell--Jones Conjecture for $K$--theory and for $L$--theory for $\IZ\pi$ hold.

Then we get for all $n \in \IZ$
  \begin{eqnarray*}
  \pi_n (\bfWh^{\DIFF} ( M ) )\otimes_{\IZ} \IQ  
  & \cong & 
\bigoplus_{k = 1}^{\infty}  H_{n-4k-1} (M ; \IQ ),
  \\
  \pi_n ( \bfP^{\DIFF} ( M ) )\otimes_{\IZ} \IQ  
  & \cong & 
\bigoplus_{k = 1}^{\infty}  H_{n-4k+1} (M ; \IQ ).
\end{eqnarray*}
\end{theorem}

For the proof of the next result, which does involve the involutions on higher algebraic $K$--theory,
  we refer to~\cite[Lecture~5]{Farrell(2002)},~\cite{Farrell-Hsiang(1978)},  
or~\cite[Section~2]{Farrell-Jones(1990b)}. 

\begin{theorem}[Rational homotopy groups of $\DIFF(M)$ for an aspherical  closed smooth manifold]
\label{the:pi_i(DIFF(M)_M_aspherical}
Let $M$ be  an aspherical closed smooth manifold of dimension $\ge 10$ with fundamental group $\pi$. 
Suppose that the Farrell--Jones Conjecture for $K$--theory and for $L$--theory
for $\IZ\pi$ hold.

Then for $1 \leq n \leq \min\{( \dim M -7 ) / 2, (\dim M - 4)/3\}$ we have
\[
\pi_n (\DIFF(M)) \otimes_{\IZ} \IQ = 
\begin{cases} 
\zentrum(\pi) \otimes_{\IZ} \IQ  &  \mbox{if} \; n=1, 
\\
\bigoplus_{j=1}^{\infty} H_{(n +1) - 4j} ( M;\IQ ) &   \mbox{if} \; n \ge 2 , \; \dim M \; \mbox{odd}, 
\\
\{0\} & \mbox{if} \; n \ge 2, \; \dim M \mbox{even}.
\end{cases}
\]
\end{theorem}

\begin{remark}[Surfaces and simply connected manifolds] \label{rem:other_manifolds} There
  are very interesting computations of the cohomology of $B\!\DIFF(M)$ in a range and under
  stabilization with taking the connected sum with $S^n \times S^n$ for $2$--dimensional
  manifolds or simply connected high-dimensional manifolds by Berglund, Galatius, Madsen,
  Randal-Williams, Weiss and others, see for
  instance~\cite{Berglund-Madsen(2013),Berglund-Madsen(2014),Galatius-Randal-Williams(2014abelian),
   Galatius-Randal-Williams(2014Cont), Galatius-Randal-Williams(2014Acta), Galatius-Randal-Williams(2014stability_I),
    Galatius-Randal-Williams(2016stability_II), Madsen-Weiss(2005),Madsen-Weiss(2007)}.
  The methods used in these papers are quite different. Notice that taking the connected
  sum with $S^n \times S^n$ will destroy asphericity except for $n = 1$, so that it is not
  clear what stabilization could mean in the context of aspherical manifolds in high
  dimensions.
\end{remark}


\section{Inheritance properties of the  Isomorphism Conjectures}
\label{sec:Inheritance_properties_of_the_Isomorphism_Conjectures}

The main result of this section is

\begin{theorem}[Inheritance properties of the Meta-Conjecture with coefficients]
\label{the:Inheritance_properties_of_the_Meta-Isomorphism_Conjecture_for_bfS_with_c}
Let $\bfS \colon \SPACES\to \SPECTRA$ be a covariant functor which respects weak
equivalences and disjoint unions. Let $\calc$ be  a class of groups which is closed under
isomorphisms, taking subgroups and taking quotients.

\begin{enumerate} 

\item \label{the:Inheritance_properties_of_the_Meta-Isomorphism_Conjecture_for_bfS_with_c:subgroups}
  Suppose that the Meta-Isomorphism
 Conjecture~\ref{con:Meta_Isomorphisms_Conjecture_for_functors_from_spaces_to_spectra_with_c}
 with  coefficients holds for $(G,\calc(G))$, i.e., it holds for $G$ with respect to the family of subgroups 
   $\calc(G) = \{H \subseteq G \mid H \in \calc\}$ of $G$.  Let $H \subseteq G$ be a subgroup.

Then \Cref{con:Meta_Isomorphisms_Conjecture_for_functors_from_spaces_to_spectra_with_c}
 holds for $(H,\calc(H))$.

\item \label{the:Inheritance_properties_of_the_Meta-Isomorphism_Conjecture_for_bfS_with_c:extensions}
Let $1 \to K \to   G \xrightarrow{p} Q \to 1$ be an extension of
groups. Suppose that $(Q,\calc(Q))$ and $(p^{-1}(H), \calc(p^{-1}(H))$ for every $H \in \calc(Q)$ satisfy 
\Cref{con:Meta_Isomorphisms_Conjecture_for_functors_from_spaces_to_spectra_with_c}.
Then $(G,\calc(G))$ satisfies 
\Cref{con:Meta_Isomorphisms_Conjecture_for_functors_from_spaces_to_spectra_with_c}.

\item \label{the:Inheritance_properties_of_the_Meta-Isomorphism_Conjecture_for_bfS_with_c:direct_products}
Suppose that \Cref{con:Meta_Isomorphisms_Conjecture_for_functors_from_spaces_to_spectra_with_c}
is true for $(H_1 \times H_2,\calc(H_1 \times H_2))$ for every $H_1, H_2 \in \calc$.

Then for two groups $G_1$ and $G_2$
\Cref{con:Meta_Isomorphisms_Conjecture_for_functors_from_spaces_to_spectra_with_c}
is true for the direct product $G_1 \times G_2$ and the family $\calc(G_1 \times G_2)$,
if and only if is true for $(G_k,\calc(G_k))$ for $k = 1,2$.

\item \label{the:Inheritance_properties_of_the_Meta-Isomorphism_Conjecture_for_bfS_with_c:directed_colimits}
Suppose that for any directed systems of spaces $\{X_i \mid i \in I\}$ indexed over an arbitrary
  directed set $I$ the canonical map
  \[
  \hocolim_{i \in I} \bfS(X_i) \to \bfS\bigl(\hocolim_{i \in I} X_i\bigr)
  \]
  is a weak homotopy equivalence. 
Let $\{G_i \mid i  \in I\}$ be a directed system of groups over a directed set $I$
  (with arbitrary structure maps). Put $G = \colim_{i \in I} G_i$.
Suppose that  \Cref{con:Meta_Isomorphisms_Conjecture_for_functors_from_spaces_to_spectra_with_c}
holds for $(G_i,\calc(G_i))$  for every $i \in I$. Then \Cref{con:Meta_Isomorphisms_Conjecture_for_functors_from_spaces_to_spectra_with_c}
holds for $(G,\calc(G))$.

\item \label{the:Inheritance_properties_of_the_Meta-Isomorphism_Conjecture_for_bfS_with_c:overgroups_of_finite_index}
The analogs of assertions~\ref{the:Inheritance_properties_of_the_Meta-Isomorphism_Conjecture_for_bfS_with_c:subgroups},~%
\ref{the:Inheritance_properties_of_the_Meta-Isomorphism_Conjecture_for_bfS_with_c:extensions},~%
\ref{the:Inheritance_properties_of_the_Meta-Isomorphism_Conjecture_for_bfS_with_c:direct_products}, 
and~\ref{the:Inheritance_properties_of_the_Meta-Isomorphism_Conjecture_for_bfS_with_c:directed_colimits}
hold for the Meta-Isomorphism 
Conjecture~\ref{con:The_Meta_Isomorphisms_Conjecture_for_functors_from_spaces_to_spectra_with_c_and_fwp}
with coefficients and finite wreath products.

Moreover, if $G$ is  a group and $H \subseteq G$ is  a subgroup of finite index, then 
\Cref{con:The_Meta_Isomorphisms_Conjecture_for_functors_from_spaces_to_spectra_with_c_and_fwp}
holds for $(G,\calc(G))$, if and only if 
\Cref{con:The_Meta_Isomorphisms_Conjecture_for_functors_from_spaces_to_spectra_with_c_and_fwp}
holds for $(H,\calc(H))$.
\end{enumerate}
\end{theorem}

Let us remark that the case of free products is missing in
\Cref{the:Inheritance_properties_of_the_Meta-Isomorphism_Conjecture_for_bfS_with_c}.
It will be treated in
\Cref{subsec:Proof_of_assertion_ref(the:main_result:inheritance)_of_Theorem_ref(the:main_result)}
below.


\subsection{The Fibered Meta-Isomorphism Conjecture for equivariant homology theories}
\label{subsec:The_Fibered_Meta_Isomorphism_Conjecture_for_equivariant_homology_theories}

Next we introduce the Meta-Conjecture and its fibered version in terms
of $G$--homology theories.  In this setting the analog of
\Cref{the:Inheritance_properties_of_the_Meta-Isomorphism_Conjecture_for_bfS_with_c}
has already been proved and we want to reduce the case coming from a
functor from spaces to spectra to this situation.

\begin{conjecture}[Meta-Isomorphism Conjecture] \label{con:Meta_Isomorphisms_Conjecture}
The group $G$ satisfies the \emph{Meta-Isomorphism Conjecture}
\index{Conjecture!Meta-Isomorphism Conjecture}
with respect to the $G$--homology theory $\calh^G_*$ and the family $\calf$ of subgroups of $G$
if the assembly map 
\[
\calh_n^G(\pr) \colon \calh^G_n(\EGF{G}{\calf}) \to \calh_n^G(G/G)
\]
induced by the projection $\pr \colon \EGF{G}{\calf} \to G/G$ is bijective for all $n \in \IZ$.
\end{conjecture}

Let $X$ be a $G$--CW--complex. Let $\alpha \colon H \to G$ be a group homomorphism.
Denote by $\alpha^*X$ the $H$--CW--complex obtained from $X$ by
\emph{restriction with $\alpha$}. Given an $H$--CW--complex $Y$, we denote the
$G$--CW--complex given by \emph{induction} by $\alpha_*Y$.

Fix a group $\Gamma$. An equivariant homology theory $H_*^{?}$ over $\Gamma$ in the sense of
\cite[Definition~2.3]{Bartels-Echterhoff-Lueck(2008colim)} assigns to a group $(G,\psi)$
over $\Gamma$, i.e., a group $G$ together with a homomorphism $\psi \colon G \to \Gamma $,
a $G$--homology theory $H_n^{G,\psi}$, sometimes denoted just by $H_*^G$. For two groups
$(G,\psi)$ and $(G',\psi')$ over $\Gamma$ and a morphism $\phi$ between them,
i.e., a
group homomorphism $\phi\colon G \to G'$ with $\psi' \circ \alpha = \psi$, one obtains
homomorphisms $\ind_{\alpha} \colon H^G_*(X,A) \to H^{G'}_*(\alpha_*(X,A))$ for every
$G$--CW--pair $(X,A)$,  which are bijective, if the kernel of $\alpha$ acts freely on
$(X,A)$, and compatible with the boundary homomorphisms associated to pairs. If $\Gamma$ is
trivial, this is just an equivariant homology theory.

\begin{conjecture}[Fibered Meta-Isomorphism Conjecture]
  \label{con:Fibered_Meta-Isomorphism_Conjectures_for_calh?_ast}
   A group $(G,\psi)$ over $\Gamma$
  satisfies the \emph{Fibered Meta-Isomorphism Conjecture with respect to $\calh^?_*$ and the
    family $\calf$ of subgroups of $G$} if for each group homomorphism 
  $\phi \colon K \to  G$ the group $K$ satisfies the Meta-Isomorphism
  Conjecture~\ref{con:Meta_Isomorphisms_Conjecture} with respect to the $K$--homology
  theory $\calh^{K,\psi \circ \phi}_*$ and the family $\phi^*\calf = \{H \subseteq G \mid \phi(H) \in \calf\}$ of subgroups of $K$.
\end{conjecture}

\begin{lemma} \label{lem:basic_inheritance_property_of_fibered_Isomorphism_conjecture}
Let $(G,\psi)$ be a group over $\Gamma$ and  $\phi \colon K \to G$ be a group
homomorphism.  If  $(G,\psi)$ satisfies
the Fibered Meta-Isomorphism Conjecture~\ref{con:Fibered_Meta-Isomorphism_Conjectures_for_calh?_ast}
with respect to the family $\calf$ of subgroups of $G$, then  the group  $(K,\psi \circ \phi)$ over $\Gamma$ satisfies
the Fibered Meta-Isomorphism Conjecture~\ref{con:Fibered_Meta-Isomorphism_Conjectures_for_calh?_ast}
with respect to the family $\phi^*\calf$.
\end{lemma}
\begin{proof}
If $\theta \colon L \to K$ is a group homomorphism, then
$\theta^*(\phi^*\calf) = (\phi \circ \theta)^*\calf$.
\end{proof}


\subsection{Some adjunctions}
\label{subsec:Some_adjunctions}

Let $\bfS \colon \SPACES\to \SPECTRA$ be a covariant functor. Throughout this section we
will assume that it respects weak equivalences and disjoint unions.

\begin{lemma} \label{lem:adjunctions_for_homology_associated_toS_upper_K}
Let $\psi \colon K_1 \to K_2$ be a group homomorphism.

\begin{enumerate}

\item \label{lem:adjunctions_for_homology_associated_toS_upper_K:restriction_of_X}
If $Z$ is a $K_1$--CW--complex and $X$ is a $K_2$--CW--complex, then there is a natural isomorphism
\[
H_n^{K_1}(\psi^*X;\bfS^{K_1}_Z) \xrightarrow{\cong} H_n^{K_2}(X;\bfS^{K_2}_{\psi_*Z}).
\]

\item \label{lem:adjunctions_for_homology_associated_toS_upper_K:induction_of_X}

If $Z$ is a $K_2$--CW--complex and $X$ is a $K_1$--CW--complex, then there is a natural isomorphism
\[
H_n^{K_1}(X;\bfS^{K_1}_{\psi^*Z}) \xrightarrow{\cong} H_n^{K_2}(\psi_*X;\bfS^{K_2}_{Z}).
\]

\end{enumerate}
\end{lemma}
\begin{proof}~\ref{lem:adjunctions_for_homology_associated_toS_upper_K:restriction_of_X}
The fourth isomorphism appearing in~\cite[Lemma~1.9]{Davis-Lueck(1998)}
together 
with~\cite[Lemma~4.6]{Davis-Lueck(1998)} applied levelwise implies that it suffices to construct a
natural weak homotopy equivalence of $\Or(K_2)$--spectra
\[
u(\psi,Z) \colon \psi_* \bfS^{K_1}_Z \xrightarrow{\simeq} \bfS^{K_2}_{\psi_* Z},
\]
where  $\psi_* \bfS^{K_1}_Z$ is the $\Or(K_2)$--spectrum obtained by induction
in the sense of~\cite[Definition~1.8]{Davis-Lueck(1998)} with the functor
$\Or(\psi) \colon \Or(K_1) \to \Or(K_2),\; K_1/H_1 \mapsto \psi_* (K_1/H_1)$
applied to the $\Or(K_1)$--spectrum $\bfS^{K_1}_Z$.  For a homogeneous space
$K_2/H$ we define $u(\psi,Z)(K_2/H)$ to be the composite
\begin{eqnarray*}
\psi_* \bfS^{K_1}_Z(K_2/H) 
& = & 
\map_{K_2}(\psi_*(K_1/?),K_2/H)_+ \wedge_{\Or(K_1)} \bfS\left(K_1/? \times_{K_1} Z\right)
\\
& \xrightarrow{\cong} &
\map_{K_1}((K_1/?),\psi^*(K_2/H))_+ \wedge_{\Or(K_1)} \bfS(K_1/? \times_{K_1} Z)
\\
& \xrightarrow{\simeq} &
\bfS(\psi^*(K_2/H) \times_{K_1} Z)
\\
& \xrightarrow{\cong} &
\bfS(K_2/H\times_{K_2} \psi_*Z)
\\
& =: &
\bfS^{K_2}_{\psi_* Z}(K_2/H).
\end{eqnarray*}
Here the first map comes from the adjunction isomorphism
\[
\map_{K_2}(\psi_*(K_1/?),K_2/H) \xrightarrow{\cong} \map_{K_1}(K_1/?),\psi^*(K_2/H)),
\]
and the third map comes from the canonical homeomorphism
\[
\psi^*(K_2/H) \times_{K_1} Z \xrightarrow{\cong} K_2/H\times_{K_2} \psi_*Z.
\]
The second map is the special case $T = \psi^*K_2/{?}$ of the natural weak homotopy equivalence defined for
any $K_1$--set $T$
\[
\kappa(T) \colon \map_{K_1}((K_1/?),T)_+ \wedge_{\Or(K_1)} \bfS\left(K_1/? \times_{K_1} Z\right)
\xrightarrow{\simeq}
\bfS(T\times_{K_1} Z),
\]
which is given by $(u \colon K_1/? \to T) \times s \mapsto \bfS(u \times_{K_1}
\id_Z)(s)$.  If $T$ is a transitive $K_1$--set, then $\kappa(T)$ is even an
isomorphism by the Yoneda Lemma.  The
left-hand side is compatible with disjoint unions in $T$, the right-hand side
is compatible with disjoint unions in $T$ up to homotopy, where we use that
$\bfS$ respects disjoint unions.
As every $K_1$--set is the disjoint union of
homogeneous $K_1$--sets, $\kappa(T)$ is a weak homotopy equivalence for every  $K_1$--set $T$.
\\[1mm]~\ref{lem:adjunctions_for_homology_associated_toS_upper_K:induction_of_X}
The third isomorphism appearing in~\cite[Lemma~1.9]{Davis-Lueck(1998)} together 
with~\cite[Lemma~4.6]{Davis-Lueck(1998)} implies that it suffices to construct a
natural weak homotopy equivalence of $\Or(K_1)$--spectra
\[
v(\psi,Z) \colon \psi^* \bfS^{K_2}_Z \xrightarrow{\simeq} \bfS^{K_1}_{\psi^* Z},
\]
where  $ \psi^* \bfS^{K_2}_Z$ is the $\Or(K_1)$--spectrum obtained by restriction in the
sense of~\cite[Definition~1.8]{Davis-Lueck(1998)} with the functor $\Or(\psi) \colon
\Or(K_1) \to \Or(K_2),\; K_1/H \mapsto \psi_* (K_1/H)$ applied to the $\Or(K_2)$--spectrum
$\bfS^{K_2}_Z$. Actually, we obtain even an isomorphism $v(\psi,Z)$ using the adjunction
\[
\psi_*(K_1/H) \times_{K_2} Z \cong K_1/H \times_{K_1} \psi^* Z
\] 
for any subgroup $H \subseteq K_1$.
\end{proof}


\subsection{The Fibered Meta--Isomorphism Conjecture with coefficients for functors from spaces to spectra}
\label{subsec:The_Fibered_Meta_Isomorphism_Conjecture_with_coefficients_for_functors_from_spaces_to_spectra}

Notice that for a homomorphism $\phi \colon H \to G$ the restriction $\phi^* Z$ of a free
$G$--CW--complex $Z$ is free again if and only if $\phi$ is injective. We have already
explained in
Remark~\ref{rem:The_condition_free_is_necessary_in_MC_spaces_spectra_with_coeff} that the
assumption that $Z$ is free is needed in
\Cref{con:Meta_Isomorphisms_Conjecture_for_functors_from_spaces_to_spectra_with_c}.
In the Fibered Meta-Isomorphism Conjecture~\ref{con:Fibered_Meta-Isomorphism_Conjectures_for_calh?_ast} 
 it is crucial not to require that $\phi \colon H \to G$ is injective since we want to have good inheritance properties.
Therefore we have to blow up $Z$ everywhere by passing to $EG \times Z$ as explained below.

Let $G$ be a group and $Z$ be a $G$--CW--complex.  Recall that $\underline{G}$
denotes the groupoid with precisely one object which has $G$ as its
automorphism group. Let $\GROUPOIDS \downarrow G$ be the category of groupoids
over $\underline{G}$. Objects are groupoids $\calg$ together with a functor $P
\colon \calg \to \underline{G}$.  A morphism from $P \colon \calg \to
\underline{G}$ to $P'\colon \calg' \to \underline{G}$ is a covariant functor
$F \colon \calg \to \calg'$ satisfying $P' \circ F = P$. Given a groupoid
$\calg$, we obtain a contravariant functor $E(? \downarrow \calg) \colon \calg
\to \SPACES$ by sending an object $x$ to the classifying space of the category
$x \downarrow \calg$ of objects in $\calg$ under $x$. We get from $Z$, by
restriction along $P$, a covariant functor $P^*Z \colon \calg \to \SPACES$
where we think of the left $G$--space  $Z$ as a covariant functor
$\underline{G} \to \SPACES$. The tensor product over $\calg$,
see~\cite[Section~1]{Davis-Lueck(1998)}, yields a space $E(? \downarrow \calg)
\times_{\calg}P^*Z(?)$.  Thus we obtain a covariant functor
\begin{equation}
  \bfS^{\downarrow G}_Z \colon \GROUPOIDS \downarrow G \to \SPECTRA, 
  \quad P \colon (\calg \to \underline{G}) \mapsto \bfS(E(? \downarrow \calg)
  \times_{\calg} P^*Z(?)).
  \label{bfS(G,B)_Z}
\end{equation}
It yields an equivariant homology theory $H_n^?(-;\bfS_Z^{\downarrow G})$ over
$G$, see~\cite[Lemma~7.1]{Bartels-Echterhoff-Lueck(2008colim)}. 
Given a homomorphism $\psi \colon K \to G$ we get an identification of $K$--homology theories
\begin{equation}
H^{K,\psi}_*(-;\bfS^{\downarrow G}_Z) \cong H_*^K(-;\bfS^K_{EK \times \psi^* Z}),
\label{identification_of_K-homology_theories}
\end{equation}
which is induced by a homotopy equivalence, natural in
$K/H$,
\begin{equation*}
E(? \downarrow \calt^K(K/H)) \times_{\calt^K(K/H)}
\psi^*Z(?)\xrightarrow{\simeq} K/H
\times_K(EK \times  \psi^*Z)
\end{equation*}
and~\cite[Lemma~4.6]{Davis-Lueck(1998)}, where $\calt$ denotes the transport
groupoid from \Cref{exa:Z_is_EG} and $\psi$ also denotes its
induced map $\calt^K(K/H) \to \underline{G}$. For any group $\psi \colon K \to G$
over $G$, inclusion $i \colon H \to K$ of a subgroup $H$ of $K$, and $n \in
\IZ$ we have canonical identifications
\begin{equation*}
H_n^{K,\psi}(K/H;\bfS_Z^{\downarrow G})) \xrightarrow{\cong} H_n^{H, \psi
\circ i}(H/H;\bfS_Z^{\downarrow G})) 
\cong \pi_n(\bfS(EH \times_H (\psi \circ i)^*Z)).
\end{equation*}

\begin{lemma} \label{lem:HG_ast(phiastY;bfSdownarrow_G_Z)} 
Let $\phi \colon H \to K$ and $\psi \colon K \to G$ be group homomorphisms.

\begin{enumerate}
\item \label{lem:HG_ast(phiastY;bfSdownarrow_G_Z):(1)} 
Let $X$ be a $G$--CW--complex and  let $Z$ be a $K$--CW--complex. Then we obtain a natural isomorphism
\[
H_n^{H,\phi}(\phi^*\psi^*X;\bfS^{\downarrow K}_{Z}) 
\xrightarrow{\cong}
H_n^{G}(X;\bfS^G_{(\psi \circ \phi)_*(EH \times \phi^*Z)}).
\]

\item \label{lem:HG_ast(phiastY;bfSdownarrow_G_Z):(2)} 
Let $X$ be an $H$--CW--complex and  let $Z$ be a $G$--CW--complex. 
Then we obtain a natural isomorphism
\[
H_n^{H,\phi}(X;\bfS^{\downarrow K}_{\psi^* Z}) 
\xrightarrow{\cong}
H_n^{H,\psi \circ \phi}\bigl(X;\bfS^{\downarrow G}_Z\bigr).
\]

\end{enumerate}
\end{lemma}
\begin{proof}~\ref{lem:HG_ast(phiastY;bfSdownarrow_G_Z):(1)} 
We get from~\eqref{identification_of_K-homology_theories}
\[
H_n^{H,\phi}(\phi^*\psi^*X;\bfS^{\downarrow K}_{Z}) := H_n^H(\phi^*\psi^*X;\bfS^H_{EH \times \phi^*Z}).
\]
Now apply \Cref{lem:adjunctions_for_homology_associated_toS_upper_K}~%
\ref{lem:adjunctions_for_homology_associated_toS_upper_K:restriction_of_X}.
\\[1mm]~\ref{lem:HG_ast(phiastY;bfSdownarrow_G_Z):(2)} 
We get from~\eqref{identification_of_K-homology_theories}
\begin{multline*}
H_n^{H,\phi}(X;\bfS^{\downarrow K}_{\psi^* Z}) 
:= H_n^H(X;\bfS^H_{EH \times \phi^* \psi^* Z}) 
\\
= H_n^H(X;\bfS^H_{EH \times (\psi \circ \phi)^* Z}) 
=: H_n^{H,\psi \circ \phi}\bigl(X;\bfS^{\downarrow G}_Z\bigr). 
\end{multline*}
\end{proof}

\begin{conjecture}[Fibered Meta-Isomorphism Conjecture for a functor from spaces to spectra with
  coefficients]
  \label{con:Fibered_Meta_Conjecture_for_a_functor_from_spaces_to_spectra_with_c}
  \index{Conjecture!Fibered Meta Isomorphisms Conjecture for a functor from spaces to spectra with
    coefficients}  Let $\bfS \colon \SPACES\to \SPECTRA$, as before, respect
    weak equivalences and disjoint unions. We say that $\bfS$ satisfies the Fibered Meta-Isomorphism Conjecture for a functor
  from spaces to spectra with coefficients for the group $G$ and the family of subgroups
  $\calf$ of $G$ if the following holds:  For any $G$--CW--complex $Z$ the equivariant homology theory
  $H_*^?(-;\bfS_Z^{\downarrow G})$ over $G$ satisfies the Fibered Meta-Isomorphism
  Conjecture~\ref{con:Fibered_Meta-Isomorphism_Conjectures_for_calh?_ast} for the group
  $(G,\id_G)$ over $G$ and the family $\calf$.
\end{conjecture}

Note that
\Cref{con:Meta_Isomorphisms_Conjecture_for_functors_from_spaces_to_spectra_with_c}
deals with the $G$--homology theory $H^G_*(-;\bfS^G)$, whereas
\Cref{con:Fibered_Meta_Conjecture_for_a_functor_from_spaces_to_spectra_with_c}
deals with the the equivariant homology theory $H^?_*(-;\bfS^{\downarrow G})$ over
$G$.  Moreover, 
\Cref{con:Fibered_Meta_Conjecture_for_a_functor_from_spaces_to_spectra_with_c}
is unchanged if we would additionally require that the $G$--CW--complex $Z$ is
free. Namely, for any $G$--CW--complex $Z$ the $G$--CW--complex $EG \times Z$ is free and
the projection $EG \times Z \to Z$ induces an isomorphism $H_*^?(-;\bfS_{EG \times
  Z}^{\downarrow G}) \xrightarrow{\cong} H_*^?(-;\bfS_Z^{\downarrow G})$ of equivariant
homology theories over $G$ because of~\eqref{identification_of_K-homology_theories}
and~\cite[Lemma~4.6]{Davis-Lueck(1998)}.

\newcommand{\ConjMICforS}
{C\ref{con:Meta_Isomorphisms_Conjecture_for_functors_from_spaces_to_spectra_with_c}\xspace}
\newcommand{\ConjMIC}{MIC\ref{con:Meta_Isomorphisms_Conjecture}\xspace}
\newcommand{\ConjFMIC}{FMIC\ref{con:Fibered_Meta-Isomorphism_Conjectures_for_calh?_ast}\xspace}
\newcommand{\ConjFMICforS}{S\ref{con:Fibered_Meta_Conjecture_for_a_functor_from_spaces_to_spectra_with_c}\xspace}

For the rest of this section, we abbreviate the different conjectures as
follows:
\begin{itemize}
  \item
    \ConjMICforS is the
    Meta-Isomorphism~\Cref{con:Meta_Isomorphisms_Conjecture_for_functors_from_spaces_to_spectra_with_c}
    for functors from spaces to spectra with coefficients.  This is the
    conjecture we want to know about in the end.
  \item \ConjMIC and
  \ConjFMIC denote the
    Meta-Isomorphism \Cref{con:Meta_Isomorphisms_Conjecture}, and
    the Fibered Meta-Isomophism
    \Cref{con:Fibered_Meta-Isomorphism_Conjectures_for_calh?_ast}.  These are
    statements about a ($G$-)equivariant homology theory.
  \item \ConjFMICforS
    denotes the Fibered Meta-Isomorphism
    \Cref{con:Fibered_Meta_Conjecture_for_a_functor_from_spaces_to_spectra_with_c}
    for a functor from spaces to spectra with coefficients. This takes as
    input a functor $\bfS$ and is the most general version of a conjecture we
    are interested it.
\end{itemize}

\begin{lemma} \label{lem:MC_with_coeff_and_FMC_with_coeff} 
  Let $\psi \colon K \to G$ be a group homomorphism.
  \begin{enumerate}
    \item \label{lem:MC_with_coeff_and_FMC_with_coeff:MC_implies_FC}
    Suppose that \ConjMICforS 
    holds for the group $G$ and the family $\calf$. Then \ConjFMICforS 
    holds for the group $K$ and the family $\psi^* \calf$.
  \item \label{lem:MC_with_coeff_and_FMC_with_coeff:FC_implies_MC} If 
    \ConjFMICforS 
    holds for the group $G$ and the family $\calf$, then 
    \ConjMICforS
    holds for the group $G$ and the family $\calf$.
  \item \label{lem:MC_with_coeff_and_FMC_with_coeff_up_from_G_to_K} 
    Suppose that 
    \ConjFMICforS 
    holds for the group $K$ and the family $\calf$. Then for every $G$--CW--complex
    $Z$ \ConjFMIC
   holds for the equivariant homology theory $H_n(-;\bfS^{\downarrow G}_Z)$ over $G$ for the group
    $(K,\psi)$ over $G$ and the family $\calf$ of subgroups of $K$.
  \end{enumerate}

\end{lemma}
\begin{proof}~\ref{lem:MC_with_coeff_and_FMC_with_coeff:MC_implies_FC} This follows from
  \Cref{lem:HG_ast(phiastY;bfSdownarrow_G_Z)}~\ref{lem:HG_ast(phiastY;bfSdownarrow_G_Z):(1)},
  since in the notation used there we have $\phi^* \psi^*\EGF{G}{\calf} =   \phi^*\EGF{K}{\psi^*\calf}$ and 
   $\phi^* \psi^* G/G = H/H$, and $(\psi \circ \phi)_*(EH \times   \phi^*Z)$ is a free $G$--CW--complex.
  \\[1mm]~\ref{lem:MC_with_coeff_and_FMC_with_coeff:FC_implies_MC} This follows from 
  applying \Cref{con:Fibered_Meta_Conjecture_for_a_functor_from_spaces_to_spectra_with_c} to the special case
  $\psi= \id_G$ and the fact that for a free $G$--CW--complex $Z$ the projection $EG \times Z \to Z$ is a
  $G$--homotopy equivalence and hence we get from~\eqref{identification_of_K-homology_theories}
  and~\cite[Lemma~4.6]{Davis-Lueck(1998)} natural isomorphisms
   \[
  H^{G,\id_G}_n(X;\bfS^{\downarrow G}_Z) \cong H^G_n(X;\bfS^G_{EG \times Z})   \cong H^G_n(X;\bfS^G_Z)
  \]
  for every $G$--CW--complex $X$ and $n \in \IZ$.
  \\[1mm]~\ref{lem:MC_with_coeff_and_FMC_with_coeff_up_from_G_to_K} This follows from
  \Cref{lem:HG_ast(phiastY;bfSdownarrow_G_Z)}~\ref{lem:HG_ast(phiastY;bfSdownarrow_G_Z):(2)}.
\end{proof}


\subsection{Strongly continuous equivariant homology theories over a group}
\label{subsec:Strongly_continuous_equivariant_homology_theories_over_a_group}

Fix a group $\Gamma$ and an equivariant homology theory $\calh^?_*$ over
$\Gamma$.

Let $X$ be a $G$--CW--complex and let $\alpha \colon H \to G$ be a group homomorphism.
The functors ${\alpha}_*\colon H{-}CW \rightleftarrows G{-}CW \colon\alpha^*$ are adjoint to one another.  In particular,
the adjoint of the identity on $\alpha^* X$ is a natural $G$--map
\begin{equation}
f(X,\alpha) \colon {\alpha}_*\alpha^*X \to X, \quad (g,x) \mapsto gx.
\end{equation}
Consider a map $\alpha \colon (H,\xi) \to (G,\mu)$ of groups over $\Gamma$.  Define the
$\Lambda$--map
\begin{eqnarray*}
& a_n = a_n(X,\alpha)\colon \calh_n^{H}(\alpha^*X)
\xrightarrow{\ind_{\alpha}} \calh_n^G({\alpha}_*\alpha^*X)
\xrightarrow{\calh_n^G(f(X,\alpha))}
\calh_n^G(X).&
\end{eqnarray*}
If $\beta \colon (G,\mu) \to (K,\nu)$ is another morphism of groups over $\Gamma$, then by
the axioms of an induction structure, see~\cite{Lueck(2002b)}, the composite $\calh_n^{H}(\alpha^*\beta^*X)
\xrightarrow{a_n(\beta^*X,\alpha)} \calh_n^{G}(\beta^*X) \xrightarrow{a_n(X,\beta)}
\calh_n^{K}(X)$ agrees with $a_n(X,\beta \circ \alpha) \colon
\calh_n^{H}(\alpha^*\beta^*X) = \calh_n^{H}((\beta\circ \alpha)^*X) \to \calh_n^{K}(X)$
for a $K$--CW--complex $X$.

Consider a directed system of groups $\{G_i \mid i \in I\}$ with $G = \colim_{i \in I}
G_i$ and structure maps $\psi_i \colon G_i \to G$ for $i \in I$ and $\phi_{i,j} \colon G_i
\to G_j$ for $i,j \in I, i \le j$.  We obtain for every $G$--CW--complex $X$ a
system 
$a_n(\psi_j^*X,\phi_{i,j}) \colon \calh^{G_i}(\psi_i^*X) \to \calh^{G_j}(\psi_j^*X)$.  We
get a map 
\begin{equation}
t_n^G(X)~:=~\colim_{i \in I} a_n(X,\psi_i) \colon \colim_{i \in  I} 
\calh_n^{G_i}(\psi_i^*(X))~\to~\calh_n^G(X).
\label{t_nG(X)}
\end{equation}

The next definition is taken from~\cite[Definition~3.3]{Bartels-Echterhoff-Lueck(2008colim)}.

\begin{definition}[Strongly continuous equivariant homology theory over a group]
\label{def:strongly_continuous_equivariant_homology_theory_over_a-group}
An equivariant homology theory $\calh^?_*$ over the group $\Gamma$ is called \emph{strongly
  continuous} if for every group $(G,\xi)$ over $\Gamma$ and every directed system of
groups $\{G_i \mid i \in I\}$ with $G = \colim_{i \in I} G_i$ and structure maps $\psi_i
\colon G_i \to G$ for $i \in I$ the 
map
\[
t^G_n(\pt) \colon \colim_{i \in I} \calh^{G_i}_n(\pt) \to
\calh^G_n(\pt)
\]
is an isomorphism for every $n \in \IZ$.
\end{definition}

\begin{lemma} \label{lem:bfS_and_homotopy_colimits}  Suppose
  that for any directed system of spaces $\{X_i \mid i \in I\}$ indexed over an arbitrary
  directed set $I$ the canonical map
  \[
  \hocolim_{i \in I} \bfS(X_i) \to \bfS\bigl(\hocolim_{i \in I} X_i\bigr)
  \]
  is a weak homotopy equivalence. 

  Then for every group $\Gamma$ and $\Gamma$--CW--complex $Z$ the equivariant homology theory over
  $\Gamma$ given by $H^?_*(-;\bfS^{\downarrow \Gamma}_Z)$ is strongly continuous.
\end{lemma}
\begin{proof} We only treat the case $\Gamma = G$ and $\psi = \id_G$, the general case of a group 
$\psi \colon G \to \Gamma$ over $\Gamma$ is completely analogous.
Consider a directed system of
groups $\{G_i \mid i \in I\}$ with $G = \colim_{i \in I} G_i$. Let $\psi_i \colon G_i \to G$ be the structure map for $i \in I$.

As $I$ is directed, the canonical map
\begin{equation}
\hocolim_{i \in I} \bfS(EG_i \times_{G_i}  \psi_i^* Z) \to  \bfS\bigl(\hocolim_{i \in I}  (EG_i \times_{G_i}  \psi_i^* Z)\bigr)
\label{lem:bfS_and_homotopy_colimits:map_1}
\end{equation}
is by assumption a weak homotopy equivalence.
We have the homeomorphisms 
\begin{eqnarray*}
EG_i \times_{G_i}  \psi_i^* Z 
& \xrightarrow{\cong} & 
(\psi_i)_* EG_i \times_G Z,
\\
\bigl(\hocolim_{i \in I} (\psi_i)_* EG_i\bigr) \times_GZ
&\xrightarrow{\cong} &
\hocolim_{i \in I}  \bigl((\psi_i)_* EG_i \times_G Z\bigr).
\end{eqnarray*}
They induce a homeomorphism
\begin{equation}
\bfS(\hocolim_{i \in I}  (EG_i \times_{G_i}  \psi_i^* Z)\bigr) \xrightarrow{\cong}
\bfS\bigl((\hocolim_{i \in I}  (\psi_i)_*EG_i) \times_G Z\bigr).
\label{lem:bfS_and_homotopy_colimits:map_2}
\end{equation}
The canonical map
\[
\hocolim_{i \in I}  (\psi_i)_*EG_i \to EG
\]
is a $G$--homotopy equivalence. The proof of this fact is a special case of the argument
appearing in the proof of~\cite[Theorem~4.3 on page~516]{Lueck-Weiermann(2012)}. It induces a weak homotopy equivalence
\begin{equation}
\bfS\bigl((\hocolim_{i \in I}  (\psi_i)_*EG_i) \times_G Z)  \to \bfS(EG \times_G Z).
\label{lem:bfS_and_homotopy_colimits:map_3}
\end{equation}
Hence we get by taking the composite of the
maps~\eqref{lem:bfS_and_homotopy_colimits:map_1},~\eqref{lem:bfS_and_homotopy_colimits:map_2}
and~\eqref{lem:bfS_and_homotopy_colimits:map_3}  a weak homotopy equivalence
\[
\hocolim_{i \in I} \bfS(EG_i \times_{G_i}  \psi_i^* Z)\to  \bfS(EG \times_G Z).
\]
As $I$ is directed, it induces after taking homotopy groups for every $n \in \IZ$ an isomorphism
\[
\colim_{i \in I} \pi_n\bigl(\bfS(EG_i \times_{G_i}  \psi_i^* Z)\bigr) \to  \pi_n\bigl(\bfS(EG \times_G Z)\bigr),
\]
which can be identified using~\eqref{identification_of_K-homology_theories} with the canonical map
\[
t_n^G(\pt) \colon \colim_{i \in I} H_n^{G_i}(\pt;\bfS^{\downarrow G}_Z) \to H_n^G(\pt;\bfS^{\downarrow G}_Z).
\]
This finishes the proof of \Cref{lem:bfS_and_homotopy_colimits}.
\end{proof}


\subsection{Proof of Theorem~\ref{the:Inheritance_properties_of_the_Meta-Isomorphism_Conjecture_for_bfS_with_c}}
\label{subsec:proof_of_Theorem_ref(the:Inheritance_properties_of_the_Meta-Isomorphism_Conjecture_for_bfS_with_c)}

In this section we give the proof of
\Cref{the:Inheritance_properties_of_the_Meta-Isomorphism_Conjecture_for_bfS_with_c}.
We use the notation from there.

\begin{proof}
\ref{the:Inheritance_properties_of_the_Meta-Isomorphism_Conjecture_for_bfS_with_c:subgroups}
  Consider a free $H$--CW--complex $Z$. Let $i \colon H \to G$ be the inclusion.
Then $i_*Z$ is a free $G$--CW--complex, $i^*\EGF{G}{\calc(G)}$ is a model for
$\EGF{H}{\calc(H)}$ and $i^*G/G = H/H$. From \Cref{lem:adjunctions_for_homology_associated_toS_upper_K}~%
\ref{lem:adjunctions_for_homology_associated_toS_upper_K:restriction_of_X},
we obtain a commutative diagram with isomorphisms as vertical maps
\[\xymatrix{
H_n^H(\EGF{H}{\calc(H)};\bfS^H_Z) \ar[r] \ar[d]_{\cong}
&
H_n^H(H/H;\bfS^G_Z) \ar[d]^{\cong}
\\
H_n^G(\EGF{G}{\calc(G)};\bfS^G_{i_*Z}) \ar[r]
&
H_n^G(G/G;\bfS^G_{i_*Z}) 
}
\]
where the horizontal maps are induced by the projections. The lower map is bijective by assumption.
Hence the upper map is bijective as well.
\\[1mm]~\ref{the:Inheritance_properties_of_the_Meta-Isomorphism_Conjecture_for_bfS_with_c:extensions}
As \ConjMICforS holds for $(Q, \mathcal{Q})$, by
\Cref{lem:MC_with_coeff_and_FMC_with_coeff}\ref{lem:MC_with_coeff_and_FMC_with_coeff:MC_implies_FC},
\ConjFMICforS holds for $(G, p^{*}\mathcal{C}(Q))$.  By the same
\Cref{lem:MC_with_coeff_and_FMC_with_coeff}\ref{lem:MC_with_coeff_and_FMC_with_coeff:MC_implies_FC},
for every $H\in\mathcal{C}(Q)$, \ConjMICforS holds for $(p^{-1}(H),
\mathcal{C}(p^{-1}(H)))$.  Naturally, $p^{-1}(H) \subseteq G$ is a group over
$G$ for which by
\Cref{lem:MC_with_coeff_and_FMC_with_coeff}\ref{lem:MC_with_coeff_and_FMC_with_coeff_up_from_G_to_K}
\ConjFMIC holds for $H^{?}_n(-; \mathbf{S}^{\downarrow G}_{Z})$ for any $G$-CW
complex $Z$ and the family $\mathcal{C}(p^{-1}(H)) =
\mathcal{C}(G)|_{p^{-1}(H)}$.  Let $L \in p^{*}\mathcal{C}(Q)$. Then, using
\Cref{lem:basic_inheritance_property_of_fibered_Isomorphism_conjecture} for
the map $L \to p^{-1}(p(L))$, \ConjFMIC holds for $(L, \mathcal{C}|_L)$ and
$H^{?}_n(-; \mathbf{S}^{\downarrow G}_{Z})$.  As \ConjFMIC holds for $(G,
p^{*}{C}(Q))$ and for every $L \in p^{*}{C}(Q)$ for $(L, \mathcal{C}|_L)$, the
Transitivity Principle,
see~\cite[Theorem~4.3]{Bartels-Echterhoff-Lueck(2008colim)}, implies that \ConjFMIC
holds for $(G, \mathcal{C})$.  By Lemma
\Cref{lem:MC_with_coeff_and_FMC_with_coeff}\ref{lem:MC_with_coeff_and_FMC_with_coeff:FC_implies_MC},
then also \ConjMICforS holds for $(G, \mathcal{C})$.
\\[1mm]~\ref{the:Inheritance_properties_of_the_Meta-Isomorphism_Conjecture_for_bfS_with_c:direct_products}
If \ConjMICforS 
holds for $(G_1 \times G_2,\calc(G_1 \times G_2))$, it holds for
$G_k$ and the family $\calc(G_k) = \calc(G_1 \times G_2)|_{G_k}$ for $k = 1,2
$ by
assertion~\ref{the:Inheritance_properties_of_the_Meta-Isomorphism_Conjecture_for_bfS_with_c:subgroups}.

Suppose that \ConjMICforS
holds for $(G_k,\calc(G_k))$ for $k = 1,2$. 
By assertion~\ref{the:Inheritance_properties_of_the_Meta-Isomorphism_Conjecture_for_bfS_with_c:extensions}
applied to the split exact sequence $1 \to H_2 \to G_1 \times H_2 \to G_1 \to 1$,
\ConjMICforS
holds
for $(G_1 \times H_2,\calc(G_1 \times H_2))$ for every $H_2 \in \calc(G_2)$. 
By assertion~\ref{the:Inheritance_properties_of_the_Meta-Isomorphism_Conjecture_for_bfS_with_c:extensions}
applied to the split exact sequence $1 \to G_1 \to G_1 \times G_2 \to G_2 \to 1$
\ConjMICforS 
holds for $(G_1 \times G_2,\calc(G_1 \times G_2))$.
\\[1mm]~\ref{the:Inheritance_properties_of_the_Meta-Isomorphism_Conjecture_for_bfS_with_c:directed_colimits}
Since  the \ConjMICforS
holds for $G_i$ and $\calc(G_i)$ for every $i \in I$ by assumption, we conclude from \Cref{lem:MC_with_coeff_and_FMC_with_coeff}~%
\ref{lem:MC_with_coeff_and_FMC_with_coeff:MC_implies_FC} that 
\ConjFMICforS
holds for the group $G_i$ and the family $\calc(G_i)$ for every $i \in I$.
\Cref{lem:MC_with_coeff_and_FMC_with_coeff}~\ref{lem:MC_with_coeff_and_FMC_with_coeff_up_from_G_to_K}
implies that for every  $i \in  I$ and $G$--CW--complex $Z$ 
\ConjFMIC
holds for the equivariant homology theory $H_n(-;\bfS^{\downarrow G}_Z)$ over $G$ for the group
$\psi_i \colon G_i \to G$  over $G$ and the family $\calc(G_i)$.   We conclude 
from~\cite[Theorem~5.2]{Bartels-Echterhoff-Lueck(2008colim)} 
and \Cref{lem:bfS_and_homotopy_colimits} that for every $G$--CW--complex $Z$  
\ConjFMIC
holds for the equivariant homology theory $H^?_*(-;\bfS^{\downarrow G}_{Z})$ over $G$ for the group
$(G,\id_G) $ over $G$ and the family $\calc(G)$. In other words,
\ConjFMICforS
holds for the group $G$ and the family $\calc(G)$.
\Cref{lem:MC_with_coeff_and_FMC_with_coeff}~\ref{lem:MC_with_coeff_and_FMC_with_coeff:FC_implies_MC} implies 
that \ConjMICforS
%
holds for the group $G$ and the family $\calc(G)$.
\\[1mm]~\ref{the:Inheritance_properties_of_the_Meta-Isomorphism_Conjecture_for_bfS_with_c:overgroups_of_finite_index}
The analogs of \ref{the:Inheritance_properties_of_the_Meta-Isomorphism_Conjecture_for_bfS_with_c:subgroups},~%
\ref{the:Inheritance_properties_of_the_Meta-Isomorphism_Conjecture_for_bfS_with_c:extensions},~%
\ref{the:Inheritance_properties_of_the_Meta-Isomorphism_Conjecture_for_bfS_with_c:direct_products}, 
and~\ref{the:Inheritance_properties_of_the_Meta-Isomorphism_Conjecture_for_bfS_with_c:directed_colimits}
hold for the Meta-Isomorphism Conjecture~\ref{con:The_Meta_Isomorphisms_Conjecture_for_functors_from_spaces_to_spectra_with_c_and_fwp}
with coefficients and finite wreath products by~\cite[Lemma~3.2, 3.15, 3.16, Satz~3.5]{Kuehl(2009)}.

For a group $G$ and two finite groups $F_1$ and $F_2$ we have $(H \wr F_1 )\wr F_2 \subset
H \wr (F_1 \wr F_2)$ and $F_1 \wr F_2$ is finite.  In particular, if $G$
satisfies
\Cref{con:The_Meta_Isomorphisms_Conjecture_for_functors_from_spaces_to_spectra_with_c_and_fwp} with wreath products, then the same is true for any wreath
product $G \wr F$ with $F$
finite.  If $H \subseteq G$ is a subgroup of finite index, then $G$ can be embedded in $H \wr F$ for some
finite group $F$, see \cite[Proof of Proposition 2.17]{Wegner-solvable}. Hence Meta-Isomorphism~\Cref{con:The_Meta_Isomorphisms_Conjecture_for_functors_from_spaces_to_spectra_with_c_and_fwp}
with coefficients passes to supergroups of finite index. This finishes the proof of \Cref{the:Inheritance_properties_of_the_Meta-Isomorphism_Conjecture_for_bfS_with_c}.
\end{proof}


\subsection{Proof of assertion~\ref{the:main_result:inheritance} of Theorem~\ref{the:main_result}}
\label{subsec:Proof_of_assertion_ref(the:main_result:inheritance)_of_Theorem_ref(the:main_result)}

By 
\Cref{lem:bfA_respects_weak_equivalences_and_disjoint_unions}, the functor
$\bfA$ satisfies all assumptions of
\Cref{the:Inheritance_properties_of_the_Meta-Isomorphism_Conjecture_for_bfS_with_c}.
The claim of the inheritance properties appearing in
\cref{the:main_result:inheritance} of \Cref{the:main_result} follows
immediately from
\Cref{the:Inheritance_properties_of_the_Meta-Isomorphism_Conjecture_for_bfS_with_c}
except for the statements about extensions, direct products and free
products.  For extensions, it follows from the inheritance under finite index
supergroups.  For direct products, note that the product of two virtually
cyclic groups is virtually abelian, hence by \cite{Ullmann-Winges(2015)} it
satisfies the conjecture.

For free products, note that 
due to the inheritance under filtered colimits, we can assume our groups are
finitely generated, in particular countable.
  For $G_1$,
  $G_2 \in \calfj_A$ consider the canonical map $p \colon G_1 * G_2 \to G_1 \times G_2$.  We
  already know that $G_1 \times G_2 \in \calfj_A$  and hence that it suffices to prove 
   $p^{-1}(C) \in \calfj_A$, where $C$ is the trivial or any infinite cyclic subgroup of $G_1   \times G_2$.  
  By~\cite[Lemma~5.2]{Roushon(2008FJJ3)} all such $p^{-1}(C)$ are 
  free and hence hyperbolic, as $G_1 \times G_2$ is countable.


\section{Proof of the Farrell--Jones Conjecture for hyperbolic and CAT(0)--groups}\label{Proof of the Farrell--Jones Conjecture for hyperbolic and CAT(0)--groups}
 Thanks to the framework established in \cite{Ullmann-Winges(2015)}, we can
 proceed similarly to the linear case as in \cite{Wegner(2012)} and reduce the proof to the construction of a transfer map. This reduction is carried out in this section, while the construction of the transfer occupies \Cref{The Transfer: Final Part of the Proof}.
 

 
\subsection{Homotopy coherent actions and homotopy transfer reducibility}
\label{subsec:proof-fj:homotopy-coherent}
  
  The geometric criterion we use to prove the conjecture relies on the notion
  of a homotopy coherent diagram, which goes back to Vogt \cite{Vogt(1973)}.
  For applications to the Farrell--Jones conjecture, it is enough to consider
  the case of a homotopy coherent diagram of shape $G$, regarding $G$ as a
  one-object groupoid. In this special case, Vogt's definition was
  rediscovered by Wegner \cite[Definition~2.1]{Wegner(2012)}, who called it
  ``strong homotopy action''. 
  \begin{dfn}\label{def:hptycohG-action}
  A \emph{homotopy coherent $G$--action} of a group $G$ on a topological space $X$ is a continuous map
  \[
  \Gamma: \coprod_{j=0}^\infty((G \times [0,1])^j\times G\times X)\to X
  \]
  with the following properties:
  \begin{equation*}
      \Gamma(\gamma_k,t_k,\dots,\gamma_1,t_1,\gamma_0,x) = \begin{cases}
       \Gamma(\dots, \gamma_j, \Gamma(\gamma_{j-1},\dots,x)) & t_j = 0 \\
       \Gamma(\dots,\gamma_j\gamma_{j-1},\dots,x) & t_j = 1 \\
       \Gamma(\gamma_k,\dots,\gamma_2,t_2,\gamma_1,x) & \gamma_0 = e, 0 < k \\
       \Gamma(\gamma_k,\dots,t_{j+1}t_j,\dots,\gamma_0,x) & \gamma_j = e, 1 \leq j < k \\
       \Gamma(\gamma_{k-1},t_{k-1},\dots,t_1,\gamma_0,x) & \gamma_k = e, 0 < k
       \\
       x & \gamma_0 = e, k=0
      \end{cases}
  \end{equation*}
  \end{dfn}
  
  The following definition is adapted from the conditions given in
  \cite[Theorem~B]{Bartels(2012)}, which does not use coherence conditions.
  We explain some notation below.
  
  \begin{dfn}\label{def:transfer-reducible}
   Let $G$ be a discrete group. Let $\cF$ be a family of subgroups of $G$.  

   Then $G$ is \emph{homotopy transfer reducible over $\cF$} if there exists a
   finite, symmetric generating set $S \subset G$ of $G$ which contains the
   trivial element, as well as $N \in \NN$ such that there are for every $n
   \in \NN$
    \begin{enumerate}
     \item a compact, contractible metric space $(X,d_X)$ such that for every $\epsilon > 0$ there is an $\epsilon$--controlled domination of $X$ by an at most $N$--dimensional, finite simplicial complex.
     \item a homotopy coherent $G$--action $\Gamma$ on $X$.
     \item a $G$--simplicial complex $\Sigma$ of dimension at most $N$ whose isotropy is contained in $\cF$.
     \item a continuous map $f \colon X \to \Sigma$ which is \emph{$(S,n)$--equivariant} in the sense that  
      \begin{itemize}
        \item for all $x \in X$ and $s \in S^n$,
          \begin{equation*}
           d^{\ell^1}(f(\Gamma(s,x)),s \cdot f(x)) \leq \frac{1}{n}.
          \end{equation*}
        \item for all $x \in X$ and $s_0,\dots,s_n \in S^n$,
          \begin{equation*}
           \diam \{ f(\Gamma(s_n,t_n,\dots,s_0,x)) \mid (t_1,\dots,t_n) \in [0,1]^n \} \leq \frac{2}{n}.
          \end{equation*}
      \end{itemize}
    \end{enumerate}
  \end{dfn}

\begin{notation} 
  Let us briefly recall some notation used in \Cref{def:transfer-reducible}.
  \begin{enumerate}
    \item Recall from~\cite[Definition 1.5]{Bartels-Lueck(2012annals)} that an
      $\epsilon$--controlled domination of a metric space $(X,d)$ by a finite
      simplicial complex $K$ consists of maps $i\colon X \to K$, $p\colon K\to
      X$ together with a homotopy $H$ from $p \circ i$ to $\id_X$ such that for
      every $x\in X$ the diameter of $\{H(x,t) \mid t\in [0,1]\}$ is at most
      $\epsilon$.
    \item The $\ell^1$--metric $d^{\ell^1}$ on a simplicial complex is defined
      in~\cite[4.2]{Bartels-Lueck-Reich(2008hyper)}.
    \item If $S$ is a finite generating set of $G$, we denote by $S^n
      \subseteq G$ the set $\{ s_1 s_2 \dots s_n \in G \mid s_i \in S\}$.  We
      always equip $G$ with the word metric $d_G$ with respect to $S$. 
      Equivalently, $S^n$ is the $n$--ball around the trivial element with
      respect to $d_G$.
  \end{enumerate}
\end{notation}

We will show in \Cref{subsec:proof-fj:main-theorem} that a group
satisfying \Cref{def:transfer-reducible} satisfies the Farrell-Jones
Conjecture with coefficients in $A$--theory with respect to the family $\cF$,
and we show in
\Cref{subsec:proof-fj:homotopy-transfer-follows-strongly-transfer} that
hyperbolic and $\mathrm{CAT}(0)$--groups satisfy
\Cref{def:transfer-reducible}, thus proving
\Cref{the:main_result}\ref{the:main_result:groups}.



\subsection{Controlled CW--complexes}
\label{subsec:proof-fj:controlled-cw}
  Let $G$ be a discrete group and let $\cF$ be a family of subgroups of $G$.
  In \cite{Ullmann-Winges(2015)} it was shown that \Cref{thm:afjc-transfer-reducible} holds for $G$ if and only if a certain spectrum $\mathbb{F}(G, W, \EGF{G}{\calf})$ is weakly contractible for every free $G$--CW--complex $W$. The spectrum $\mathbb{F}(G, W, \EGF{G}{\calf})$ is the algebraic $K$--theory of a Waldhausen category of controlled retractive $G$--CW--complexes, similar in spirit to the obstruction category for the Isomorphism Conjecture in algebraic $K$--theory, cf.~\cite{Bartels-Farrell-Jones-Reich(2004)}, \cite[Section~3]{Bartels-Lueck-Reich(2008hyper)}.
  Let us recall the relevant definitions from \cite{Ullmann-Winges(2015)} in
  this and the next section.
  
  A \emph{coarse structure} is a triple $\fZ=(Z, \mathfrak{C}, \mathfrak{S})$
  such that $Z$ is a Hausdorff $G$--space, $\fC$ is a collection of reflexive,
  symmetric and $G$--invariant relations on $Z$ which is closed under taking
  finite unions and compositions,
  see \cite[Definition~2.1]{Ullmann-Winges(2015)}, and $\fS$ is a collection of
  $G$--invariant subsets of $Z$ which is closed under taking finite unions.
  See \cite[Definition~3.23]{Ullmann-Winges(2015)} for the notion of a
  \emph{morphism of coarse structures}.
  
  Fix a coarse structure $\fZ$.
  
  A \emph{labeled $G$--CW--complex relative W}, see \cite[Definition~2.3]{Ullmann-Winges(2015)}, is a pair $(Y, \kappa)$, where $Y$ is a free $G$--CW--complex relative $W$ together with a $G$--equivariant function $\kappa \colon \cells Y \rightarrow Z$. Here, $\diamond Y$ denotes the (discrete) set of relative cells of $Y$.
  
  A \emph{$\fZ$--controlled map} $f \colon (Y_1,\kappa_1) \rightarrow (Y_2, \kappa_2)$ is a $G$--equivariant, cellular map $f \colon Y_1 \rightarrow Y_2$ relative $W$ such that for all $k \in \mathbb{N}$ there is some $C \in \mathfrak{C}$ for which
  \[
  (\kappa_2,\kappa_1)(\{(e_2,e_1) \mid e_1 \in \cells_k Y_1, e_2\in \cells Y_2, \langle f(e_1) \rangle \cap e_2 \neq \varnothing\}) \subseteq C
  \]
  holds, where $\langle f(e_1)\rangle$ denotes the smallest non-equivariant
  subcomplex of $Y_2$ which contains $f(e_1)$.
  
  A \emph{$\fZ$--controlled $G$--CW--complex relative W} is a labeled $G$--CW--complex $(Y,\kappa)$ relative $W$, such that the identity is a $\fZ$--controlled map and for all $k \in \mathbb{N}$ there is some $S \in \mathfrak{S}$ such that
  \[
    \kappa(\diamond_k Y)\subseteq S.
  \]
  
  A \emph{$\fZ$--controlled retractive space relative $W$} is a
  $\fZ$--controlled $G$--CW--complex  $(Y,\kappa)$ relative $W$ together with
  a $G$--equivariant retraction $r \colon Y \to W$, i.e., a left inverse to
  the structural inclusion $W \hookrightarrow Y$. The $\fZ$--controlled
  retractive spaces relative $W$ form a category $\cR^G(W,\fZ)$ in which
  \emph{morphisms} are $\fZ$--controlled maps which additionally respect the chosen retractions.
  
  The category of controlled $G$--CW--complexes (relative $W$) and controlled
  maps admits a notion of \textit{controlled homotopies}, see
  \cite[Definition~2.5]{Ullmann-Winges(2015)} via the objects $(Y
  \leftthreetimes [0,1], \kappa \circ pr_Y)$, where $Y \leftthreetimes [0,1]$
  denotes the reduced product which identifies $W \times [0,1] \subseteq Y
  \times [0,1]$ to a single copy of $W$ and $pr_Y: \cells (Y \leftthreetimes
  [0,1]) \rightarrow \diamond Y$ is the canonical projection.
  In particular, we obtain a notion of \emph{controlled homotopy equivalence} (or \emph{$h$--equivalence}).
  
  A $\fZ$--controlled retractive space $(Y, \kappa)$ is called \textit{finite} if it is finite-dimensional, the image of $Y \backslash W$ under the retraction meets the orbits of only finitely many path components of $W$ and for each $z \in Z$ there is some open neighborhood $U$ of $z$ such that $\kappa^{-1}(U)$ is finite, see \cite[Definition~3.3]{Ullmann-Winges(2015)}.
  
  A $\fZ$--controlled retractive space $(Y, \kappa)$ is called
  \textit{finitely dominated}, if there are a finite $\fZ$--controlled,
  retractive space $D$, a morphism $p \colon D \rightarrow Y$ and a
  $\fZ$--controlled map $i \colon Y \rightarrow D$ such that $p \circ i$ is
  controlled homotopic to $\id_Y$.  

  The finite, respectively finitely dominated, $\fZ$--controlled retractive spaces form full subcategories $\cR^G_f(W,\fZ) \subset \cR^G_{fd}(W,\fZ) \subset \cR^G(W,\fZ)$.
  All three of these categories support a Waldhausen category structure in
  which inclusions of $G$--invariant subcomplexes up to isomorphism are the
  cofibrations and controlled homotopy equivalences are the weak equivalences,
  see \cite[Corollary~3.22]{Ullmann-Winges(2015)}.  We denote this class of
  weak equivalences by $h$.
  
  Note that a controlled homotopy equivalence is a \emph{morphism}, but only
  admits a controlled homotopy inverse \emph{map}, which does not need to be
  compatible with the retractions to $W$.  This is similar to the classical
  situation  \cite[Section~2.1]{Waldhausen(1985)}.
  
\subsection{The obstruction category}
\label{subsec:proof-fj:the-obstruction-category}
  Let $M$ be a metric space with free, isometric $G$--action. Define the
  \emph{bounded morphism control condition on $M$},
  $\fC_{bdd}(M)$, to be the collection of all subsets $C \subset M \times M$
  which are of the form
  \begin{equation*}
      C = \{ (m,m') \in M \times M \mid d(m,m') \leq \alpha \}
  \end{equation*}
  for some $\alpha \geq 0$. 
  
  Let $X$ be a $G$--CW--complex. Define further the \emph{$G$--continuous
  control condition} $\fC_{Gcc}(X)$ to be the collection of all $C \subset (X \times [1,\infty[) \times (X \times [1,\infty[)$ which satisfy the following:
  \begin{enumerate}
      \item For every $x \in X$ and every $G_x$--invariant open neighborhood
        $U$ of $(x,\infty)$ in $X \times [1,\infty]$, there exists a
        $G_x$--invariant open neighborhood $V \subset U$ of $(x,\infty)$ such
        that $(((X \times [1,\infty[) \smallsetminus U) \times V) \cap C = \varnothing$.
      \item Let $p_{[1,\infty[} \colon X \times [1,\infty[ \to [1,\infty[$ be the projection map. Equip $[1,\infty[$ with the Euclidean metric. Then there exists some $B \in \fC_{bdd}([1,\infty[)$ such that $C \subset p^{-1}_{[1,\infty[}(B)$.
      \item $C$ is symmetric, $G$--invariant and contains the diagonal.
  \end{enumerate}
  We can combine the two morphims control conditions into one set of
  conditions on $M \times X \times [1, \infty[$: Let $p_M \colon M \times X
  \times [1,\infty[ \to M$ and $p_{X \times [1,\infty[} \colon M \times X
  \times [1,\infty[ \to X \times [1,\infty[$ denote the projection maps.
  Then $\fC(M,X)$ is the collection of all subsets $C \subset (M \times
  X \times [1,\infty[)^2$ which are of the form
  \begin{equation*}
   C = p_M^{-1}(B) \cap p_{X \times [1,\infty[}^{-1}(C')    
  \end{equation*}
  for some $B \in \fC_{bdd}(M)$ and $C' \in \fC_{Gcc}(X)$.
  
  Finally, define $\fS(M,X)$ to be the collection of all subsets $S \subset M
  \times X \times [1,\infty[$ which are of the form $S = K \times [1,\infty[$
  for some $G$--compact subset $K \subset M \times X$.  Recall that
  $\EGF{G}{\calf}$ denotes the classifying space of $G$ with respect to
  $\calf$.  We also consider $G$ as a metric space with the word metric
  induced by a generating set $S$.
\begin{definition}
  With the above definitions we obtain a coarse structure
  \begin{equation*}
      \JJ(M,X) := (M \times X \times [1,\infty[, \fC(M,X), \fS(M,X)).
  \end{equation*}
  Define the  ``obstruction category'' as the category of finite controlled
  CW-complexes relative $W$, i.e., as
  \begin{equation*}
    \cR^G_f(W,\JJ(G,\EGF{G}{\calf})),h),
  \end{equation*}
  cf.~\cite[Example~2.2 and Definition~6.1]{Ullmann-Winges(2015)}.  The
  spectrum $\FF(G,W,\EGF{G}{\calf})$ alluded to before is the non-connective
  $K$--theory
  spectrum of $\cR^G_f(W,\JJ(G,\EGF{G}{\calf}))$ with respect to the
  $h$--equivalences, cf.~\cite[Section~5]{Ullmann-Winges(2015)} and
  \Cref{def:delooping-coarse-structure} below.  If $M = G$, we often
  abbreviate $\JJ(G,X)$ as $\JJ(X)$.
\end{definition}
  By \cite[Corollary~6.11]{Ullmann-Winges(2015)}, a group $G$ satisfies the
  Farrell-Jones Conjecture \ref{con:FJC_for_A-theory_with_c} with coefficients in $A$--theory with respect to
  $\cF$ if and only if $\FF(G,W,\EGF{G}{\calf})$ is weakly contractible for
  every free $G$--CW--complex $W$.
  
\subsection{The target of the transfer}
\label{subsec:proof-fj:target-of-the-transfer}
  Suppose that $G$ is homotopy transfer reducible in the sense of \Cref{def:transfer-reducible}.
  The key step in proving the weak contractibility of
  $\FF(G,W,\EGF{G}{\calf})$ will be the construction of a ``transfer map". We
  need a generalization of the coarse structure $\JJ(M,X)$ to define the
  target of the transfer.
  
  Suppose that $(M_n)_n$ is a sequence of metric spaces with a free, isometric $G$--action. Let $X$ be a $G$--CW--complex. Following \cite[Section~7]{Ullmann-Winges(2015)}, define the coarse structure
  \begin{equation*}
      \JJ((M_n)_n,X) := \big( \coprod_n M_n \times X \times [1,\infty[, \fC((M_n)_n,X), \fS((M_n)_n,X) \big)
  \end{equation*}
  as follows: Members of $\fC((M_n)_n,X)$ are of the form $C = \coprod_n C_n$
  with $C_n \in \fC(M_n,X)$, and we additionally require that $C$ satisfies
  the \emph{uniform metric control conditon}: There is some $\alpha > 0$,
  independent of $n$, such that for all $((m,x,t)$, $(m',x',t')) \in C$ we
  have $d(m,m') < \alpha$. Members of $\fS((M_n)_n,X)$ are sets of the form $S
  = \coprod_n S_n$ with $S_n \in \fS(M_n,X)$. The resulting category
  $\cR^G(W,\JJ((M_n)_n,X))$ has a canonical faithful functor into the product category $\prod_n \cR^G(W,\JJ(M_n,X))$. 
  
  Fix a symmetric, finite generating set $S$ of $G$. Let $d_G$ denote the word metric on $G$ with respect to $S$. Since $G$ is homotopy transfer reducible by assumption, there exists a natural number $N\in \NN$ such that we can choose for each $n\in \NN$
   \begin{enumerate}
     \item a compact, contractible metric space $(X_n,d_{X_n})$ such that for every $\epsilon > 0$ there is an $\epsilon$--controlled domination of $X_n$ by an at most $N$--dimensional, finite simplicial complex;
     \item a homotopy coherent $G$--action $\Gamma_n$ on $X_n$;
     \item a $G$--simplicial complex $\Sigma_n$ of dimension at most $N$ whose isotropy is contained in $\cF$;
     \item a map $f_n \colon X \to \Sigma_n$ which is $(S,n)$--equivariant,
       i.e.,
      \begin{enumerate}
        \item for all $x \in X_n$ and $s \in S^n$,
          \begin{equation}\label{eq:homotopy-trans-reduc:almost-equivariant}
           d^{\ell^1}(f(\Gamma_n(s,x)),s \cdot f_n(x)) \leq \frac{1}{n};
          \end{equation}
        \item for all $x \in X_n$ and $s_0,\dots,s_n \in S^n$,
          \begin{equation}\label{eq:homotopy-trans-reduc:diameter-bounded}
           \diam \{ f_n(\Gamma_n(s_n,t_n,\dots,s_0,x)) \mid (t_1,\dots,t_n) \in [0,1]^n \} \leq \frac{2}{n}.
          \end{equation}
      \end{enumerate}
   \end{enumerate}
   \begin{definition}\label{def:metric-on-Sigma-G}
     We equip $\Sigma_n \times G$ with the metric $ n \cdot d^{\ell^1}(x,y) +
     d_G(g,h) $.
   \end{definition}
   Recall that an extended metric satisfies the usual axioms of a metric, but
   it is allowed to take the value $\infty$.  The following definition will be
   used to produce a metric on $X_n \times G$ for each $n \in \NN$.  
     
   \begin{dfn}\label{def-metric}
   Let $(X,d_X)$ be a metric space, $\Gamma$ a homotopy coherent $G$--action
   on $X$, and $S \subset G$ a finite subset containing the trivial element.
   Let $k \in \NN$ and $\Lambda > 0$. Define on $X \times G$ the extended metric
    \begin{equation*}
     d_{S,k,\Lambda}( (x,g), (y,h) ) \in [0,\infty]
    \end{equation*}
   to be the infimum over the numbers
    \begin{equation*}
      l + \sum_{i=0}^l \Lambda \cdot d_X(x_i,z_i), 
    \end{equation*}
   where the infimum is taken over all $l \in \NN$, $x_0,\dots,x_l$, $z_0,\dots,z_l \in X$ and $a_1,\dots,a_l$, $b_1,\dots,b_l \in S$ such that
    \begin{enumerate}
     \item $x_0 = x$ and $z_l = y$;
     \item $ga_1^{-1}b_1\dots a_l^{-1}b_l = h$;
     \item for each $1 \leq i \leq l$ there are elements $r_0,\dots,r_k,s_0,\dots,s_k \in S$ such that $a_i = r_k\dots r_0$, $b_i = s_k \dots s_0$ and $\Gamma(r_k,t_k,\dots,r_0,z_{i-1}) = \Gamma(s_k,u_k,\dots,s_0,x_i)$ for some $t_1, \dots,t_k,u_1,\dots,u_k \in [0,1]$.
    \end{enumerate}
    If no such data exist, take the infimum to be $\infty$.
   \end{dfn}
  
  This definition is analogous to
  \cite[Definition~3.4]{Bartels-Lueck(2012annals)} and
  \cite[Definition~2.3]{Wegner(2012)}. Since we only consider the coherent
  $G$--action $\Gamma_n$ on $X_n$, we drop $\Gamma_n$ from the notation of
  \cite{Wegner(2012)}. The proof of the next lemma is analogous to the one
  given in \cite[Lemma~3.5]{Bartels-Lueck(2012annals)}.

  \begin{lem}\label{lem:metricisnice}
   Let $k \in \NN$.
    \begin{enumerate}
     \item For all $\Lambda > 0$, the function $d_{S,k,\Lambda}$ is an extended metric on $X \times G$ which is $G$--invariant if we let $G$ act on $X \times G$ by $\gamma \cdot (x,g) = (x,\gamma g)$. It is a metric if and only if $S$ generates $G$.
     \item We have $d_{S,k,\Lambda}((x,g),(y,h)) < 1$ if and only if $g = h$ and $\Lambda \cdot d_X(x,y) < 1$ holds, in which case we have $d_{S,k,\Lambda}((x,g),(y,h)) = \Lambda \cdot d_X(x,y)$. In particular, the topology induced by $d_{S,k,\Lambda}$ is the product topology. \qed
    \end{enumerate} 
  \end{lem}

\subsection{The actual target of the transfer}
\label{subsec:proof-fj:actual-target-of-the-transfer}
  We now specialize the construction of
  \Cref{subsec:proof-fj:target-of-the-transfer} to our needs.
  Assume that $G$ is homotopy transfer reducible, i.e., it satisfies
  \Cref{def:transfer-reducible}.  That definition provides us for every $n$
  with a metric space $X_n$, as well as $\Gamma_n$, $f_n$ and $\Sigma_n$.
  From \Cref{def-metric} and \Cref{lem:metricisnice} we obtain for any
  sequence $(\Lambda_n)_n$ a sequence of metric spaces $(X_n \times G, d_{S^n,
  n, \Lambda_n})_n$.
  Although we do not need to restrict to a specific choice of
  $(\Lambda_n)_n$ until a little later, we wish to avoid spreading our choices
  throughout the whole proof. Therefore, we will now fix a specific sequence
  $(\Lambda_n)_n$.
  
  Since each $X_n$ is compact, $f_n$ is uniformly continuous. Hence, there
  exists for each $n \in \NN$ some $\delta_n > 0$ such that for all $x,y \in
  X_n$ with $d_X(x,y) < \delta_n$ we have $d^{\ell^1}(f_n(x),f_n(y)) <
  \frac{1}{n}$. 
  \begin{definition}\label{def:lambda-delta}
    Choose such $\delta_n$ for all $n$ and set
    \begin{equation*}
      \Lambda_n := \frac{n+1}{\delta_n}.
    \end{equation*}
    Define a metric $d_n$ on $X_n \times G$ by
    \begin{equation*}
      d_n((x,g),(y,h)) := d_{S^n,n,\Lambda_n}((x,g),(y,h)) + d_G(g,h).
    \end{equation*}
  \end{definition}

  Then $X_n \times G$ carries a free and isometric $G$--action if we let $G$
  act on the right factor. If we make no explicit mention of a metric, we will
  view $X_n \times G$ as a metric space with respect to $d_n$ in what follows.
  Similarly, $\Sigma_n \times G$ carries a diagonal $G$--action and will
  always be understood as a metric space with respect to the metric $n\cdot
  d^{\ell^1} + d_G$ from \Cref{def:metric-on-Sigma-G}.  Abbreviate $E :=
  E_\cF(G)$.  
  
  The category
  $\cR^G(W,\JJ((X_n \times G)_n,E))$ will be the target of the ``transfer''.  However,
  we need to equip it with another class of weak equivalences.
  These $h^{fin}$--equivalences were introduced in the proof of
  \cite[Theorem~10.1]{Ullmann-Winges(2015)}. Basically, they ignore the
  behavior of an object on finitely many factors and behave like
  $h$--equivalences otherwise. 
  \begin{dfn}\label{def:truncated-object}
   Let $(M_n)_n$ be a sequence of metric spaces with free, isometric $G$--action (e.g.~$M_n = X_n \times G$).
   
   Let $(Y_n)_n$ be an object of $\cR^G(W, \JJ((M_n)_n,E))$. For $\nu \in \NN$, we denote by $(-)_{n > \nu}$ the endofunctor which sends $(Y_n)_n$ to the sequence $(\widetilde{Y}_n)_n$ with $\widetilde{Y}_n = \ast$ for $n \leq \nu$ and $\widetilde{Y}_n = Y_n$ for $n > \nu$.
   
   A morphism $(f_n)_n \colon (Y_n)_n \rightarrow (Y_n')_n$ is an \emph{$h^{fin}$--equivalence} if there is some $\nu \in \NN$, such that $(f_n)_{n>\nu}\colon (Y_n)_{n>\nu} \rightarrow (Y_n')_{n>\nu}$ is an $h$--equivalence.
  \end{dfn}

\begin{lem}\label{lem:functoriality-of-obstruction-category}
  Let $(M_n)_n$, $(N_n)_n$ be sequences of metric spaces with free, isometric
  $G$-action. Let $(g_n)_n\colon (M_n)_n \to (N_n)_n$ be a uniformly expanding
  sequence of $G$--equivariant maps, i.e., for every $\alpha > 0$ there is
  some $\beta > 0$ such that for all $n \in \NN$ and $x,y \in M_n$ we have
  $d(g_n(x),g_n(y)) < \beta$ whenever $d(x,y) < \alpha$.

  Then $(g_n)_n$ induces a map $\cR^G(W, \JJ((M_n)_n,E)) \to \cR^G(W,
  \JJ((N_n)_n,E))$ which also respects $h$-- and $h^{fin}$--equivalences, as
  well as finiteness conditions.
\end{lem}

\begin{proof}
  As $(g_n)_n$ induces a map on $\JJ((M_n)_n, E)$ which respects the control
  conditions, it also respects the $h$--equivalences.  As it maps $M_n$ to
  $N_n$, it also respects the $h^{fin}$--equivalences.
\end{proof}

We will discuss the difference between the $h$-- and the
$h^{fin}$--equivalences in \Cref{subsec:proof-fj:squeezing}.
%

\subsection{Non-connective algebraic $K$--theory of controlled CW--complexes}
\label{subsec:proof-fj:negative-A-groups}
Before we turn to the main theorem, we need to briefly recall the definition
of algebraic $K$--theory in our setting.  Let $\fZ = (Z, \fC, \fS)$ be a
coarse structure.  Then $\cR^G_f(W,\fZ)$ and its variants are Waldhausen
categories, hence their algebraic $K$--theory is defined
by~\cite{Waldhausen(1985)}.  However, we need the non-connective delooping from
\cite[Section~5]{Ullmann-Winges(2015)}, which we briefly recall for
completeness.

\begin{dfn}\label{def:delooping-coarse-structure}
   Let $\fZ = (Z, \fC, \fS)$ be a coarse structure. For $n \in \NN$ define the coarse structure $\fZ(n) = (\RR^n \times Z, \fC(n), \fS(n))$ as follows:
   A set $C \subset (\RR^n \times Z)^2$ is in $\fC(n)$ if and only if:
   \begin{enumerate}
    \item $C$ is symmetric, $G$--invariant and contains the diagonal.
    \item $C \subseteq p_n^{-1}(C')$ for some $C' \in \fC_{bdd}(\RR^n)$, where $p_n \colon \RR^n \times Z \to \RR^n$ is the projection map.
    \item For all $K \subset \RR^n$ compact, there is a $C' \in \fC$ such that
     \begin{equation*}
      C \cap ((K \times Z) \times (K \times Z)) \subset p_Z^{-1}(C'),
     \end{equation*}
      where $p_Z \colon \RR^n \times Z \to Z$ is the projection map.
   \end{enumerate}
   Let $\fS(n)$ be the collection of all $S \subset \RR^n \times Z$ such that $S = p_Z^{-1}(S')$ for some $S' \in \fS(Z)$.
  \end{dfn}

  Consider for all $n$ also the restricted coarse structures
  \begin{equation*}
  \begin{split}
   \fZ(n+1)^+ &:= \fZ(n+1) \cap (\RR^n \times \RR_{\geq 0} \times Z), \\
   \fZ(n+1)^- &:= \fZ(n+1) \cap (\RR^n \times \RR_{\leq 0} \times Z). \\
  \end{split}
  \end{equation*}
  Note that $\fZ(n+1) \cap (\RR^n \times \{ 0 \} \times Z) = \fZ(n)$. The inclusion maps give rise to a commutative square
  \begin{equation*}\label{diag:deloopingsquare}
   \begin{tikzpicture}
    \matrix (m) [matrix of math nodes, column sep=2em, row sep=2em, text depth=.5em, text height=1em]
    {h S_\bullet \cR^G_f(W,\fZ(n)) & h S_\bullet \cR^G_f(W,\fZ(n+1)^+) \\ h S_\bullet \cR^G_f(W,\fZ(n+1)^-) & h S_\bullet \cR^G_f(W,\fZ(n+1)). \\};
    \path[->]
    (m-1-1) edge (m-1-2)
    (m-1-1) edge (m-2-1)
    (m-1-2) edge (m-2-2)
    (m-2-1) edge (m-2-2);
   \end{tikzpicture}
  \end{equation*}
  By an Eilenberg swindle, the top right and bottom left corners of this square are contractible. This provides us with structure maps for a spectrum
  \begin{equation*}
   \KK^{-\infty}(\cR^G_f(W,\fZ),h)_n := K(\cR^G_f(W,\fZ(n)), h)
  \end{equation*}
  which we call the \emph{non-connective algebraic $K$--theory spectrum} of $\cR^G_f(W,\fZ)$. The construction is functorial in $\fZ$.

  Non-connective algebraic $K$--theory is a functor on coarse structures and
  morphisms of coarse structures by \cite[Section~5]{Ullmann-Winges(2015)}.
  All the arguments which will follow do not interact with a possible
  $\RR^n$--coordinate, hence can also be carried out for $n > 0$, similar
  to \cite[Section~9]{Ullmann-Winges(2015)}.
  From the next section onwards, our proofs will only treat the case $n=0$. 
\subsection{The main theorem}
\label{subsec:proof-fj:main-theorem}
  In this part we show the following result. 
  
  \begin{thm}\label{thm:afjc-transfer-reducible}
   Let $G$ be a discrete group and let $\cF$ be a family of subgroups of $G$.
   If $G$ is homotopy transfer reducible over $\cF$, then $G$ satisfies the
   Farrell-Jones \Cref{con:FJC_for_A-theory_with_c} with coefficients in $A$--theory with respect to $\cF$.
 \end{thm}
  \Cref{the:main_result}~\ref{the:main_result:groups} follows from
  \Cref{thm:afjc-transfer-reducible} in conjunction with the inheritance properties established in \Cref{sec:Inheritance_properties_of_the_Isomorphism_Conjectures}, cf.~the introduction. We derive the validity of \Cref{the:main_result}~\ref{the:main_result:groups} for hyperbolic and $\mathrm{CAT}(0)$--groups in \Cref{cor:hypcatfollowproof} below.

  We follow the strategy of \cite[Section~5]{Wegner(2012)}.  We construct a
  commutative diagram of Waldhausen categories and exact functors
  \[
   \begin{tikzpicture}
    \matrix (m) [matrix of math nodes, column sep=2em, row sep=2em, text depth=.5em, text height=1em, ampersand replacement=\&]
    {(\cR_{fd}^G(W, \JJ((X_n \times G)_n, E)), h^{fin}) \& 
          (\cR_{fd}^G(W, \JJ((\Sigma_n \times G)_n, E)), h^{fin}) \\ 
     \&
          (\cR_{fd}^G(W, \JJ((G)_n, E)), h^{fin})\\
     (\cR^G_f(W, \JJ(E)),h) \& 
          (\cR^G_{fd}(W, \JJ(E)),h).\\
        };
    \path[->]
    (m-1-1) edge node[above]{$F$} (m-1-2)
    (m-1-1) edge node[below]{$p_{X_n \times G\rightarrow G}\;\;\;\;\;$} (m-2-2)
    (m-3-2) edge node[right]{$\Delta$} (m-2-2)
    (m-1-2) edge node[right]{$p_{\Sigma_n \times G\rightarrow G}$} (m-2-2)
    (m-3-1) edge node[above]{$\incl$} (m-3-2);
    \path[dashed][->]
    (m-3-1) edge node[left]{$\trans$}  (m-1-1);
   \end{tikzpicture}
  \]
  We define the maps $\trans$, $\Delta$ and $F$ below
  and show the following.
  
  \begin{prop}\label{prop:main-diag-claims}
  \mbox{}
  
   \begin{enumerate}
    \item\label{item:main-diag-claims-1} The arrow $\trans$ exists after
      applying non-connective algebraic $K$--theory. It will be induced by a
      map of spectra whose domain is weakly equivalent to
      $\KK^{-\infty}(\cR^G_f(W, \JJ(E)),h)$. The square formed by $p_{X_n
      \times G\rightarrow G} \circ \trans$ and $\Delta \circ \incl$ commutes up
      to levelwise weak equivalence of spectra.
    \item\label{item:main-diag-claims-2} The functor $F$, defined below,
    is well-defined.
    \item\label{item:main-diag-claims-3} The algebraic $K$--theory of $\cR_{fd}^G(W, \JJ((\Sigma_n \times G)_n, E), h^{fin})$ vanishes.
    \item\label{item:main-diag-claims-4} The map $\Delta\circ \incl$ is
      injective on non-connective algebraic $K$--theory groups.
   \end{enumerate}
  \end{prop}
  
  Given all of this, the proof can be finished as in
  \cite[Section~5]{Wegner(2012)}. It is a diagram chase on the level of
  homotopy groups.

  Before proving \Cref{prop:main-diag-claims}, let us define the maps of the
  diagram.  The maps $p$ are induced by the indicated projections on control
  spaces, ``$\incl$'' is the inclusion of the finite into the finitely dominated
  objects.  The functor $\Delta$ is induced by the diagonal map into $\prod_n
  \cR^G_{fd}(W,\JJ(G_n,X))$, which factors over $(\cR_{fd}^G(W, \JJ((G)_n, E)),
  h^{fin})$ because its image consists of uniformly controlled objects and
  maps.  Also note that every $h$--equivalence is an $h^{fin}$--equivalence.

  The functor $F$ is defined using the maps $f_n$ from
  \Cref{subsec:proof-fj:target-of-the-transfer}.  It is induced by the maps $F_n
  \colon X_n \times G \to \Sigma_n \times G$, $(x,g) \mapsto (g f_n(x), g)$.
  We show in \Cref{subsec:proof-fj:squeezing} that with our choices these are
  uniformly bounded.
  The map ``$\trans$'' is constructed in \Cref{The Transfer: Final Part of
  the Proof}.
 
\subsection{Squeezing}
\label{subsec:proof-fj:squeezing}
  All claims made in \Cref{prop:main-diag-claims} except part
  \ref{item:main-diag-claims-1} admit fairly short proofs which we give in
  this section.  Part \ref{item:main-diag-claims-1} will be shown in
  \Cref{subsec:transfer-map}


  \Cref{prop:main-diag-claims}~\ref{item:main-diag-claims-3} follows from the
  ``Squeezing Theorem'' \cite[Theorem~10.1]{Ullmann-Winges(2015)} and the fact
  that non-connective $K$--theory does not distinguish between finite and
  finitely dominated objects \cite[Remark~5.5]{Ullmann-Winges(2015)}.
  Indeed, 
  in the proof of Theorem 10.1 of \cite{Ullmann-Winges(2015)}, in equation (21), a
  homotopy fiber sequence
  \begin{multline}\label{eq:homotopy-fiber-sequence}
    \KK^{-\infty}\left(\colim_n \prod_{k=1}^n \cR_f^G(W, \JJ(M_k,E)), h \right)
    \longrightarrow \\
    \KK^{-\infty}(\cR_f^G(W, \JJ((M_n)_n,E)), h) \longrightarrow 
    \KK^{-\infty}(\cR_f^G(W, \JJ((M_n)_n,E)), h^{fin})
  \end{multline}
  is established for any sequence of metric spaces $(M_n)_n$.  Then it is
  shown there that under the assumptions from
  \Cref{subsec:proof-fj:target-of-the-transfer} on $(\Sigma_n \times G)_n$,
  the first map is a weak equivalence by proving that the last object is
  weakly contractible.
  
  
  Let us discuss \Cref{prop:main-diag-claims}~\ref{item:main-diag-claims-2}
  next. It suffices to show the following.
  
  \begin{lem}
   The map $(F_n)_n \colon \JJ((X_n \times G)_n, E)\rightarrow \JJ((\Sigma_n \times G)_n, E)$ is a morphism of coarse structures.
  \end{lem}

  \begin{proof}
    By \Cref{lem:functoriality-of-obstruction-category}, it suffices to check
    that $(F_n)_n$ is a uniformly expanding sequence.  Since the the proof is
    fairly lengthy (though still straightforward), we give the details.

    Let us recall the definitions.  The metric $d_n$ on $X_n \times G$ was
    defined in \Cref{def:lambda-delta}, the metric $n\cdot d^{\ell^1} + d_G$ on
    $\Sigma_n \times G$ was chosen in \Cref{def:metric-on-Sigma-G}.  Let
    $\alpha > 0$.

    Let $n \in \NN$ and $(x,g)$, $(y,h) \in X_n \times G$. Suppose that
    $d_n((x,g),(y,h)) < \alpha$. To prove that $F_n$ is uniformly expanding we
    have to show that $n \cdot d^{\ell^1}(gf_n(x),hf_n(y)) + d_G(g,h) \leq
    \beta$ for some $\beta > 0$ which is independent of $n$.  
    
    In fact, it suffices to show this for $n \geq \alpha$.  Then we have by
    \Cref{def:lambda-delta}
    \[
      \frac{\alpha}{\Lambda_n} = \frac{\alpha \delta_n}{n+1} \leq
      \frac{n\delta_n}{n+1} < \delta_n.
    \]

   By definition of $d_n$, we have
   $d_G(g,h) < \alpha$  and $d_{S^n,n,\Lambda_n}((x,g),(y,h)) < \alpha$.
   Hence, there exist 
   \begin{itemize}
    \item $l \in \NN$
    \item $x_0,\dots,x_l$, $z_0,\dots,z_l \in X_n$
    \item $a_1,\dots,a_l$, $b_1,\dots,b_l \in S^n$
   \end{itemize}
   such that
   \begin{enumerate}
    \item $x_0 = x$, $z_l = y$
    \item $g a_1^{-1} b_1 \dots a_l^{-1} b_l = h$
    \item for each $1 \leq i \leq l$ there are elements
      $r_0,\dots,r_n,s_0,\dots,s_n \in S^n$ such that $a_i = r_n\dots r_0$,
      $b_i = s_n \dots s_0$ and, for some $t_1, \dots,t_n,u_1,\dots,u_n \in
      [0,1]$, $\Gamma_n(r_n,t_n,\dots,t_0,z_{i-1}) =
      \Gamma_n(s_n,u_n,\dots,s_0,x_i)$ holds
    \item $l + \sum_{i=0}^l \Lambda_n \cdot d_{X_n}(x_i,z_i) < \alpha$
   \end{enumerate}
   This implies $l < \alpha$, $d_{X_n}(x_i, z_i) < \frac{\alpha}{\Lambda_n} <
   \delta_n$.  By \Cref{def:lambda-delta} of $\delta_n$, this implies
   $d^{\ell^1}(f_n(x_i),f_n(z_i)) < \frac{1}{n}$.

   We proceed by induction on $l$. For $l = 0$, we have $g=h$ and
   $d^{\ell^1}(f_n(x),f_n(y)) < \frac{1}{n}$.
   For the induction step, with $a_1 = r_n\dots r_0$ and $b_1 = s_n \dots s_0$
   we have
   \begin{equation*}
   \begin{split}
    d^{\ell^1}(g f_n(x), h f_n(y))
    &= d^{\ell^1}(g f_n(x), g a_1^{-1} b_1 \dots a_l^{-1} b_l f_n(y) ) \\
    &= d^{\ell^1}(f_n(x), a_1^{-1} b_1 \dots a_l^{-1} b_l f_n(y) ) \\
    &\leq d^{\ell^1}(f_n(x_0),f_n(z_0)) \\
    &\quad + d^{\ell^1}(f_n(z_0), a_1^{-1} f_n(\Gamma_n(r_n,1,\dots,1,r_0,z_0)) ) \\
    &\quad + d^{\ell^1}( f_n(\Gamma_n(r_n,1,\dots,1,r_0,z_0)) , f_n(\Gamma_n(r_n,t_n,\dots,t_1,r_0,z_0)) ) \\
    &\quad + d^{\ell^1}( f_n(\Gamma_n(s_n,u_n,\dots,u_1,s_0,x_1)), f_n(\Gamma_n(s_n,1,\dots,1,s_0,x_1)) ) \\
    &\quad + d^{\ell^1}( f_n(\Gamma_n(s_n,1,\dots,1,s_0,x_1)), b_1 f_n(x_1) ) \\
    &\quad + d^{\ell^1}( f_n(x_1), f_n(z_1) ) \\
    &\quad + d^{\ell^1}( f_n(z_1), a_2^{-1} b_2 \dots a_l^{-1} b_l f_n(y) )
   \end{split}
   \end{equation*}
   We give an estimate for each summand. We already know
   \begin{equation*}
    d^{\ell^1}(f_n(x_0),f_n(z_0)) < \frac{1}{n}, \qquad
    d^{\ell^1}( f_n(x_1), f_n(z_1) )< \frac{1}{n}.
   \end{equation*}
   For the second summand, we have by
   \Cref{eq:homotopy-trans-reduc:almost-equivariant}
   \begin{equation*}
   \begin{split}
    d^{\ell^1}(f_n(z_0), a_1^{-1} f_n(\Gamma_n(r_n,1,\dots,1,r_0,z_0)) ) 
    &= d^{\ell^1}(a_1 \cdot f_n(z_0), f_n(\Gamma_n(a_1,z_0))) 
    \leq \frac{1}{n},
   \end{split}
   \end{equation*}
   similarly for $d^{\ell^1}( f(\Gamma_n(s_n,1,\dots,1,s_0,x_1)), b_1 f(x_1) )$.
   Furthermore, we have
   \begin{equation*}
   \begin{split}
    d^{\ell^1}( f_n(\Gamma_n(r_n,1,\dots,1,r_0,z_0)) , f_n(\Gamma_n(r_n,t_n,\dots,t_1,r_0,z_0)) ) &\leq \frac{2}{n} \\
    d^{\ell^1}( f_n(\Gamma_n(s_n,u_n,\dots,u_1,s_0,x_1)), f_n(\Gamma_n(s_n,1,\dots,1,s_0,x_1)) ) &\leq \frac{2}{n} 
   \end{split}  
   \end{equation*}
   by \Cref{eq:homotopy-trans-reduc:diameter-bounded}. Finally, we choose the
   induction hypothesis to be
   \[
   d^{\ell^1}( f_n(z_1), a_2^{-1} b_2 \dots a_l^{-1} b_l f_n(y) ) <
   \frac{8(l-1) + 1}{n}.
   \]
   Thus we obtain 
   \[
   d^{\ell^1}(g f_n(x), h f_n(y))<\frac{8l + 1}{n}.
   \]
   Since we also have $d_G(g,h) < \alpha$, we conclude that
   \begin{equation*}
    n \cdot d^{\ell^1}(gf_n(x),hf_n(y)) + d_G(g,h) < 9\alpha + 1.\qedhere
   \end{equation*}
  \end{proof}
  
\subsection{Injectivity of the $\Delta$-map}
\label{subsec:proof-fj:injectivity-delta}

  Now we show \Cref{prop:main-diag-claims}~\ref{item:main-diag-claims-4}. Namely, we have to show that $\Delta$ induces an injective map on algebraic $K$--theory. Our argument is a straightforward adaptation of the argument used in \cite[Section~5]{Wegner(2012)}.
  As usual, we abbreviate $\pi_m(\KK^{-\infty}(\dots))$ by $K_m(\dots)$.
  \begin{lem}
   The map 
   \[
    K_m(\Delta)\circ K_m(incl)\colon K_m(\cR^G_f(W, \JJ(E)),h)\rightarrow K_m(\cR_{fd}^G(W, \JJ((G)_n, E)), h^{fin})
   \]
   is injective for each $m \geq 0$.
  \end{lem}
  \begin{proof}
   The map $K_m(incl)\colon K_m(\cR^G_f(W, \JJ(E)),h)\rightarrow
   K_m(\cR^G_{fd}(W, \JJ(E)),h)$ is an isomorphism by
   \cite[Remark~5.5]{Ullmann-Winges(2015)}. Hence, we only have to show that
   $K_m(\Delta)$ is injective. To increase readability, we shorten
   $\mathcal{R}_{fd}^G(W,...)$ to $\mathcal{R}(...)$ in the following commutative diagram:
   \[
    \begin{tikzpicture}
    \matrix (m) [matrix of math nodes, column sep=2em, row sep=2em, text depth=.5em, text height=1em, ampersand replacement=\&]
    {\& K_m(\cR(\JJ(E)),h)\\K_m(\prod^{fin}\cR( \JJ(E)),h) \& K_m(\cR(\JJ((G)_n, E)), h) \& K_m(\cR(\JJ((G)_n, E)), h^{fin}) \\ \bigoplus_{n\in\NN}K_m(\cR(\JJ(E)),h) \& \prod_{n\in\NN}K_m(\cR(\JJ(E)),h) \\};
    \path[->]
    (m-1-2) edge node[right]{$\Delta_\ast$} (m-2-2)
    (m-1-2) edge node[above]{$\Delta_\ast$} (m-2-3)
    (m-2-1) edge node[right]{$\cong$} (m-3-1)
    (m-2-2) edge node[right]{$\prod_{n\in\NN}p_n$} (m-3-2)
    (m-2-2) edge node[below]{$\id_\ast$} (m-2-3)
    (m-3-1) edge node[above]{$\incl$} (m-3-2)
    (m-2-1) edge node[above]{$\incl$} (m-2-2);
    \end{tikzpicture}
   \]
   The middle row is exact due to the homotopy fiber sequence
   \eqref{eq:homotopy-fiber-sequence}, where we abbreviated $ \colim_n
   \prod_{k=1}^n $ as $\prod^{fin}$. The left vertical map is an isomorphism,
   because algebraic $K$--theory commutes with directed colimits and is
   compatible with finite products.  The map is defined using the projections
   onto the factors of the product.  Note that after projection on any $n$, the
   middle column is the identity.  A diagram chase finishes the proof.
  \end{proof}
  
\subsection{Homotopy transfer reducible follows from strongly transfer reducible}
\label{subsec:proof-fj:homotopy-transfer-follows-strongly-transfer}
  We can now prove \Cref{the:main_result}~\ref{the:main_result:groups} for hyperbolic and $\mathrm{CAT}(0)$--groups.
  
  In \cite[Definition~3.1]{Wegner(2012)}, Wegner defined when a group $G$ is
  \emph{strongly transfer reducible over a family $\cF$}.  As we will not need
  the precise definition here, we refer to \emph{loc.~cit.} for the
  definition.  We will use the definitions from
  \Cref{subsec:proof-fj:homotopy-coherent}.
  \begin{thm}\label{thm:strongly-trans-reducible}
   Let $G$ be strongly transfer reducible over $\cF$. 
   Then $G$ is homotopy transfer reducible over $\cF$ and 
   the Farrell-Jones \Cref{con:FJC_for_A-theory_with_c} for $A$--theory
   with coefficients holds for $G$ relative to $\cF$.
 \end{thm}
  \begin{proof}
    Assume that $G$ is strongly transfer reducible.  We show it is homotopy
    transfer reducible and apply \Cref{thm:afjc-transfer-reducible}.

   According to \cite[Proposition~3.6]{Wegner(2012)}, there exists $N \in \NN$ such that there are for every $n \in \NN$
    \begin{enumerate}
     \item a compact, contractible metric space $(X,d_X)$ such that for every $\epsilon > 0$ there is an $\epsilon$--controlled domination of $X$ by an at most $N$--dimensional, finite simplicial complex;
     \item a homotopy coherent $G$--action $\Gamma$ on $X$;
     \item a $G$--simplicial complex $\Sigma$ of dimension at most $N$ whose isotropy is contained in $\cF$;
     \item a positive real number $\Lambda$;
     \item a $G$--equivariant map $\phi \colon G \times X \to \Sigma$ such that
     \[
     n \cdot d^{\ell^1}(\phi(g,x), \phi(h,y)) \leq d_{S^n,n,\Lambda}((g,x),(h,y))
     \]
     holds for all $(g,x)$, $(h,y) \in G \times X$, where $G$ acts on the
     $G$--factor. 
    \end{enumerate}
   Fix $n \in \NN$ and choose $X$, $\Gamma$, $\Sigma$, $\Lambda$ and $\phi$ as
   above. Define $f := \phi|_{\{e\} \times X} \colon X \to \Sigma$. Then we
   have for all $x \in X$ and $s \in S^n$ \begin{equation*}
     \begin{split}
      n \cdot d^{\ell^1}(f(\Gamma(s,x)), s \cdot f(x))
      &= n \cdot d^{\ell^1}(\phi(e,\Gamma(s,x)),\phi(s,x)) \\
      &\leq d_{S^n,n,\Lambda}((e,\Gamma(s,x)),(s,x)) \\
      &\leq 1.
     \end{split}
    \end{equation*}
   Similarly, we find for all $x \in X$, $s_0,\dots,s_n \in S^n$ and $t_1,\dots,t_n,u_1,\dots,u_n \in I^n$
    \begin{equation*}
     \begin{split}
      n \cdot d^{\ell^1}&(f(\Gamma(s_n,t_n,\dots,s_0,x)), f(\Gamma(s_n,u_n,\dots,s_0,x)) ) \\
      &= n \cdot d^{\ell^1}(\phi(e,\Gamma(s_n,t_n,\dots,s_0,x)),\phi(e,\Gamma(s_n,u_n,\dots,s_0,x)) ) \\
      &\leq d_{S^n,n,\Lambda}((e,\Gamma(s_n,t_n,\dots,s_0,x)),(e,\Gamma(s_n,u_n,\dots,s_0,x)) ) \\
      &\leq 2.
     \end{split}
    \end{equation*}
   Hence, $G$ is homotopy transfer reducible over $\cF$ and we can apply \Cref{thm:afjc-transfer-reducible}.
 \end{proof}

   \begin{cor}\label{cor:hypcatfollowproof}
     The Farrell-Jones \Cref{con:FJC_for_A-theory_with_c_and_fwp} for
     $A$--theory with coefficients and finite wreath products is true for
     hyperbolic and $\mathrm{CAT}(0)$--groups.
   \end{cor}
   \begin{proof}
   By \cite[Example~3.2 and Theorem~3.4]{Wegner(2012)}, finitely generated
   hyperbolic groups as well as $\mathrm{CAT}(0)$--groups are strongly
   transfer reducible with respect to the family of virtually cyclic
   subgroups. Thus these groups are homotopy
   transfer reducible over the same family by
   \Cref{thm:strongly-trans-reducible}.  If a group is homotopy transfer
   reducible with respect to $\cF$, then the wreath product $G \wr F$ with a
   finite group $F$ is homotopy transfer reducible over $\cF \wr F$.  This
   follows as in \cite[Section 5]{Bartels-Lueck-Reich-Rueping(2014)};
   basically, one takes the $F$--fold product of $X$ and $\Sigma$ and uses e.g.
   \cite[Lemma 11.14]{Ullmann-Winges(2015)} for the estimate.  As the
   $A$--theoretic Farrell--Jones Conjecture with coefficients holds for
   virtually finitely generated abelian groups
   \cite[Proposition~11.9]{Ullmann-Winges(2015)}, the $A$--theoretic
   Farrell--Jones Conjecture with coefficients and finite wreath products
   holds for hyperbolic and $\mathrm{CAT}(0)$--groups by the
   Transitivity Principle, \cite[Proposition~11.2]{Ullmann-Winges(2015)}.
%
  \end{proof}


\section{The transfer: Final part of the proof}\label{The Transfer: Final Part of the Proof} 
  
  We turn now to the construction of the transfer map whose existence was claimed in the first part of \Cref{prop:main-diag-claims}~\ref{item:main-diag-claims-1}. This map will be induced by a ``transfer'' construction on controlled retractive spaces. Here we employ an analog of the classical construction of the transfer, see \cite{Bartels-Lueck(2012annals)}, \cite{Bartels-Lueck-Reich(2008hyper)} and \cite{Wegner(2012)}, on each cell 
  and glue these together according to the CW--structure.  Unfortunately, this construction is not fully functorial on $\cR^G_f(W,\JJ(E))$. To avoid this problem, we restrict our attention to subcategories of ``cellwise $0$--controlled morphisms'' (defined below) as a domain for the transfer. The idea to use these categories is due to Arthur Bartels and Paul Bubenzer.
  
  The transfer is defined using the ideas of the linear counterpart.
  However, they only work well for one cell at a time.  To extend, we need to
  refine the idea of crossing a controlled CW-complex with the singular
  complex of a metric space, by allowing different, but compatible singular
  complexes for each cell.  While this provides us with a transfer on objects,
  only ``cellwise $0$-controlled morphisms'' behave well with respect to
  this construction.  We could also transfer any map, but then the target
  gets a more lax control condition, and this makes the construction
  non-functorial.

  We show that transferring only ``cellwise $0$-controlled morphisms'' is
  enough to construct the transfer, but transferring the other morphisms is
  needed to show that it preserves weak equivalences.
  

\subsection{The domain of the transfer}

  We now define the appropriate subcategories of cellwise $0$--controlled morphisms, which will serve as the source of the transfer.

Let $M$ be a metric space with a free, isometric $G$--action, and consider the category $\cR^G(W,\JJ(M,E))$. For $(Y,\kappa) \in \cR^G(W,\JJ(M,E))$, let $\kappa_M$ denote the composition of the control map $\kappa$ with the projection map $M \times E \times [1,\infty[ \to M$.

\begin{dfn}\label{def:0-controlled-morphism}
 Let $f \colon (Y_1,\kappa_1) \to (Y_2,\kappa_2)$ be a morphism in
 $\cR^G_f(W,\JJ(M,E))$.  We say that $f$ is \emph{regular}, if the image
 of each open cell in $Y_1$ is either equal to an open cell in $Y_2$ or completely
 contained in $W$.  That is, either $f(\mathrm{int}\ e)= \mathrm{int}\ e'$ or
 $f(\mathrm{int}\ e) \subseteq W$.
 
 We say that $f$ is \emph{cellwise $0$--controlled over $M$} if $f$ is regular and satisfies the property that $\kappa_{1,M}(e) = \kappa_{2,M}(f(e))$ for all cells $e \in \cells Y_1$.
\end{dfn}

The composition of two morphisms which are cellwise $0$--controlled over $M$ is again cellwise $0$--controlled over $M$, so we can consider the subcategory
\begin{equation*}
 \cR^G_f(W,\JJ(M,E))_0 \subset \cR^G_f(W,\JJ(M,E)) 
\end{equation*}
which has the same objects as $\cR^G_f(W,\JJ(M,E))$, but contains only those
morphisms which are cellwise $0$--controlled over $M$. The category
$\cR^G_f(W,\JJ(M,E))_0$ inherits cofibrations and weak equivalences from $\cR^G_f(W,\JJ(M,E))$. It is a Waldhausen subcategory of $\cR^G_f(W,\JJ(M,E))$. 

For $\alpha > 0$, we may further restrict to the full subcategory
\begin{equation*}
 \cR^G_f(W,\JJ(M,E))_\alpha \subset \cR^G_f(W,\JJ(M,E))_0 
\end{equation*}
consisting only of those objects which are \emph{$\alpha$--controlled over
$M$}, i.e., those $(Y,\kappa)$ such that $\kappa_M(\cells \gen{e}) \subset
B_\alpha(\kappa_M(e))$ for every cell $e \in \cells Y$.  (Recall that
$\gen{e}$ denotes the smallest subcomplex of $Y$ containing $e$.)  The
category $\cR^G_f(W,\JJ(M,E))_\alpha$ also inherits the structure of a
Waldhausen category, as the pushout of $\alpha$--controlled
complexes along cellwise $0$--controlled morphisms is again
$\alpha$--controlled.

Finally, we can filter $\cR^G_f(W,\JJ(M,E))_\alpha$ by
\begin{equation*}
 \cR^G_f(W,\JJ(M,E))_{\alpha,0} \subset \cR^G_f(W,\JJ(M,E))_{\alpha,1} \subset \dots \subset \cR^G_f(W,\JJ(M,E))_\alpha, 
\end{equation*}
where $\cR^G_f(W,\JJ(M,E))_{\alpha,d}$ denotes the full subcategory of $\cR^G_f(W,\JJ(M,E))_\alpha$ containing those objects whose dimension is at most $d$. Note that
\begin{equation*}
 \cR^G_f(W,\JJ(M,E))_\alpha = \colim_d \cR^G_f(W,\JJ(M,E))_{\alpha,d},
\end{equation*}
as each object in $\cR^G_f(W,\JJ(M,E))$ is finite-dimensional.

\begin{prop}\label{prop:domain-of-transfer}
 There is a natural weak equivalence
 \begin{equation*}
  \hocolim_{\alpha,d} K(\cR^G_f(W,\JJ(M,E))_{\alpha,d}) \xrightarrow{\sim} K(\cR^G_f(W,\JJ(M,E))).
 \end{equation*}
\end{prop}
\begin{proof}
 We have $\cR^G_f(W,\JJ(M,E))_0 = \colim_{\alpha,d} \cR^G_f(W,\JJ(M,E))_{\alpha,d}$. Since $K$--theory commutes with directed colimits, we obtain a natural weak equivalence
 \begin{equation*}
  \hocolim_{\alpha,d} K(\cR^G_f(W,\JJ(M,E))_{\alpha,d}) \xrightarrow{\sim} K(\cR^G_f(W,\JJ(M,E))_0).
 \end{equation*}
 Now consider the inclusion functor $\cR^G_f(W,\JJ(M,E))_0 \hookrightarrow
 \cR^G_f(W,\JJ(M,E))$. We show that Waldhausen's Approximation Theorem
 \cite[Theorem~1.6.7]{Waldhausen(1985)} applies. 
 
 The cylinder functor on $\cR^G_f(W,\JJ(M,E))$ constructed in 
 \cite[Lemma~3.14]{Ullmann-Winges(2015)} restricts to a cylinder
 functor on $\cR^G_f(W,\JJ(M,E))_0$, in particular the inclusion of the
 source is always cellwise $0$-controlled. By definition, the inclusion functor
 satisfies the first part of the approximation property. To verify the second
 part of the approximation property, let $f \colon Y_1 \to Y_2$ be an
 arbitrary morphism in $\cR^G_f(W,\JJ(M,E))$. Then the factorization of $f$
 via the cylinder functor $Y_1 \rightarrowtail Mf \xrightarrow{\sim} Y_2$
 decomposes $f$ into
 a cellwise $0$--controlled morphism and a weak equivalence. So the
 Approximation Theorem implies that the inclusion functor induces an
 equivalence on algebraic $K$--theory.
\end{proof}

\begin{rem}
 The upshot of \Cref{prop:domain-of-transfer} is that we do not have to define a ``global'' transfer functor on $\cR^G_f(W,\JJ(E))$. Instead, it suffices to define a transfer functor $\trans^{\alpha, d}\colon \cR^G_f(W,\JJ(E))_{\alpha,d}\rightarrow (\cR^G_{fd}(W, \JJ((X_n\times G)_n, E)), h^{fin})$  on each subcategory, such that the induced diagrams on $K$--theory
 \begin{equation*}\label{eq:hocomm}
  \begin{tikzpicture}
  \matrix (m) [matrix of math nodes, column sep=4.5em, row sep=2em, text depth=.5em, text height=1em]
  {K(\cR^G_f(W,\JJ(E))_{\alpha,d}, h)& K(\cR^G_{fd}(W, \JJ((X_n\times G)_n, E)), h^{fin}) \\ K(\cR^G_f(W,\JJ(E))_{\alpha+1,d+1}, h)\\};
  \path[->]
  (m-1-1) edge node[above]{$K(\trans^{\alpha,d})$} (m-1-2)
  (m-1-1) edge (m-2-1)
  (m-2-1) edge node[below right]{$K(\trans^{\alpha+1,d+1})$} (m-1-2);
 \end{tikzpicture}
 \end{equation*}
 are homotopy commutative.
\end{rem}


\subsection{Balanced products of CW--complexes}
\label{sec:balanced-products}

We introduce a slight generalization of the balanced products discussed in
\cite{Davis-Lueck(1998)} as a means to define the transfer in
\Cref{subsec:transfer-on-objects}. 

Let $W$ be a topological space and $\cC$ a small category. 
A \emph{$\cC$--CW--complex relative $W$} is a functor $Y$ from $\cC$ to
topological spaces such that $Y(c)$ is a CW--complex relative $W$ and the
morphisms in $\cC$ are mapped to cellular maps relative $W$.  A (relative)
free $\cC$--$n$--cell based at $c$, $c \in \cC$, is a pair $(\eta, \partial
\eta)$ of $\cC$--CW--complexes relative $W$, where $\eta = \cC(c,-) \times
D^{n} \amalg W$, $\partial \eta =\cC(c,-) \times S^{n-1} \amalg W$.  Attaching
a free $\cC$--cell $\eta$ to $Y$ means taking the pushout along a map
$\partial\eta \to Y$.  Note that $W$ itself defines a (constant) covariant
$\cC$--CW--complex relative $W$.

We say that $Y$ is a \emph{free} $\cC$--CW--complex relative $W$ if it comes
equipped with a filtration $W = \skel{-1}{Y} \subset \skel{0}{Y} \subset
\skel{1}{Y} \subset \dots$ such that $Y = \colim_n \skel{n}{Y}$ and for every
$n \geq 0$ 
there exists a pushout in the category of $\cC$--CW--complexes relative $W$
\begin{equation*}
 \begin{tikzpicture}
  \matrix (m) [matrix of math nodes, column sep=3em, row sep=2em, text depth=.5em, text height=1em]
  {\left(\coprod_{i \in I_n} \cC(c_i,-) \times S^{n-1}\right)  \amalg W &
  \skel{n-1}{Y}(-) \\
  \left( \coprod_{i \in I_n} \cC(c_i,-) \times D^n \right)  
  \amalg W& \skel{n}{Y}(-). \\};
  \path[->]
  (m-1-1) edge (m-1-2)
  (m-1-1) edge (m-2-1)
  (m-1-2) edge (m-2-2)
  (m-2-1) edge (m-2-2);
 \end{tikzpicture}
\end{equation*}
Hence, a free $\cC$--CW--complex arises by attaching free $\cC$--cells. The set of free
$\cC$--$n$--cells is in bijection with $I_n$.  Note that the attaching map of a
$\cC$--$n$--cell based at $c$ is the same as a map $S^{n-1} \to
\skel{n-1}{Y}(c)$, hence we can consider $\eta$ as a map $D^n \to Y(c)$.

Let $Y$ be a covariant $\cC$--CW--complex relative $W$ and $X \colon \cC^{op}
\to \CWCOMPLEXES$ be a contravariant $\cC$--CW--complex. Define the
\emph{reduced balanced product} $X  \rightthreetimes_\cC Y$ as
the pushout
\begin{equation*}
 \begin{tikzpicture}
  \matrix (m) [matrix of math nodes, column sep=2em, row sep=2em, text depth=.5em, text height=1em]
  {X \times_\cC W & X \times_\cC Y \\ * \times_\cC W \cong W & X
    \rightthreetimes_\cC Y. \\};
  \path[->]
  (m-1-1) edge (m-1-2)
  (m-1-1) edge (m-2-1)
  (m-1-2) edge (m-2-2)
  (m-2-1) edge (m-2-2);
 \end{tikzpicture} 
\end{equation*}
\begin{prop}
  Let $X$ be a contravariant $\cC$--space, $Y$ a covariant $\cC$--CW--complex
  relative $W$ and $Z$ a space relative $W$.  There is a natural homeomorphism
  \begin{equation*}
    \hom^W(X \rightthreetimes_{\cC} Y, Z) \cong
    \hom_{\cC}^W(Y, \hom(X,Z)).
  \end{equation*}
  Here, $\hom(X,Z)$ is a covariant $\cC$--space  relative $W$ via the inclusion that sends a
  point $w \in W(c)$ to the constant map $X(c) \to \{w\} \subseteq Z$, 
  $\hom^W_{\cC}$ denotes the natural transformations which are
  relative $W$, 
  and $\hom^W$ denotes just the set
  of maps relative $W$.
\end{prop}
\begin{proof}
  By definition, a map $X \rightthreetimes_{\cC} Y \to Z$ is the same as
  three compatible maps from $W \leftarrow X \times_{\cC} W \to X \times_{\cC}
  Y$ to $Z$.  Using that $\hom(X\times_{\cC}Y, Z)$ is isomorphic to
  $\hom_{\cC}(Y, \hom(X,Z))$, the result is easy to deduce.
\end{proof}

It follows, that $X \rightthreetimes_{\cC} Y$ commutes with
colimits in the ``$Y$''-variable.  We can therefore determine the cell
structure of $X \rightthreetimes_{\cC} Y$.  The attachment of a free
$\cC$--$n$--cell $\eta$ to $Y$ gives a pushout $X \rightthreetimes_{\cC}
\eta(-) \cup_{X \rightthreetimes_{\cC} \partial \eta} Y$.  Now
\begin{equation*}
  X \rightthreetimes_{\cC} \eta \cong 
  ((X \times_{\cC} \cC(c,-) )\times D^n) \amalg W \cong
  (X(c) \times D^n) \amalg W
\end{equation*}
and similarly for $\partial \eta$.  First, this gives a filtration on $X
\rightthreetimes_{\cC} Y$, namely 
\begin{equation}
  \ldots \subseteq X \rightthreetimes_{\cC} \skel{n-1}{Y} \subseteq X
  \rightthreetimes_{\cC} \skel{n}{Y} \subseteq \ldots
  \label{eq:balance_product_skeletal_filtration}
\end{equation}
Second, as $X(c)$ is a CW--complex, we can now read off the cell structure of $X
\rightthreetimes_{\cC} Y$:

\begin{prop}[{cf.~\cite[Lemma~3.19(2)]{Davis-Lueck(1998)}}]\label{prop:balanced-product-cells}
 Let $Y$ be a free covariant $\cC$--CW--complex relative $W$ and $X$ a contravariant $\cC$--CW--complex.
 
 Then $X \rightthreetimes_\cC Y$ is a CW--complex relative $W$, and there is a canonical identification
 \begin{equation*}
  \cells (X \rightthreetimes_\cC Y) \cong \{ (\xi,\eta) \mid \eta \text{ is a free $\cC$--cell based at $c$}, \xi \in \cells X(c) \}.
 \end{equation*}
 Let $(\xi,\eta) \in \cells (X \rightthreetimes_\cC Y)$. If $\Phi \colon D^p \to X(c)$ and $\Psi \colon \cC(c,-) \times D^q \to Y(-)$ are characteristic maps for $\xi$ and $\eta$, respectively, then
 \begin{equation*}
  D^p \times D^q \to X \rightthreetimes_\cC Y, \quad (a,b) \mapsto [\Phi(a),\Psi(\id_c,b)]
 \end{equation*}
 is a characteristic map for $(\xi,\eta)$.
 
 Let $(\xi,\eta)$, $(\xi',\eta') \in \cells (X \rightthreetimes_\cC Y)$ be
 two cells, with $\eta$ based at $c$ and $\eta'$ based at $c'$. Then
 $(\xi,\eta) \subset \gen{(\xi',\eta')}$ if and only if there exists a
 morphism $\gamma \colon c \to c'$ such that $\gamma_* \eta \subset
 \gen{\eta'} \subset Y(c')$ and $\xi \subset \gen{\gamma^*\xi'}
 \subset X(c)$.
\end{prop}

In greater generality, \eqref{eq:balance_product_skeletal_filtration} gives a
filtration for an inclusion $Y_1 \hookrightarrow Y_2$ of $\cC$--spaces in which $Y_2$ is obtained from $Y_1$ by the attachment of free $\cC$--cells. This observation allows us to translate the constructions for geometric modules to CW--complexes.

\medskip

Let us conclude this section with a short remark about functoriality of the balanced product construction. In addition to the obvious functoriality properties, we have the following: Let $X$ be a contravariant and $Y$ be a covariant $\cC$--space. Let $F \colon \cD \to \cC$ be a functor. Then there is an induced map
\begin{equation*}
 \iota_F \colon F^*X \times_\cD F^*Y \to X \times_\cC Y, \quad [d,x,y] \mapsto [F(d),x,y].
\end{equation*}
This map is functorial in the sense that $\iota_{F_2}\iota_{F_1} =
\iota_{F_2F_1}$ for any two composable functors $F_1$, $F_2$. In particular,
if $F \colon \cC \xrightarrow{\cong} \cC$ is an automorphism of the indexing
category, then $\iota_F$ is an isomorphism.


\subsection{Conventions}\label{subsec:transfers}

For the following sections, fix the following data:
\begin{enumerate}
 \item natural numbers $\alpha,d \in \NN$ and a natural number $n > \max \{ d+1, \alpha \}$
 \item a natural number $N \in \NN$
 \item a compact and contractible metric space $(X,d_X)$ such that for every $\epsilon > 0$ there is an $\epsilon$--controlled domination of $X$ by an at most $N$--dimensional, finite simplicial complex
 \item a homotopy coherent $G$--action $\Gamma$ on $X$
 \item a positive real number $\Lambda$
\end{enumerate}
As before, we consider $X \times G$ equipped with the metric $d_{S^n,n,\Lambda} + d_G$.  


\subsection{$Y$ as a $\diamond_+ Y$--CW--complex and a $(\diamond_+
  Y)^{op}$--CW--complex from $X$}
\label{subsec:y-spaces}
Let $(Y, \kappa) \in \cR^G_f(W,\JJ(E))_{\alpha,d}$. If $c \in \cells Y$ is a cell of $Y$, we will frequently need to refer to the $G$--component of $\kappa(c)$; we denote this by $\kappa_G(c)$.

Define a relation $\leq$ on the set of cells $\cells Y$ by saying that $c \leq
c'$ if and only if $c \subset \gen{c'}$. Then $\cells Y$ forms a poset under
the relation $\leq$.  We define $\cellsPlus Y$ as the category given by this
poset where we add an additional initial object (which corresponds to $W$).
The complex $Y$ itself gives rise to a covariant $\cellsPlus Y$--CW--complex (relative $W$) $\cC_Y$ by setting
\begin{equation*}
 \cC_Y(c) := \gen{c}
\end{equation*}
and sending a morphism $c \leq c'$ to the obvious inclusion $\gen{c}
\hookrightarrow \gen{c'}$. Observe that $\cC_Y$ is a free $\cellsPlus
Y$--CW--complex; the set of free $\cellsPlus Y$--cells of $\cC_Y$ is in canonical
bijection with the cells of $Y$.  Note that a cellwise
$0$-controlled map $Y \to Y'$ gives rise to a functor $\cellsPlus Y \to \cellsPlus
Y'$.  Last, each cell in $\cellsPlus Y$ has a dimension $|c|$,
where we assign the initial object the dimension $-1$.

The metric space $X$ gives rise to a contravariant $\cellsPlus Y$--CW--complex, but the construction is more involved. We mimic the construction used in \cite{Wegner(2012)}, but do not pass to the cellular chain complex. Instead, we simply stick with the space of controlled simplices. 

In the first step, we pass from the homotopy coherent $G$--action $\Gamma$ on $X$ to an honest $G$--action on a closely related space. This is accomplished by strictifying the homotopy coherent diagram $\Gamma$, see \cite[proof of Proposition~5.4]{Vogt(1973)}. Define $M\Gamma$ to be the space
\begin{equation*}
 M\Gamma := \left( \coprod_{k \geq 0} 
   G^{k+1} \times [0,1]^k \times X \right)\big/ \sim,
\end{equation*}
where $\sim$ is the equivalence relation generated by
\begin{equation*}
  (\gamma_{k+1}, t_k, \gamma_k,\dots,\gamma_1,x)
   \sim 
  \begin{cases}
   (\gamma_{k+1},t_k,\dots,\gamma_2,x) & \gamma_1 = e \\
   (\gamma_{k+1},\dots,t_it_{i-1},\dots,\gamma_1,x) & \gamma_i = e, 2 \leq i \leq k \\
                              (\gamma_{k+1},\dots,\gamma_{i+1}\gamma_i,\dots,\gamma_1,x)
                              & t_i = 1, 1 \leq i \leq k \\
                              (\gamma_{k+1},\dots,\gamma_{i+1},
    \Gamma(\gamma_i,\dots,\gamma_1,x)) & t_i = 0, 1
                              \leq i \leq k.
                             \end{cases}
\end{equation*}
Then $G$ acts on $M\Gamma$ by
\begin{equation*}
 g \cdot [\gamma,t_k,\gamma_k,\dots,\gamma_1,x] := [g \gamma, t_k, \gamma_k, \dots, \gamma_1,x].
\end{equation*}


We have a map $X \to M \Gamma$ via $x \mapsto [e,x]$.  Let 
$R \colon M\Gamma \to X$ be the retraction induced by $\Gamma$; explicitly,
$R([\gamma,t_k,\gamma_k,\dots,\gamma_1,x]) =
\Gamma(\gamma,t_k,\gamma_k,\dots,\gamma_1,x)$. Using the axioms of a homotopy
action from~\Cref{def:hptycohG-action}, one checks this is a well-defined map.
The homotopy
\begin{equation}\label{eq:strong-deformation-rectraction-MGamma}
 H \colon M\Gamma \times [0,1] \to M\Gamma, \quad ([\gamma,t_k,\gamma_k,\dots,\gamma_1,x],u) \mapsto [e,u,\gamma,t_k,\gamma_k,\dots,\gamma_1,x]
\end{equation}
then shows that $X$ is a strong deformation retract of $M\Gamma$.


The space $M\Gamma$ comes with a filtration by subspaces $M\Gamma^{l,r}$,
where we set
\begin{equation*}
  M\Gamma^{l,r} := \{ [e, t_k, \gamma_k,\dots,\gamma_1,x]
    \in M\Gamma \mid k \leq l, \gamma_i \in B_{r}(e) \}.
\end{equation*}
For $\delta > 0$, define $S^\delta_\bullet(M\Gamma^{l,r})$ to be the
subsimplicial set of the singular simplicial set
$S_\bullet(M\Gamma^{l,r})$ containing those singular simplices $\sigma
\colon \Delta^{\abs{\sigma}} \to M\Gamma^{l,r}$ which fulfill
\begin{equation*}
 \diam_{X \times G}( (R \circ \sigma)(\Delta^{\abs{\sigma}}) \times  \{
   e
 \} ) \leq \delta,
\end{equation*}
where diameters in $X \times G$ are taken with respect to the metric
$d_{S^n,n,\Lambda}$. Note that we could replace $e$ by any other
group element without changing the diameter, as the metric is $G$-invariant.

Finally, we can define the contravariant $\cellsPlus Y$--CW--complex
$\sing{\alpha,d}{X,Y}$: Let $\delta_k^{d}:=4(d+1-k)$
and $l_k^{d}:=d+1-k$. Typically, we will omit
$d$ from the notation. On objects, we set
\begin{equation*}
 \sing{\alpha,d}{X,Y}(c) := \norm{ S^{\delta_{\abs{c}}}_\bullet(M\Gamma^{l_{\abs{c}},\alpha}) },
\end{equation*}
where $\abs{c}$ denotes the dimension of the cell $c$ and $\norm{-}$ is fat
geometric realization, i.e., the realization after forgetting the
degeneracies.  Note that we have the canonical inclusion $\iota_{c'}$ of
$\sing{\alpha,d}{X,Y}(c')$ into $\norm{ S^{}_\bullet(M\Gamma^l) }$.  The
latter has an honest $G$--action. For a morphism $c' \to c = c \leq
c'$ in $(\cellsPlus Y)^{op}$ 
define $\sing{\alpha,d}{X,Y}(c' \to c)$
  as the factorization of 
\begin{equation*}
 \sing{\alpha,d}{X,Y}(c' \to c) := \kappa_G(c)^{-1}\kappa_G(c') \cdot
 \iota_{c'}(-)
\end{equation*}
over $\iota_{c}$.
We have to check that it is well-defined, i.e., that it actually
factors.  We require the following observation.

\begin{lem}\label{lem:distance-to-translate} 
  Let $[e,t_{b},\dots,\gamma_1,x] \in M\Gamma^{l,\alpha}$. Suppose that $l < n$ and $\alpha\leq n$. Let $h \in B_\alpha(e)$. Then
 \begin{equation*}
   d_{S^n,n,\Lambda}(R([h,t_{b},\dots,\gamma_1,x],g),
   (R([e,t_{b},\dots,\gamma_1,x]),gh) ) \leq 2
 \end{equation*}
 for all $g \in G$.
\end{lem}
\begin{proof}
  Note that $b < n$.
  We use the definition of the metric. Let
  $x_0=z_0=\Gamma(h,t_{b},\gamma_{b},\dots,\gamma_1,x)$,
  $x_1=z_1=x$ and $x_2=z_2=
  \Gamma(e,t_{b},\gamma_{b},\dots,\gamma_1,x)$.
  Furthermore, we set $a_1=e$, $b_1=h\gamma_{b}\dots\gamma_1$,
  $a_2=\gamma_{b}\dots\gamma_1$ and $b_2=e$. Now we can estimate
 \begin{equation*}
  \begin{split}
    d&_{S^n,n,\Lambda}\big(
    (\Gamma(h,t_{b},\gamma_{b},\dots,\gamma_1,x),g),
    (\Gamma(e,t_{b},\gamma_{b},\dots,\gamma_1,x),gh) \big) \\
    &\leq 2 + \Lambda \cdot
    d_X(\Gamma(h,t_{b},\gamma_{b},\dots,\gamma_1,x),\Gamma(h,t_{b},\gamma_{b},\dots,\gamma_1,x)) \\
    &\quad+ \Lambda \cdot d_X(x,x) + \Lambda \cdot
    d_X(\Gamma(e,t_{b},\gamma_{b},\dots,\gamma_1,x),\Gamma(e,t_{b},\gamma_{b},\dots,\gamma_1,x)) \\
   &=2.
  \end{split}
 \end{equation*}
\end{proof}

\begin{cor}
  Assume $ \alpha \leq n$ and $d + 1 < n$.  Then the functor
  $\sing{\alpha,d}{X,Y} \colon (\cellsPlus Y)^{op} \to \mathrm{CW}$ is
  well-defined.
\end{cor}
\begin{proof}
 Since $M\Gamma$ carries an honest $G$--action, functoriality is clear as soon as we have convinced ourselves that $\sing{\alpha,d}{X,Y}$ is well-defined on morphisms.
 Let $c' \to c$ be a morphism in $(\cells Y)^{op}$, and let $\sigma \in S^{\delta_{\abs{c'}}}_\bullet(M\Gamma^{l_{\abs{c'}},\alpha})$.
   We only need to check non-identity morphisms.  Hence we assume $|c| \geq
   |c'| + 1$.
 Let $[e,t_{b},\gamma_{b},\dots,\gamma_1,x]$ be a point in the image of $\sigma$. By definition, we have $b \leq l_{\abs{c'}}$ and $\gamma_i \in B_\alpha(e)$ for all $i$.
 Set $\gamma_{c,c'} := \kappa_G(c)^{-1}\kappa_G(c')$. Note that $\gamma_{c,c'} \in B_\alpha(e)$ since $Y$ is $\alpha$--controlled over $G$.
 Then we obtain
 \begin{equation*}
  \begin{split}
   \gamma_{c,c'} \cdot &[e,t_{b},\gamma_{b},\dots,\gamma_1,x] \\
   &= [\gamma_{c,c'},t_{b},\gamma_{b},\dots,\gamma_1,x] \\
   &= [e,1,\gamma_{c,c'},t_{b},\gamma_{b},\dots,\gamma_1,x] \in M\Gamma^{l_{\abs{c'}+1},\alpha} \subset M\Gamma^{l_{\abs{c}},\alpha}.
  \end{split}
 \end{equation*}
 Hence, $\gamma_{c,c'} \cdot \sigma$ is a singular simplex in $M\Gamma^{l_{\abs{c}},\alpha}$.
 
 Let $[e,t'_{b'},\gamma'_{b'},\dots,\gamma'_1,x']$ be another point in the image of $\sigma$. Then
 \begin{equation*}
  R(\gamma_{c,c'} \cdot [e,t_{b},\gamma_{b},\dots,\gamma_1,x]) = \Gamma(\gamma_{c,c'},t_{b},\gamma_{b},\dots,\gamma_1,x),
 \end{equation*}
 and similarly for $[e,t'_{b'},\gamma'_{b'},\dots,\gamma'_1,x']$. We calculate
 \begin{equation*}
  \begin{split}
   d&_{S^n,n,\Lambda}\big( (\Gamma(\gamma_{c,c'},t_{b},\gamma_{b},\dots,\gamma_1,x),\kappa_G(c)), (\Gamma(\gamma_{c,c'},t'_{b'},\gamma'_{b'},\dots,\gamma'_1,x'),\kappa_G(c)) \big) \\
   &\leq d_{S^n,n,\Lambda}\big( (\Gamma(\gamma_{c,c'},t_{b},\gamma_{b},\dots,\gamma_1,x),\kappa_G(c)), (\Gamma(e,t_{b},\gamma_{b},\dots,\gamma_1,x),\kappa_G(c')) \big) \\
   &\quad + d_{S^n,n,\Lambda}\big( (\Gamma(e,t_{b},\gamma_{b},\dots,\gamma_1,x),\kappa_G(c')), (\Gamma(e,t'_{b'},\gamma'_{b'},\dots,\gamma'_1,x'),\kappa_G(c')) \big) \\
   &\quad + d_{S^n,n,\Lambda}\big( (\Gamma(e,t'_{b'},\gamma'_{b'},\dots,\gamma'_1,x'),\kappa_G(c')), (\Gamma(\gamma_{c,c'},t'_{b'},\gamma'_{b'},\dots,\gamma'_1,x'),\kappa_G(c)) \big) \\
   &\leq 2 + \delta_{\abs{c'}} + 2 \\
   &= 4(d + 1 - (\abs{c'}-1)) \\
   &\leq 4(d + 1 - \abs{c}) = \delta_{\abs{c}},
  \end{split}
 \end{equation*}
 where we used \Cref{lem:distance-to-translate} for the second inequality. This shows that multiplication by $\gamma_{c,c'}$ indeed defines a map
 \begin{equation*}
  \gamma_{c,c'} \cdot - \colon S^{\delta_{\abs{c'}}}_\bullet(M\Gamma^{l_{\abs{c'}},\alpha}) \to S^{\delta_{\abs{c}}}_\bullet(M\Gamma^{l_{\abs{c}},\alpha}).
 \end{equation*}
 So the functor $\sing{\alpha,d}{X,Y}$ is well-defined.
\end{proof}

\subsection{The transfer on objects}
\label{subsec:transfer-on-objects}
Recall that by our assumptions in \Cref{subsec:transfers} we have $\alpha < n,
d+1 < n$.
\begin{dfn}\label{dfn:transfer}
 The \emph{transfer $\trans^{\alpha,d}_X(Y)$ of $Y$ with respect to $X$} is defined to be
 \begin{equation*}
  \trans^{\alpha,d}_X(Y) := \sing{\alpha,d}{X,Y}
  \rightthreetimes_{\cellsPlus Y} \cC_Y.
 \end{equation*}
 If $\alpha$, $d$ or both of them are understood, we abbreviate $\trans^{\alpha,d}_X(Y)$ to $\trans^d_X(Y)$, $\trans^\alpha_X(Y)$ or $\trans_X(Y)$, respectively.
\end{dfn}

Since $\cC_Y$ is a free $\cellsPlus Y$--CW--complex, the space $\trans_X(Y)$
is a CW--complex relative $W$ by \Cref{prop:balanced-product-cells}. The
natural transformation $\sing{\alpha,d}{X,Y} \to *$ to the constant functor
with value the one-point space induces a map $\sing{\alpha,d}{X,Y}
\rightthreetimes_{\cellsPlus Y} \cC_Y \to * \rightthreetimes_{\cellsPlus
Y} \cC_Y \cong Y$ of CW--complexes relative $W$. We regard
$\sing{\alpha,d}{X,Y} \rightthreetimes_{\cellsPlus Y} \cC_Y$ as a retractive space via this map.

We equip $\trans_X(Y)$ with a $G$--action as follows. Observe that $G$ acts on
the indexing category $\cells Y$; let $\mu_g \colon \cells Y \to \cells Y$
denote the functor induced by the action of $g \in G$. The action of $g$ on $Y$ induces a natural isomorphism $\cC_Y \xrightarrow{\cong} \cC_Y \circ \mu_g$, and hence a cellular homeomorphism
\begin{equation*}
 \tau_g \colon \sing{\alpha,d}{X,Y} \rightthreetimes_{\cellsPlus Y} \cC_Y
 \xrightarrow{\cong} \sing{\alpha,d}{X,Y} \rightthreetimes_{\cellsPlus Y} (\cC_Y \circ \mu_g). 
\end{equation*}
Observing that $\sing{\alpha,d}{X,Y} \circ \mu_g = \sing{\alpha,d}{X,Y}$, we
obtain from the functoriality of $\rightthreetimes_{\cellsPlus Y}$ in \Cref{sec:balanced-products} a cellular homeomorphism
\begin{equation*}
 \iota_{\mu_g} \colon \sing{\alpha,d}{X,Y} \rightthreetimes_{\cellsPlus Y}
 (\cC_Y \circ \mu_g) = (\sing{\alpha,d}{X,Y} \circ \mu_g)
 \rightthreetimes_{\cellsPlus Y} (\cC_Y \circ \mu_g) \xrightarrow{\cong}
 \sing{\alpha,d}{X,Y} \rightthreetimes_{\cellsPlus Y} \cC_Y.
\end{equation*}
Define the action map of $g \in G$ as the composition
\begin{equation*}
 g \cdot - := \iota_{\mu_g} \circ \tau_g \colon \trans_X(Y) \xrightarrow{\cong} \trans_X(Y).
\end{equation*}
Explicitly, this map is given by $g \cdot [c,x,y] \mapsto [gc,x,gy]$, and defines a group action by cellular homeomorphisms.

Again by \Cref{prop:balanced-product-cells}, we have a canonical identification
\begin{equation*}
 \cells \trans_X(Y) \cong \{ (\sigma, c) \mid c \in \cells Y, \sigma \in S^{\delta_{\abs{c}}}_\bullet(M\Gamma^{l_{\abs{c}},\alpha}) \},
\end{equation*}
which translates the $G$--action on the set of cells of $\trans_X(Y)$ to $g \cdot (\sigma, c) = (\sigma, gc)$. Hence, $\trans_X(Y)$ is a free $G$--CW--complex.

Continuing to use the above identification of $\cells \trans_X(Y)$, we define a control map for $\trans_X(Y)$: Let $\beta_p$ denote the barycenter of the standard $p$--simplex. Then set
\begin{equation*}
 \trans_X(\kappa) \colon \cells \trans_X(Y) \to X \times G \times E \times [1,\infty[, \quad (\sigma,c) \mapsto ((R \circ \sigma)(\beta_{\abs{\sigma}}),\kappa(c)).
\end{equation*}

\begin{lem}\label{lem:transfer-control-objects}
 The pair $(\trans_X(Y),\trans_X(\kappa))$ is an object in $\cR^G(W,\JJ(X \times G,E))$ which is $(\alpha + \delta_0 + 2)$--controlled over $X \times G$.
\end{lem}
\begin{proof}
 By construction, the labeled $G$--CW--complex $(\trans_X(Y),
 \trans_X(\kappa))$ satisfies the $G$--continuous control condition. It also
 has the correct support, since $X$ is compact.
 So it is only necessary to check that it satisfies bounded control over $X \times G$. Let $(\sigma,c)$ and $(\sigma',c')$ be cells such that $(\sigma,c) \subset \gen{(\sigma',c')}$. By \Cref{prop:balanced-product-cells}, this is equivalent to the conditions $c \subset \gen{c'}$ and  $\sigma \subset \gen{\kappa_G(c)^{-1}\kappa_G(c')\sigma'} \subset \norm{ S^{\delta_{\abs{c}}}_\bullet(M\Gamma^{l_{\abs{c}},\alpha}) }$. Set $\gamma_{c,c'} := \kappa_G(c)^{-1}\kappa_G(c')$. Then we have
 \begin{equation*}
  \begin{split}
   d_{S^n,n,\Lambda}&((\trans_X(\kappa)(\sigma,c), \trans_X(\kappa)(\sigma',c')) \\
   &= d_{S^n,n,\Lambda}\big( ((R \circ \sigma)(\beta_{\abs{\sigma}}),\kappa_G(c)), ((R \circ \sigma')(\beta_{\abs{\sigma'}}),\kappa_G(c')) \big) \\
   &\leq d_{S^n,n,\Lambda}\big( ((R \circ \sigma)(\beta_{\abs{\sigma}}),\kappa_G(c)), (R \circ \gamma_{c,c'}\sigma')(\beta_{\abs{\sigma'}}), \kappa_G(c)) \big) \\
   &\quad + d_{S^n,n,\Lambda}\big( (R \circ \gamma_{c,c'}\sigma')(\beta_{\abs{\sigma'}}), \kappa_G(c)), ((R \circ \sigma')(\beta_{\abs{\sigma'}}),\kappa_G(c')) \big) \\
   &\leq \delta_{\abs{c}} + 2 \leq \delta_0 + 2,
  \end{split}
 \end{equation*}
 where the last inequality follows from our assumption and
 \Cref{lem:distance-to-translate}. Using $d_G(\kappa_G(c), \kappa_G(c')) \leq
 \alpha$, we conclude that $(\trans_X(Y),\trans_X(\kappa))$ is $(\alpha+\delta_0+2)$--controlled over $X \times G$.
\end{proof}

\subsection{The transfer on cellwise $0$-controlled morphisms}
\label{subsec:transfer-cellwise-0}
In the next step, we extend the assignment $(Y,\kappa) \mapsto (\trans_X(Y), \trans_X(\kappa))$ to a functor $\cR^G_f(W,\JJ(E))_{\alpha,d} \to \cR^G(W,\JJ(X \times G, E))$.
Let $f \colon (Y_1,\kappa_1) \to (Y_2,\kappa_2)$ be a morphism in
$\cR^G_f(W,\JJ(E))_{\alpha,d}$. Since $f$ is a regular map, we have an induced
functor $\cellsPlus f \colon \cellsPlus Y_1 \to \cellsPlus Y_2$, which is compatible
with the $G$--actions.  ($G$ acts trivially on the initial
object.)
Define a natural transformation $\cC_f \colon \cC_{Y_1} \to \cC_{Y_2} \circ
\cellsPlus f$ by
\begin{equation*}
 \cC_{f,c} \colon \cC_{Y_1}(c) = \gen{c} \xrightarrow{f} f(\gen{c}) =
 \gen{f(c)} = (\cC_{Y_2} \circ \cellsPlus f)(c).
\end{equation*}
Define a natural transformation $\sing{}{f} \colon \sing{\alpha,d}{X,Y_1} \to
\sing{\alpha,d}{X,Y_2} \circ \cellsPlus f$ by
\begin{equation*}
 \sing{}{f,c} \colon \norm{ S^{\delta_{\abs{c}}}(M\Gamma^{l_{\abs{c}},\alpha}) } \subset \norm{ S^{\delta_{\abs{f(c)}}}_\bullet(M\Gamma^{l_{\abs{f(c)}},\alpha}) }.
\end{equation*}
Then $\trans_X(f)$ is defined as the composition
\begin{equation*}
 \sing{\alpha,d}{X,Y_1} \rightthreetimes_{\cellsPlus Y_1} \cC_{Y_1}
 \xrightarrow{\sing{}{f} \rightthreetimes_{\cellsPlus Y_1} \cC_f}
 (\sing{\alpha,d}{X,Y_2} \circ (\cellsPlus f)) \rightthreetimes_{\cellsPlus
 Y_1} \cC_{Y_2} \circ (\cellsPlus f)) \xrightarrow{\iota_{(\cellsPlus f)}}
 \sing{\alpha,d}{X,Y_2} \rightthreetimes_{\cellsPlus Y_2} \cC_{Y_2}.
\end{equation*}

\begin{lem}\label{lem:transfer-functor}
 This defines a functor
 \begin{equation*}
  \trans_X \colon \cR^G_f(W,\JJ(E))_{\alpha,d} \to \cR^G(W,\JJ(X \times G,E))_0.
 \end{equation*}
\end{lem}
\begin{proof}
 To see that the construction of $\trans_X(f)$ is functorial, it is best to translate the above formalism again into an explicit mapping rule. Concretely, $\trans_X(f)$ is given by $[c,x,y] \mapsto [f(c),x,f(y)]$, and functoriality becomes obvious.
 
 We also need to check that $\trans_X(f)$ is a controlled map. It suffices to
 consider bounded control over $X \times G$. Note that $\trans_X(f)$ is
 regular as $f$ and $\sing{}{f,c}$ are regular.  Hence it is enough to compute for $(\sigma, c) \in \cells \trans_X(Y_1)$
 \begin{equation*}
  \begin{split}
   d&_{S^n,n,\Lambda}(\trans_X(\kappa_1)(\sigma,c), \trans_X(\kappa_2)(\trans_X(f)(\sigma,c))) + d_G(\kappa_{1,G}(c),\kappa_{2,G}(f(c))) \\
   &= d_{S^n,n,\Lambda}\big( ((R \circ
   \sigma)(\beta_{\abs{\sigma}}),\kappa_{1,G}(c)), ((R \circ
   \sigma)(\beta_{\abs{\sigma}}),\kappa_{2,G}(f(c))) \big) + 0\\
   &= 0.
  \end{split}
 \end{equation*}
 So $\trans_X(f)$ is in fact cellwise $0$--controlled over $X \times G$. Hence, we have defined a functor $\trans_X \colon \cR^G_f(W,\JJ(E))_{\alpha,d} \to \cR^G(W,\JJ(X \times G,E))_0$.
\end{proof}

\begin{rem}
 One can adapt the constructions presented in this paper to chain complexes over geometric modules to obtain a linear transfer. The linearization map, which assigns to a CW--complex its cellular chain complex, translates our transfer functor into its linear counterpart.
 
 Moreover, the natural inclusion of geometric modules into chain complexes makes these constructions compatible with the transfers defined in \cite{Bartels-Lueck(2012annals)}, \cite{Bartels-Lueck-Reich(2008hyper)} and \cite{Wegner(2012)}. Thus, the transfer for geometric $\ZZ[G]$--modules corresponds to our construction restricted to $0$--dimensional CW--complexes.
\end{rem}

\subsection{Transferring cofibrations} 
Our next aim is to show that the transfer is a functor of categories with cofibrations. Since the cofibrations under consideration are essentially inclusions of CW--subcomplexes it comes as no surprise that this result relies on an analysis of the CW--structure of $\trans_X(Y)$.

Let $(Y, \kappa) \in \cR^G_f(W,\JJ(E))$ as before.  It defines a $\cellsPlus
Y$--CW--complex $\cC_Y$ relative $W$.  Let $B\subseteq Y$ be a subcomplex.  We
get a $\cellsPlus Y$--CW--complex $\cC_Y^B$ relative $W$ by setting
$\cC_Y^B(c) := \gen{c} \cap B$.  As before, $\colim_{\cellsPlus Y} \cC_Y^B
\cong B$.

Let $A \subseteq B \subseteq Y$ be subcomplexes.  Assume that $B$
arises from $A$ by attaching cells $\eta_i, i \in I$.  From
\Cref{sec:balanced-products} we get a pushout diagram
  \begin{equation*}
  \begin{tikzpicture}
   \matrix (m) [matrix of math nodes, column sep=3em, row sep=2em, text depth=.5em, text height=1em]
   { {\coprod^W_{i\in I}\partial \eta_i}  & \cC^A_Y \\
     {\coprod^W_{i\in I}\eta_i} & \cC^B_Y \\};
   \path[->]
   (m-1-1) edge (m-1-2)
   (m-1-1) edge (m-2-1)
   (m-1-2) edge (m-2-2)
   (m-2-1) edge (m-2-2);
  \end{tikzpicture}
 \end{equation*}
in $\cellsPlus Y$--CW--complexes relative $W$.  This becomes a pushout diagram
in retractive spaces, if we equip everything with the retractions into $W$
arising from $Y$.  Now, $\sing{\alpha,d}{X,Y} \rightthreetimes_{\cellsPlus
Y} (-) $ commutes with pushouts, so we get the following result.

\begin{lem}\label{lem:attaching-cells-transfer}
There is a pushout diagram
\begin{equation*}
  \begin{tikzpicture}
   \matrix (m) [matrix of math nodes, column sep=3em, row sep=2em, text depth=.5em, text height=1em]
   {\sing{\alpha,d}{X,Y} \rightthreetimes_{\cellsPlus Y} \big(\coprod_{i \in
   I}^W
   \partial \eta_i \big) & 
       \sing{\alpha,d}{X,Y} \rightthreetimes_{\cellsPlus Y} \cC^A_Y \\
    \sing{\alpha,d}{X,Y} \rightthreetimes_{\cellsPlus Y} \big(\coprod_{i \in
    I}^W
    \eta_i \big) & 
       \sing{\alpha,d}{X,Y} \rightthreetimes_{\cellsPlus Y} \cC^B_Y \\};
   \path[->]
   (m-1-1) edge (m-1-2)
   (m-1-1) edge (m-2-1)
   (m-1-2) edge (m-2-2)
   (m-2-1) edge (m-2-2);
  \end{tikzpicture}
 \end{equation*}
in $\cR^G(W,\JJ(X \times G,E))$.   Here the coproducts on the left are
disjoint unions over $W$.  As these are cells, the space on the lower
left is isomorphic
to
\begin{equation*}
  \coprod_{i \in I} \big( \norm{
    S^{\delta^d_{\abs{c_i}}}_\bullet(M\Gamma^{l^d_{\abs{c_i}},\alpha}) }
    \times D^{\abs{c_i}} \big) \amalg W ,
\end{equation*}
when $\eta_i$ is a cell
based at $c_i$, and similarly for the upper left. \qed
\end{lem}

This enables us to do inductive arguments over the cells in $Y$.   Note that
if $A \subseteq Y$ is a subcomplex, we can interpret $A$ as a $\cellsPlus
A$--CW--complex, or as $\cellsPlus Y$--CW--complex.  We can define the
transfer also for $A$ as a $\cellsPlus Y$--CW--complex, and it is canonically
isomorphic to \Cref{dfn:transfer}.

\begin{lem}\label{lem:transfer-cofibrations}
 The functor $\trans_X$ preserves the zero object, cofibrations and admissible pushout diagrams, i.e., it is a functor of categories with cofibrations.
\end{lem}
\begin{proof}
 Let $f \colon (Y_1,\kappa_1) \rightarrowtail (Y_2,\kappa_2)$ be a cofibration
 in $\cR^G_f(W,\JJ(E))_{\alpha,d}$. Without loss of generality, we can assume
 that $Y_2$ is obtained from $Y_1$ by attaching free $G$--cells, and that $f$
 is the inclusion of the subcomplex $Y_1$ into $Y_2$. Then it follows from
 \Cref{lem:attaching-cells-transfer} and interpreting $Y_1$ as a $\cellsPlus
 Y_2$--CW--complex that $\trans_X(f)$ is also a cofibration. For the same
 reason, $\trans_X$ preserves all relevant pushout squares.  Last, it maps the
 zero object $W$ to $W$.
\end{proof}

\subsection{The transfer on general morphisms and weak equivalences}
Next, we construct natural transformations between our transfer functors for various indices. Once we have shown that they are weak equivalences, it follows that the diagram in \Cref{eq:hocomm} is homotopy commutative. In addition, these enter the proof that $\trans^{\alpha,d}$ preserves weak equivalences.

\begin{dfn}\label{def:transfer-vary-constraints}
  Let $\alpha' > \alpha$ and $d' > d$ satisfying $n > \max \{d'+1, \alpha'
  \}$. Then both $\sing{\alpha,d}{X,Y}$ and $\sing{\alpha',d'}{X,Y}$ are
  defined and give rise to transfer functors $\trans^{\alpha,d}_X$ and
  $\trans^{\alpha',d'}_X$. For every $(Y,\kappa) \in
  \cR^G_f(W,\JJ(E))_{\alpha,d}$, there is a natural transformation
  $\sing{\alpha,d}{X,Y} \to \sing{\alpha',d'}{X,Y}$ which is given at $c \in
  \cellsPlus Y$ by the obvious inclusion 
 \begin{equation*}
  \norm{ S^{\delta^{d}_{\abs{c}}}_\bullet(M\Gamma^{l^{d}_{\abs{c}},\alpha}) } \subset \norm{ S^{\delta^{d'}_{\abs{c}}}_\bullet(M\Gamma^{l^{d'}_{\abs{c}},\alpha'}) }. 
 \end{equation*}
 Hence, we obtain an induced natural morphism
 \begin{equation*}
  \rho^{\alpha,\alpha',d,d'}_Y \colon \trans^{\alpha,d}_X(Y) \to \trans^{\alpha',d'}_X(Y).
 \end{equation*}
\end{dfn}

\begin{lem}\label{cor:MXdom}
 Let $\delta > 0$. Consider $X$ as a subspace of $M\Gamma^{l,s}$ via the embedding $x \mapsto [e,x]$. There exists a $\delta$--controlled strong deformation retraction
 \begin{equation*}
  H \colon \norm{S_\bullet^\delta(M\Gamma^{l,s})} \times [0,1] \to \norm{S_\bullet^\delta(M\Gamma^{l,s})}    
 \end{equation*}
 onto $\norm{S^\delta_\bullet(X)}$. 
\end{lem}
\begin{proof}
 There is a (topological) inclusion
 \begin{equation*}
  i \colon \norm{S^\delta_\bullet(M\Gamma^{l,s})} \times [0,1] \to \norm{S^\delta_\bullet(M\Gamma^{l,s} \times [0,1])}    
 \end{equation*}
 which maps each prism $\Delta^p \times [0,1]$ to its canonical triangulation.
 The strong deformation retraction from
 \Cref{eq:strong-deformation-rectraction-MGamma} restricts to a
 strong deformation retraction
 \begin{equation*}
  H' \colon M\Gamma^{l,s} \times [0,1] \to M\Gamma^{l,s}
 \end{equation*}
 of $M\Gamma^{l,s}$ onto $X$, it is given by
 \begin{equation*}
  H'([e,t_k,\gamma_k,\dots,\gamma_1,x],u) := [e,u \cdot t_k,\gamma_k, \dots,
  \gamma_1, x].
 \end{equation*}
 It has the property that $R \circ H'(m,u) = R(m)$, so
 \begin{equation*}
     \diam_{X \times G} \{ (R \circ H'(m,u),g) \mid u \in [0,1] \} = 0
 \end{equation*}
 for all $m \in M\Gamma^{l,s}$ and $g \in G$. Hence, $H'$ induces a map
 \begin{equation*}
  H'_* \colon \norm{S^\delta_\bullet(M\Gamma^{l,s} \times [0,1])} \to \norm{S^\delta_\bullet(M\Gamma^{l,s})}.
 \end{equation*}
 Then $H := H'_* \circ i$ is a strong deformation retraction onto $\norm{S^\delta_\bullet(X)}$. As the target is
 $\delta$-controlled, $H$ is $\delta$-controlled. 
\end{proof}

\begin{cor}\label{cor:smallvaryconstraints}
 Let $\delta > 0$, $l \leq l'$ and $s \leq s'$.
 Then the canonical inclusion map $\norm{S_\bullet^\delta(M\Gamma^{l,s})} \hookrightarrow \norm{S_\bullet^{\delta}(M\Gamma^{l',s'})}$ is a $2\delta$--controlled homotopy equivalence (with respect to $d_{S^n,n,\Lambda}$).
\end{cor}
\begin{proof}
 The corollary above yields homotopy equivalences $\norm{S_\bullet^\delta(X)}\hookrightarrow \norm{S_\bullet^\delta(M\Gamma^{l,s})}$ and $\norm{S_\bullet^\delta(X)}\hookrightarrow \norm{S_\bullet^{\delta}(M\Gamma^{l',s'})}$ which are $\delta$--controlled. Since the triangle 
 \[
  \begin{tikzpicture}
   \matrix (m) [matrix of math nodes, column sep=1.5em, row sep=1.5em, text depth=.5em, text height=1em, ampersand replacement=\&]
   {\norm{\cS_\bullet^\delta(X)} \&\norm{\cS_\bullet^\delta(M\Gamma^{l,s})}   \\ \& \norm{\cS_\bullet^{\delta}(M\Gamma^{l',s'})}  \\};
   \path[->]
   (m-1-1) edge node[above]{} (m-1-2)
   (m-1-1) edge node[below left]{} (m-2-2)
   (m-1-2) edge node[below right]{} (m-2-2);
  \end{tikzpicture}
 \]
 commutes, the result follows.
\end{proof}

\begin{prop}\label{prop:transfer-vary-constraints}
 The morphisms $\rho^{\alpha,\alpha',d,d'}_Y$ are weak equivalences.
\end{prop}
\begin{proof}
 We prove that $\rho^{\alpha,\alpha',d,d'}_Y$ is a weak equivalence for all $Y
 \in \cR^G_f(W,\JJ(E))_{\alpha,d}$ by induction over the dimension of $Y$. For
 $(-1)$--dimensional objects, which is the start of the induction, the claim
 is trivial. For the induction step, we apply
 \Cref{lem:attaching-cells-transfer} to see that the inclusion of the
 $p$--skeleton into the $(p+1)$--skeleton induces a pushout
\begin{equation*}
  \begin{tikzpicture}
   \matrix (m) [matrix of math nodes, column sep=3em, row sep=2em, text
   depth=.5em, text height=1em]
  {\sing{\alpha,d}{X,Y} \rightthreetimes_{\cellsPlus Y} 
        \big(\coprod_{c \in \cells_{p+1}Y}^W \partial c \big) & 
     \sing{\alpha,d}{X,Y} \rightthreetimes_{\cellsPlus Y} \skel{p}{\cC_Y} \\
   \sing{\alpha,d}{X,Y} \rightthreetimes_{\cellsPlus Y} 
       \big(\coprod_{c \in \cells_{p+1}Y}^W c \big) & 
     \sing{\alpha,d}{X,Y} \rightthreetimes_{\cellsPlus Y} \skel{p+1}{\cC_Y}
   \\};
   \path[->]
   (m-1-1) edge (m-1-2)
   (m-1-1) edge (m-2-1)
   (m-1-2) edge (m-2-2)
   (m-2-1) edge (m-2-2);
  \end{tikzpicture}
 \end{equation*}
 in $\cR^G_f(W,\JJ(X \times G, E))_0$.  The lower left corner is, again by
 \Cref{lem:attaching-cells-transfer}, identified as
 \begin{equation*}
  \coprod_{c \in \cells_{p+1}Y} \big( \norm{
    S^{\delta^d_{\abs{c}}}_\bullet(M\Gamma^{l^d_{\abs{c}},\alpha}) }
    \times D^{\abs{c}} \big) \amalg W,
 \end{equation*}
 and similarly for the the upper corner, with $D^{\abs{c}}$ replaced with
 $\partial D^{\abs{c}}$.
 There is an analogous pushout square
 with $\alpha$ and $d$ replaced by $\alpha'$ and $d'$, respectively. Moreover,
 the former square maps to the latter via the transformation
 $\rho^{\alpha,\alpha',d,d'}_Y$.  On the left-hand sides this is identified
 with the canonical inclusion maps. This transformation is a weak equivalence
 on the top right corner of the diagram by induction hypothesis, and it is a
 $2\delta_{p+1}$--controlled homotopy equivalence on the top left and bottom
 left corners combining \Cref{cor:smallvaryconstraints} and
 \Cref{lem:controlled-excision} below. Hence, the gluing
 lemma implies that it is also a weak equivalence on the bottom right corner.
 This finishes the induction step and finite dimensionality of $Y$ proves the
 claim.
\end{proof}

In order to show exactness, we will need that the transfer maps
$h$--equivalences to $h^{fin}$--equivalences later.  The following lemma
implies this.

\begin{prop}\label{prop:transfer-exact} 
 Let $f$ be a weak equivalence in $\cR^G_f(W,\JJ(E))_{\alpha,d}$ which is an $\alpha'$--controlled homotopy equivalence over $G$. Suppose that $n > \max \{d+2, \alpha, \alpha' \}$.
 
 Then $\trans_X(f)$ is a controlled homotopy equivalence.
\end{prop}
\begin{proof}
 To show the proposition, we exploit the fact that maps which are not cellwise
 $0$--controlled over $G$ can also be transferred, but in a less functorial
 fashion. Let $(Y_1,\kappa_1)$ and $(Y_2,\kappa_2)$ be objects of
 $\cR^G_f(W,\JJ(E))_{\alpha,d}$, and $f \colon Y_1 \to Y_2$ be an arbitrary
 map in $\cR^G_f(W,\JJ(E))$. Choose $\alpha' > 0$ such that $f$ is
 $\alpha'$--controlled over $G$. We construct an induced map
 \begin{equation*}
  \trans_{\alpha,\alpha'}(f) \colon \trans^{\alpha,d}_X(Y_1) \to \trans^{\max \{ \alpha,\alpha' \}, d+1}_X(Y_2).    
 \end{equation*}
 To define $\trans_{\alpha,\alpha'}(f)$, consider for the beginning a single
 cell $c \in \cellsPlus Y_1$, denote by $\eta_c$ the corresponding
 $\cellsPlus Y_1$--$n$--cell of $\cC_{Y_2}$. We define the
 function 
 \begin{equation*}
  \begin{split}
   t_c \colon 
   \norm{ S^{\delta_{\abs{c}}}_\bullet(M\Gamma^{l_{\abs{c}},\alpha}) } 
     \rightthreetimes_{\cellsPlus Y_1} \eta_{c} & \to 
   \sing{\max \{ \alpha,\alpha' \},d+1}{X,Y_2} 
     \rightthreetimes_{\cellsPlus Y_2} \cC_{Y_2} \\
   (x,y) & \mapsto [\supp(f(y)), \gamma^c_y \cdot x, f(y)],
  \end{split}
 \end{equation*}
 where $\supp(f(y))$ denotes the support of $f(y)$, i.e., the unique open cell
 $\supp(f(y))$ of $Y_2$ such that $f(y) \in \supp(f(y))$, and $\gamma^c_y :=
 \kappa_{2,G}(\supp(f(y)))^{-1} \kappa_{1,G}(c)$.  We will glue the different
 $t_c$ together to get the transfer for $f$.
 
 Let us check that the target space is large enough such that $t_c(x,y)$ is
 contained in it: Recall that
 $\gamma^c_y \cdot x$ is defined via the $G$--action which
 $\norm{S_\bullet(M\Gamma)}$ inherits from $M\Gamma$. Since $f$ is
 $\alpha'$--controlled, we have $\gamma^c_y \in B_{\alpha'}(e)$. Therefore, we
 can regard multiplication with $\gamma^c_y$ as a map
 $M\Gamma^{l_{\abs{c}},\alpha} \to M\Gamma^{l_{\abs{c}}+1,\max \{
 \alpha,\alpha' \}}$. In addition, $M\Gamma^{l_{\abs{c}}+1,\max \{ \alpha,
 \alpha' \}}$ is contained in $M\Gamma^{l_{\abs{\supp(f(y))}} +1,  \max \{
 \alpha,\alpha' \}}$, so $[\supp(f(y)),\gamma^c_y \cdot x, f(y)]$ defines a
 point in the target space.
 
 We need to check that $t_c$ is continuous. It suffices to show continuity on
 finite subcomplexes. These are metrizable, so it is enough to show that $t_c$
 is sequentially continuous. Let $(x_l,y_l)_l$ be a convergent sequence in
 $\norm{ S^{\delta_{\abs{c}}}_\bullet(M\Gamma^{l_{\abs{c}},\alpha}) } \times
 \eta_{c}$ with limit point $(x,y)$.  As $f$ is continuous, $f(y_l)$ converges
 against $f(y)$.  Hence $S := \{ \supp(f(y_l)) \mid l \in \NN \}$ is a finite
 set, and we can assume that for each $s \in S$ there are infinitely many $l$
 such that $\supp(f(y_l)) = s$.  We treat the $s$ individually and restrict to
 the corresponding subsequence.  If $s$ happens to be equal to $\supp(f(y))$,
 $\gamma^c_{y}= \gamma^c_{y_l}$ and continuity follows. Otherwise, $f(y)$ must
 still lie in the closure of the cell $s$, i.e.,
 $\supp(f(y)) \subset \gen{s}$. Hence,
 \begin{equation*}
  \begin{split}
   [\supp(f(y)), &\gamma^c_y \cdot x, f(y)] \\
   &= [\supp(f(y)),\kappa_{2,G}(\supp(f(y)))^{-1} \kappa_{2,G}(s) \kappa_{2,G}(s)^{-1} \kappa_{1,G}(c) x, f(y)] \\
   &= [s, \kappa_{2,G}(s)^{-1} \kappa_{1,G}(c) x, f(y)],
  \end{split}
 \end{equation*}
 and continuity becomes obvious.
 
 Suppose now that $c \leq c'$ in $\cells Y_1$. For $y \in \gen{c}$ and $x \in \norm{ S^{\delta_{\abs{c'}}}_\bullet(M\Gamma^{l_{\abs{c'}},\alpha}) }$  we obtain
 \begin{equation*}
  \begin{split}
   t_{c'}(x,y)
   &= [ \supp(f(y)), \gamma^{c'}_y x, f(y)] \\
   &= [ \supp(f(y)), \kappa_{2,G}(\supp(f(y)))^{-1} \kappa_{1,G}(c) \kappa_{1,G}(c)^{-1} \kappa_{1,G}(c') x, f(y)] \\
   &= t_c(\kappa_{1,G}(c)^{-1} \kappa_{1,G}(c') x,y).
  \end{split}
 \end{equation*}
 Therefore, the collection $\{ t_c \}_{c \in \cells Y_1}$ induces a continuous, cellular map relative $W$
 \begin{equation*}
  \trans_{\alpha,\alpha'}(f) \colon \trans^{\alpha,d}_X(Y_1) \to \trans^{\max\{ \alpha,\alpha' \},d+1}_X(Y_2).
 \end{equation*}
 Using \Cref{lem:distance-to-translate} and
 \Cref{lem:transfer-control-objects}, it is not hard to show that
 $\trans_{\alpha,\alpha'}(f)$ is $(\max \{\alpha,\alpha' \} + \alpha' +
 \delta_0 +
 4)$--controlled over $X \times G$, as $n > \alpha'$.
 Note that
 for cellwise $0$--controlled
 maps, $\trans_{\alpha,\alpha'}(f)$ agrees with the previous defined transfer
 from \Cref{subsec:transfer-cellwise-0}.  The only reason we increased $d$ is
 the argument which follows, it was not needed in the construction so far.
 
 Suppose now that $f \colon (Y_1,\kappa_1) \to (Y_2,\kappa_2)$ is a weak
 equivalence in $\cR^G_f(W,\JJ(E))_{\alpha,d}$ which is $\alpha'$--controlled
 over $G$ as a homotopy equivalence, i.e., its inverses and the homotopies
 are $\alpha'$--controlled over $G$. Then there exists some
 $\alpha'$--controlled
 map $\overline{f} \colon (Y_2,\kappa_2) \to (Y_1,\kappa_1)$ such that
 $\overline{f}f$ and $f\overline{f}$ are $\alpha'$--controlled homotopic to
 the identity. Consider the diagram
 \begin{equation*}
  \begin{tikzpicture}
   \matrix (m) [matrix of math nodes, column sep=10em, row sep=4em, text depth=.5em, text height=1em]
   {\trans^{\alpha,d}_X(Y_1) & \trans^{\alpha,d}_X(Y_2) \\
    \trans^{\max \{\alpha,\alpha'\},d+1}_X(Y_1) & \trans^{\max \{\alpha,\alpha'\},d+1}_X(Y_2) \\};
   \path[->]
    (m-1-1) edge node[above]{$\trans^{\alpha,d}_X(f)$} (m-1-2)
    (m-2-1) edge node[below]{$\trans^{\max \{ \alpha,\alpha' \}, d+1}_X(f)$} (m-2-2);
   \path[>->]
    (m-1-1) edge node[left]{$\rho^{\alpha,\max \{ \alpha,\alpha' \},d,d+1}_{Y_1}$} (m-2-1)
    (m-1-2) edge node[right]{$\rho^{\alpha,\max \{ \alpha,\alpha' \},d,d+1}_{Y_2}$} (m-2-2);
   \path[dashed,->]
    (m-1-2) edge node[above left]{$\trans_{\alpha,\alpha'}(\overline{f})$} (m-2-1);
  \end{tikzpicture}
 \end{equation*}
 in which the outer square commutes. The vertical maps $\rho^{\alpha,\max \{\alpha,\alpha'\},d,d+1}_{Y_i}$, $i=1,2$, are weak equivalences by \Cref{prop:transfer-vary-constraints}. We claim that the two triangles involving the dashed diagonal map $\trans_{\alpha,\alpha'}(\overline{f})$ commute up to controlled homotopy. If this is true, it follows that $\trans^{\alpha,d}_X(f)$ is a weak equivalence.
 
 Let $h \colon Y_1 \leftthreetimes [0,1] \to Y_1$ be an $\alpha'$--controlled homotopy from $\overline{f}f$ to $\id_{Y_1}$. Note that $\trans_{\alpha,\alpha'}(\overline{f}) \circ \trans^{\alpha,d}_X(f) = \trans_{\alpha,\alpha'}(\overline{f}f)$. Since $n > d+2$, we can apply $\trans_X^{\alpha,d}$ also to $Y_1 \leftthreetimes [0,1]$ and consider the controlled map
 \begin{equation*}
  \trans_{\alpha,\alpha'}(h) \colon \trans^{\alpha,d}_X(Y_1 \leftthreetimes [0,1]) \to \trans^{\max \{\alpha,\alpha' \}, d+1}_X(Y_1).
 \end{equation*}
 The domain of this map is not equal to $\trans^{\alpha,d}_X(Y_1)
 \leftthreetimes [0,1]$, but it is contained in $\trans^{\alpha,d}_X(Y_1)
 \leftthreetimes [0,1]$ as a controlled strong deformation retract. This
 follows by an induction argument similar to
 \Cref{prop:transfer-vary-constraints}.  Essentially, we can construct both
 objects as the balanced products over $\cellsPlus (Y_1 \leftthreetimes
 [0,1])$ and use that the inclusion 
 $\norm{ S^{\delta_{\abs{c}}}_\bullet(M\Gamma^{l_{\abs{c}},\alpha}) } \to
 \norm{ S^{\delta_{\abs{c}}+1}_\bullet(M\Gamma^{l_{\abs{c}+1},\alpha}) } $ is
 a controlled deformation retraction by \Cref{cor:smallvaryconstraints} and
 \Cref{lem:controlled-excision}.  The retraction induces the required
 controlled homotopy.
 
 The argument for the second triangle is analogous.
\end{proof}

\subsection{Restricting the target category}
\label{subsect:transfer-restricting-target}
  Now we show that the transfer functor factors over the full subcategory of
  finitely dominated objects.  The following result was already used in the
  previous chapter.

\begin{lem}\label{lem:controlled-excision}
 Let $(M,d)$ be a metric space.
 \begin{enumerate}
  \item Let $\delta > 0$. The natural inclusion map $\norm{
    S^\delta_\bullet(M) } \to \norm{ S_\bullet(M) }$ is a homotopy
    equivalence.
  \item Let $0 < \delta \leq \delta'$. Then the inclusion map $\norm{
    S^\delta_\bullet(M) } \to \norm{ S^{\delta'}_\bullet(M) }$ is a
    $\delta'$--controlled homotopy equivalence (with respect to the metric on
    $M$, labelling simplices by the image of their barycenter).
  \item Suppose $\abs{K}$ is the realization of an ordered (abstract)
    simplicial complex $K$ and suppose that $p \colon \abs{K} \to M$ is
    a continuous map. Let $\kappa \colon \cells K \to M$ be the labelling
    sending a cell (=simplex) to the image of its barycenter under $p$. Let
    $\delta > 0$. Let $S^\delta_\bullet(\abs{K},p)$ denote the
    (semi)simplicial set of all singular simplices $\sigma$ in $\abs{K}$ such
    that the diameter of $p
    \circ \sigma$ is at most $\delta$.
  
  If the characteristic maps of all simplices of $K$ lie in $S^\delta_\bullet(\abs{K},p)$, then the canonical map $\abs{K} \to \norm{ S^\delta_\bullet(\abs{K},p) }$ is a $\delta$--controlled homotopy equivalence (measuring control in $M$ via $p$).
 \end{enumerate}
\end{lem}
\begin{proof}
 The proof proceeds in analogy to \cite[Lemma~6.7]{Bartels-Lueck-Reich(2008hyper)}. The first part follows directly from an appropriate formulation of excision, e.g.~\cite[Theorem~4.6.9]{Fritsch-Piccinini(1990)}.
  
 For the second part, let $\cA$ be the poset of closed subsets of $X$, considered as a category. Then the $\cA$--CW--complex $\sing{\delta}{\cA}$ given by
 \begin{equation*}
  \sing{\delta}{\cA}(A) := \norm{ S^\delta_\bullet(A) }
 \end{equation*}
 is a free $\cA$--CW--complex, whose free cells are of the form
 $\hom_\cA(\sigma(\Delta^{\abs{\sigma}}),-)\times D^{\abs{\sigma}}$ since
 $\cells \norm{ S^\delta_\bullet(A) } \cong \coprod_{\sigma \in
 S^\delta_\bullet(X) } \hom_\cA(\sigma(\Delta^{\abs{\sigma}}),A)$. Since both
 $\norm{ S^\delta_\bullet(A) } \hookrightarrow \norm{ S_\bullet(A) }$ and
 $\norm{ S^{\delta'}_\bullet(A) } \hookrightarrow \norm{ S_\bullet(A) }$ are
 homotopy equivalences for every $A \in \cA$, so is the inclusion $\norm{
 S^\delta_\bullet(A) } \hookrightarrow \norm{ S^{\delta'}_\bullet(A) }$.
 Hence, the natural transformation $\sing{\delta}{\cA} \to
 \sing{\delta'}{\cA}$ is a homotopy equivalence of (free) $\cA$--CW--complexes
 by \cite[Corollary~3.5]{Davis-Lueck(1998)}.  This in particular means that
 there is an inverse map, compatible with the structure map, as well as
 compatible
 homotopies.  It is easy to check that such a map has the right control.
 
 The third claim follows by similar reasoning, substituting the poset of subcomplexes for the poset of closed subsets.
\end{proof}

\begin{lem}\label{lem:finite-domination-singular-complex}
  Suppose that $(X,d_X)$ admits a finite $\epsilon$--domination.  Then 
  $\norm{ S^\delta_\bullet(M\Gamma^{l,\alpha}) }$ is $4 \delta + 6
  \Lambda \epsilon$-dominated over $X \times G$ (with respect to the metric
  $d_{S^n, n , \Lambda}$).
\end{lem}

\begin{proof}
 By \Cref{cor:MXdom}, the complex $\norm{ S^{\delta}_\bullet(X) }$ is a $\delta$--controlled strong deformation retract of $\norm{ S^{\delta}_\bullet(M\Gamma^{l,\alpha})}$. Choose an appropriate controlled retraction $r$.
 
 Pick an $\epsilon$--domination of $X$ by a finite simplicial complex
 $\abs{K}$, i.e., a sequence of maps $X \xrightarrow{\iota} \abs{K}
 \xrightarrow{\pi} X$ together with a homotopy $h \colon \pi \circ \iota
 \simeq \id_X$, and such that the diameter, measured with respect to the
 original metric $d_X$ on $X$, of $h(x,[0,1])$ is at most $\epsilon$ for every
 $x \in X$. Then the given domination induces maps
 \begin{equation*}
  \norm{ S^\delta_\bullet(X) } \xrightarrow{\iota_*} \norm{
    S^{\delta+2\Lambda\epsilon}_\bullet(\abs{K},\pi) } \xrightarrow{\pi_*}
    \norm{ S^{\delta+2\Lambda\epsilon}_\bullet(X) },
 \end{equation*}
 where similarly to \Cref{subsec:y-spaces} we measure distances of points in
 $\abs{K}$ via $\pi$.
 These are maps of labeled complexes over $X \times G$: Pick an arbitrary group element $g \in G$. Then we label simplices $\sigma$ in $S^\delta_\bullet(X)$ or $S^{\delta+2\Lambda\epsilon}_\bullet(X)$ by $(\sigma(\beta_{\abs{\sigma}}),g)$ and simplices $\sigma$ in $S^{\delta+2\Lambda\epsilon}_\bullet(\abs{K},\pi)$ by $(\pi(\sigma(\beta_{\abs{\sigma}})),g)$ (cf.~\Cref{cor:MXdom} and \Cref{cor:smallvaryconstraints}).
 
 Choose an iterated barycentric subdivision $K'$ of $K$ such that the characteristic map of each simplex of $K'$ is a simplex in $S^{\delta+2\Lambda\epsilon}_\bullet(\abs{K},\pi)$; note that $K'$ is ordered if we subdivide at least once. Since $\abs{K'}$ is then naturally a subcomplex of $S^{\delta+2\Lambda\epsilon}_\bullet(\abs{K},\pi)$, we endow it with the induced control map.
 The canonical inclusion $i \colon\abs{K'} \to \norm{ S^{\delta+2\Lambda\epsilon}_\bullet(\abs{K},\pi) }$ is a $(\delta+2\Lambda\epsilon)$--controlled homotopy equivalence over $X \times G$ by \Cref{lem:controlled-excision}. Choose an appropriate controlled homotopy inverse $p$.
 
 Finally, the inclusion $\norm{ S^\delta_\bullet(X) } \hookrightarrow \norm{ S^{\delta+2\Lambda\epsilon}_\bullet(X) }$ is a $(\delta+2\Lambda\epsilon)$--controlled homotopy equivalence by \Cref{lem:controlled-excision}; let $f$ be an appropriate controlled homotopy inverse.
 
 Then
 \begin{align*}
  \norm{ S^\delta_\bullet(M\Gamma^{l,\alpha}) } \xrightarrow{r} \norm{ S^\delta_\bullet(X) } \xrightarrow{\iota_*} \norm{ S^{\delta+2\Lambda\epsilon}_\bullet(\abs{K},\pi) } \xrightarrow{p} \abs{K'}
 \end{align*}
 and
 \begin{align*}
  \abs{K'} \xrightarrow{i} \norm{ S^{\delta+2\Lambda\epsilon}_\bullet(\abs{K},\pi) } \xrightarrow{\pi_*} \norm{ S^{\delta+2\Lambda\epsilon}_\bullet(X) } \xrightarrow{f}  \norm{ S^\delta_\bullet(X) } \hookrightarrow \norm{ S^\delta_\bullet(M\Gamma^{l,\alpha}) }
 \end{align*}
 yield the desired domination of $\norm{ S^\delta_\bullet(X) }$; from the previous control estimates we see that there is a $(4\delta+6\Lambda\epsilon)$--controlled homotopy between the composition of these two maps and the identity on $\norm{ S^\delta_\bullet(M\Gamma^{l,\alpha}) }$.
\end{proof}

\begin{prop}\label{prop:transfer-finite-domination}
  Suppose that $(X,d_X)$ admits a finite $\epsilon$--domination for every
  $\epsilon$.  Let $(Y,\kappa) \in \cR^G_f(W,\JJ(E))_{\alpha,d}$.  

  Then $\trans_X(Y,\kappa)$ is controlled finitely dominated, i.e., it defines
  an object in $\cR^G_{fd}(W,\JJ(X \times G, E))$.  We can choose the control
  estimate to be independent of the constants $\Lambda$ and $n$ from the metric.
\end{prop} 
\begin{proof}
 We prove the claim by induction on the dimension of $Y$. The case of a $(-1)$--dimensional object is trivial and provides the start of the induction.
 
 For the induction step, we use \Cref{lem:attaching-cells-transfer} to obtain
 a pushout square
\begin{equation*}
  \begin{tikzpicture}
   \matrix (m) [matrix of math nodes, column sep=3em, row sep=2em, text
   depth=.5em, text height=1em]
  {
  \coprod_{c \in \cells_{p+1}Y} \big( \norm{
    S^{\delta^d_{p+1}}_\bullet(M\Gamma^{l^d_{p+1},\alpha}) }
    \times \partial D^{\abs{c}} \big) \amalg W &
     \sing{\alpha,d}{X,Y} \rightthreetimes_{\cellsPlus Y} \skel{p}{\cC_Y} \\
  \coprod_{c \in \cells_{p+1}Y} \big( \norm{
    S^{\delta^d_{p+1}}_\bullet(M\Gamma^{l^d_{p+1},\alpha}) }
    \times D^{\abs{c}} \big) \amalg W &
     \sing{\alpha,d}{X,Y} \rightthreetimes_{\cellsPlus Y} \skel{p+1}{\cC_Y}
   \\};
   \path[->]
   (m-1-1) edge (m-1-2)
   (m-1-1) edge (m-2-1)
   (m-1-2) edge (m-2-2)
   (m-2-1) edge (m-2-2);
  \end{tikzpicture}
 \end{equation*}
 in $\cR^G(W,\JJ(X \times G,E))$.
%
 By induction hypothesis, the object at the top right corner of this square is
 finitely dominated. So we only need to find a controlled finite domination
 for $\norm{ S^{\delta_{p+1}}_\bullet(M\Gamma^{l_{p+1},\alpha}) }$, as $Y$
 itself is (locally) finite.  By
 \Cref{lem:finite-domination-singular-complex}, such a domination indeed
 exists.
 
 Note that the same bound works if we increase $n$, and we can choose
 $\epsilon$ to be $\frac{1}{\Lambda}$.  Then the estimate of the metric does
 not depend on $n$ and $\Lambda$, which finishes the proof.
\end{proof}

 
 Finally, we show that, after forgetting the labelling in $X$, the transfer does not alter the homotopy type of a given object.
 
\begin{prop}\label{prop:transfer-section}
  Let $P \colon \cR^G_{fd}(W,\JJ(X \times G, E)) \to \cR^G_{fd}(W,\JJ(E))$ denote the functor induced by the projection map $X \times G \to G$. Let $(Y,\kappa) \in \cR^G_f(W,\JJ(E))_{\alpha,d}$.
 
 Then there is an $\alpha$--controlled natural weak equivalence
 \begin{equation*}
  P(\trans_X(Y)) \xrightarrow{\sim} Y.
 \end{equation*}
\end{prop}
\begin{proof}
 The relevant map is induced by the projection map $M\Gamma \to *$. As in the
 proofs of \Cref{prop:transfer-exact} and
 \Cref{prop:transfer-finite-domination}, the claim follows by another
 induction along the skeleta of $Y$, using \Cref{lem:attaching-cells-transfer} and \Cref{lem:controlled-excision} together with the fact that the projection map $\norm{ S_\bullet(X) } \to *$ is a homotopy equivalence. Since the bounded control is only over $(G,d_G)$, it is not hard to check that the weak equivalence is $\alpha$--controlled.
\end{proof}


\subsection{The transfer map}
\label{subsec:transfer-map}

  We combine all of the results established so far to show
  \Cref{prop:main-diag-claims}~\ref{item:main-diag-claims-1}. 

Let $N \in \NN$. Suppose that we have chosen for every $n \in \NN$
\begin{enumerate}
 \item a compact, contractible metric space $(X_n,d_{X_n})$ such that for every $\epsilon > 0$ there is an $\epsilon$--controlled domination of $X_n$ by an at most $N$--dimensional, finite simplicial complex;
 \item a homotopy coherent $G$--action $\Gamma_n$ on $X_n$;
 \item a positive real number $\Lambda_n$.
\end{enumerate}
We equip $X_n \times G$ with the metric $d_{S^n, n, \Lambda_n} + d_G$. As in
\Cref{subsec:transfers}, we set
\begin{equation*}
    \delta_k := 4(d+1-k), \quad l_k := d+1-k.
\end{equation*}

\begin{prop}\label{prop:full-transfer}
 Let $\alpha, d \in \NN$. The assignment
 \begin{equation*}
  \begin{split}
   (Y,\kappa) &\mapsto \trans^{\alpha,d}(Y,\kappa) := \big( \trans^{\alpha,d}_{X_n}(Y,\kappa) \big)_{n > \max \{ d+1, \alpha \} } \\
   f &\mapsto \big( \trans^{\alpha,d}_{X_n}(f) \big)_{n > \max \{d+1, \alpha \} }
  \end{split}
 \end{equation*}
 defines an exact functor
 \begin{equation*}
   \trans^{\alpha,d} \colon (\cR^G_f(W,\JJ(E)),h)_{\alpha, d} \to (\cR^G_{fd}(W,\JJ((X_n
  \times G)_n,E)), h^{fin}).
\end{equation*}
\end{prop}
\begin{proof}
 According to \Cref{lem:transfer-control-objects}, \Cref{lem:transfer-functor} and \Cref{lem:transfer-cofibrations} the assignment yields a functor of categories with cofibrations
 \[
   \trans^{\alpha,d} \colon \cR^G_f(W,\JJ(E))_{\alpha,d} \to \prod_{n\in \NN}\cR^G_{fd}(W,\JJ(X_n \times G,E)).
 \]
 
 We have to show that it factors over the subcategory $\cR^G_{fd}(W,\JJ((X_n \times G)_n,E))$. This is the case if all objects and morphisms in the image of $\trans^{\alpha,d}$ are uniformly boundedly controlled over $X_n \times G$. Essentially, we have to see that all of the necessary control estimates are independent of $n\in \NN$.

 For this, recall that the map
 \begin{equation*}
   \cR^G_{fd}(W,\JJ((X_n \times G)_n,E)) \to \prod_{n\in
   \NN}\cR^G_{fd}(W,\JJ(X_n \times G,E)) 
 \end{equation*}
 works as follows.  Essentially, an object in the source is a CW--complex
 relative $W$, where we have a partition of its cells into $\NN$--many sets and
 no boundary and no map is allowed to hit a cell which is in a different
 set.  Hence, we can write the object as the coproduct (over $W$), indexed
 by $\NN$, of
 CW--complexes relative $W$.  The collection of summands defines an element in
 the target.  If the transfer satisfies a uniform metric control
 condition, it factors over this map.  Hence we need to check that the
 previous results of this section give uniform bounds for all $n$.
 
 
 
 Since $\trans^{\alpha,d}_{X_n}(Y,\kappa)$ is
 $(\alpha+\delta_0+2)$--controlled over $X_n \times G$ for every $n$ by
 \Cref{lem:transfer-control-objects} and $\trans_{X_n}^{\alpha, d}(f)$ is
 cellwise $0$-controlled by \Cref{lem:transfer-functor}, all
 objects and morphisms are uniformly bounded, as desired.
 \Cref{prop:transfer-finite-domination} shows that each component
 $\trans^{\alpha,d}(Y,\kappa)$ is finitely dominated, but we need it
 uniformly.  For this, note that the proof of
 \Cref{prop:transfer-finite-domination} can actually be done with $(X_n)_n$
 replacing $X$.  Roughly, we would get an extra coproduct over $\NN$
 everywhere, and everything else would need to get an extra index, which is
 why \Cref{prop:transfer-finite-domination} is not stated that way.  However,
 the control estimations come from applications of the gluing lemma and an
 induction over the cells of $Y$.  But the gluing lemma preserves
 the property of everything being uniformly controlled, and we start the
 induction with uniform control arising from $f$ and the $\delta_k$, so we can
 do the same induction.

 \Cref{prop:transfer-exact} tells us that $\trans^{\alpha,d}_{X_n}$ sends
 $h$--equivalences to $h^{fin}$--equivalences since it applies for
 sufficiently large $n$.  Again, the proof can be done for $(X_n)_n$ instead
 of $X$, and the control estimates come from an induction over the cells of
 $Y$ and the gluing lemma, so they will be uniform.
\end{proof}

\begin{prop}\label{prop:full-transfer-compatibility}
 Let $\alpha, d \in \NN$ and $i_{\alpha,d} \colon \cR^G_f(W,\JJ(E))_{\alpha,d} \hookrightarrow \cR^G_f(W,\JJ(E))_{\alpha+1,d+1}$ be the obvious inclusion functor. Then there is a natural $h^{fin}$--equivalence
 \begin{equation*}
  \trans^{\alpha,d} \xrightarrow{\sim} \trans^{\alpha+1,d+1} \circ i_{\alpha,d}.
 \end{equation*}
\end{prop}
\begin{proof}
 There is a natural transformation $\trans^{\alpha,d} \to \trans^{\alpha+1,d+1} \circ i_{\alpha,d}$ given by the sequence $(\rho^{\alpha,\alpha+1,d,d+1})_{n > \max \{ d+2,\alpha+1 \} }$ from \Cref{def:transfer-vary-constraints}. These are homotopy equivalences by \Cref{prop:transfer-vary-constraints}, and the control estimates in the proof of \Cref{prop:transfer-vary-constraints} show that they are also uniformly boundedly controlled homotopy equivalences.
\end{proof}

To obtain the transfer map whose existence was claimed in \Cref{prop:main-diag-claims}, we proceed as follows. Let $k \in \NN$. Consider the inclusion $j_k \colon \cR^G_f(W,\JJ(E))_{k,k} \hookrightarrow \cR^G_f(W,\JJ(E))_{k+1,k+1}$. By \Cref{prop:full-transfer-compatibility}, there is a natural weak equivalence
\begin{equation*}
 \rho_k \colon \trans^{k,k} \xrightarrow{\sim} \trans^{k+1,k+1} \circ j_k.
\end{equation*}
Hence, we obtain an induced homotopy
\begin{equation*}
 K(\trans^{k,k}) \simeq K(\trans^{k+1,k+1}) \circ K(j_k). 
\end{equation*}
Thinking of $\hocolim_k K(\cR^G_f(W,\JJ(E))_{k,k},h)$ as the mapping telescope
of
\begin{equation*}
 K(\cR^G_f(W,\JJ(E))_{1,1},h) \xrightarrow{K(j_1)} K(\cR^G_f(W,\JJ(E))_{2,2},h) \xrightarrow{K(j_2)} \dots,
\end{equation*}
these homotopies serve to define a map
\begin{equation*}
 \trans \colon \hocolim_k K(\cR^G_f(W,\JJ(E))_{k,k},h) \to K(\cR^G_{fd}(W,\JJ((X_n \times G)_n,E)),h^{fin}).
\end{equation*}

\begin{prop}
 The map $\trans$ satisfies \Cref{prop:main-diag-claims}~\ref{item:main-diag-claims-1}.
\end{prop}
\begin{proof}
  That $\trans$ is the required map and that the diagram commutes up to
  homotopy is immediate from \Cref{prop:domain-of-transfer} and
 \Cref{prop:transfer-section}, noting again that the latter proof can be
 done uniformly.
\end{proof}


\typeout{--------------------------------   References -----------------------------------------------}

\end{document}